\documentclass[10pt]{amsart}
\usepackage{amsxtra, amsfonts, amsmath, amsthm, amstext, amssymb, amscd, mathrsfs, verbatim, color}
\usepackage{threeparttable}
\usepackage{mathtools}
\usepackage[ansinew]{inputenc}\usepackage[T1]{fontenc}
\usepackage[all,cmtip]{xy}
\theoremstyle{plain}

\newtheorem{theorem}{Theorem}[section]
\newtheorem{lemma}[theorem]{Lemma}
\newtheorem{definition-theorem}[theorem]{Definition-Theorem}
\newtheorem{proposition}[theorem]{Proposition}
\newtheorem{corollary}[theorem]{Corollary}

\newtheorem{conjecture}[theorem]{Conjecture}

\theoremstyle{definition}
\newtheorem{definition}[theorem]{Definition}
\newtheorem{example}[theorem]{Example}
\newtheorem{remark}[theorem]{Remark}
\newtheorem{notation}[theorem]{Notation}
\newcommand \bth[1] { \begin{theorem}\label{t#1} }
\newcommand \ble[1] { \begin{lemma}\label{l#1} }

\newcommand \bpr[1] { \begin{proposition}\label{p#1} }
\newcommand \bco[1] { \begin{corollary}\label{c#1} }
\newcommand \bde[1] { \begin{definition}\label{d#1}\rm }
\newcommand \bex[1] { \begin{example}\label{e#1}\rm }
\newcommand \bre[1] { \begin{remark}\label{r#1}\rm }

\newcommand \bnota[1] {\begin{notation}\label{n#1}\rm }
\newcommand {\ele} { \end{lemma} }

\newcommand {\epr} { \end{proposition} }
\newcommand {\eco} { \end{corollary} }
\newcommand {\ede} { \end{definition} }
\newcommand {\eex} { \end{example} }
\newcommand {\ere} { \end{remark} }
\newcommand {\enota} { \end{notation} }

\begin{document}
\title[Duality for GGP relevant pairs]{Duality for generalized Gan-Gross-Prasad relevant pairs for $p$-adic $\mathrm{GL}_n$} 

\author[Kei Yuen Chan]{Kei Yuen Chan}
\address{Department of Mathematics, the University of Hong Kong, Hong Kong, China }
\email{kychan1@hku.hk}
\maketitle

\begin{abstract}

The main goal of this article is to formulate a notion, called a generalized GGP relevant pair, governing the quotient branching law for $p$-adic general linear groups. Such notion relies on a commutation relation between derivatives (from Jacquet functors) and integrals (from parabolic inductions), for which we provide both representation-theoretic and combinatorial perspectives. Our main result proves a duality on those relevant pairs, which is compatible with a dual restriction in branching law.



\end{abstract}

\section{Introduction}

Branching law is a classical problem in the representation theory. Even formulating a precise branching law is a subtle problem, which has already drawn a lot of attention in recent years \cite{GGP12, GGP20}. This article aims to setup essential ingredients towards establishing the general quotient branching law in the sequel \cite{Ch22+c}. 

In Section \ref{ss commutation probl}, we first define derivatives and integrals, and explain a commutation problem between derivatives and integrals.  Such commutation suggests our terminology of strongly commutative triples in Section \ref{ss strong commut triple}. In Section \ref{ss combin commut triples}, we give an equivalent combinatorial definition for those triples. In Section \ref{ss commutate from branching law}, we explain how those commutative triples arise in branching laws.  The main duality result is in Section \ref{ss duality triples}.

For the interests of readers, some results on derivatives and integrals will be discussed in wider generality e.g. inner forms of general linear groups over a non-Archimedean local field and $\square$-irreducible representations.

\subsection{Derivatives and integrals} \label{ss commutation probl}

Let $F$ be a non-Archimedean local field. Let $D$ be a finite-dimensional central division $F$-algebra. Let $G_n=\mathrm{GL}_n(D)$ be the general linear group over $D$. Let $\mathrm{Alg}(G_n)$ be the category of complex smooth representations of $G_n$. Let $\mathrm{Irr}(G_n)$ be the set of irreducible smooth representations of $G_n$. Let $\mathrm{Irr}=\sqcup_n\mathrm{Irr}(G_n)$.

We first define some operators on $\mathrm{Irr}(G_n)$. Denote, by $\times$, the normalized parabolic induction. Following Lapid-M\'inguez \cite{LM19}, an irreducible representation $\sigma$ of $G_n$ is said to be {\it $\square$-irreducible} if $\sigma \times \sigma$ is still irreducible. There are some recent interests in the problem of characterizing $\square$-irreducible representations, see \cite{LM19}. In particular, they classify when a regular representation is $\square$-irreducible. Let $\mathrm{Irr}^{\square}(G_n)$ be the set of $\square$-irreducible representations of $G_n$. Let $\mathrm{Irr}^{\square}=\sqcup_n \mathrm{Irr}^{\square}(G_n)$. 

Let $I^L_{\sigma}(\pi)$   (resp. $I^R_{\sigma}(\pi)$)  be the unique simple submodule of $\sigma \times \pi$ (resp. $ \pi \times \sigma$), called the {\it left (resp. right) $\sigma$-integral \footnote{It is called a left multiplier in \cite{LM16} while we prefer to use integrals to emphasis the relation to derivatives.} of $\pi$}. The uniqueness part is shown in \cite{LM19} adapting the proof of Kang-Kashiwara-Kim-Oh \cite{KKKO15}. One may view $I^L_{\sigma}$ and $I^R_{\sigma}$ as operators from $\mathrm{Irr}$ to $\mathrm{Irr}$. 





For $n_1+\ldots+n_r=n$, let $P_{n_1, \ldots, n_r}$ be the standard parabolic subgroup containing block diagonal matrices $\mathrm{diag}(g_1, \ldots , g_r)$ with $g_k \in G_k$ and upper triangular matrices. Let $N_{n_1, \ldots, n_r}$ be the unipotent radical of $P_{n_1, \ldots, n_r}$.  We shall sometimes abbreviate $N_{n-i,i}$ by $N_i$. For a unipotent radical $N$ and $\pi \in \mathrm{Alg}(G_k)$, we write $\pi_N$ to be the corresponding normalized Jacquet module for $\pi$.

 For $\sigma \in \mathrm{Irr}^{\square}(G_i)$ and $\pi' \in \mathrm{Irr}(G_n)$, define $D^R_{\sigma}(\pi)$ (resp. $D^L_{\sigma}(\pi)$) to be the unique irreducible representation (if it exists) such that
\begin{align} \label{eqn embedding derivative}
  D^R_{\sigma}(\pi) \boxtimes \sigma \hookrightarrow  \pi_{N}, \quad (\mbox{resp. }\sigma \boxtimes D^L_{\sigma}(\pi) \hookrightarrow \pi_{N'})
\end{align}
where $N=N_{n-i,i}$ (resp. $N'=N_{i,n-i}$). We set $D^R_{\sigma}(\pi)=0$ (resp. $D^L_{\sigma}(\pi)=0$) if it does not exist. This is well-defined by \cite[Theorem 4.1D]{LM22}. By \cite[Corollary 2.4]{LM19} and the second adjointness of parabolic induction, we have $D^R_{\sigma}\circ I^R_{\sigma}(\pi) \cong \pi$ and similarly $I^R_{\sigma}\circ D^R_{\sigma}(\pi)\cong \pi$ if $D^R_{\sigma}(\pi)\neq 0$.

When one operator is applied on the left and another operator is applied on the right, it is natural to ask whether those operators commute. For example, the commutation for $I^L_{\sigma}\circ I^R_{\sigma'}(\pi)\cong I^R_{\sigma'}\circ I^L_{\sigma}(\pi)$ is related to the associativity a binary operation on nilpotent orbits studied by Aizenbud-Lapid \cite{AL22} and Lapid-M\'inguez \cite{LM22}.  We shall use geometric lemma to impose a natural condition to guarantee  $D^R_{\sigma'}\circ I^L_{\sigma}(\tau) \cong I^L_{\sigma}\circ D^R_{\sigma'}(\tau)$ in Definition \ref{def strong comm} below



\subsection{Strongly commutative triples} \label{ss strong commut triple}

It is shown later in Proposition \ref{prop strong commute imply commute} that the following conditions in Definition \ref{def strong comm} guarantees a commutation of a left integral and a right derivative.

\begin{definition} \label{def strong comm}
Let $\pi \in \mathrm{Irr}$. Let $\sigma' \in \mathrm{Irr}^{\square}$ and let $\sigma \in \mathrm{Irr}^{\square}(G_r)$.
\begin{itemize}
\item  A triple $(\sigma, \sigma', \pi)$ is said to be {\it  pre-RdLi-commutative} if $D^R_{\sigma}(\pi)\neq 0$ and the composition of the maps:
\begin{align} \label{eqn embedding for strong commut}
 \quad D^R_{\sigma}\circ I^L_{\sigma'}(\pi)\boxtimes \sigma \hookrightarrow I^L_{\sigma'}(\pi)_{N_r} \hookrightarrow (\sigma'\times \pi)_{N_r} \stackrel{s}{\rightarrow} \sigma' \dot{\times}^1 (\pi_{N_r}) 
\end{align}
is non-zero, where the first injection is the unique embedding from (\ref{eqn embedding derivative}), and the second map is induced from the embedding $I^L_{\sigma'}(\pi)\hookrightarrow \sigma' \times \pi$. Here $Rd$ refers to right derivatives and $Li$ refers to left integrals, and $s: (\sigma' \times \pi)_{N_r} \rightarrow \sigma \dot{\times}^1 (\pi_{N_r})$ be the natural surjection onto the top layer in the geometric lemma (see Section \ref{ss geo lem} below for the notion $\dot{\times}^1$).
\item We say that a triple $(\sigma, \sigma', \pi)$ is {\it  strongly RdLi-commutative} if $(\sigma, \sigma', \pi)$ is  pre-RdLi-commutative and the map (\ref{eqn embedding for strong commut}) above factors through the map 
\[ (\sigma' \times D^R_{\sigma}(\pi))\boxtimes \sigma \hookrightarrow  \sigma' \dot{\times}^1 \pi_{N_r} \]
induced from (\ref{eqn embedding derivative}).
\end{itemize}
\end{definition}

Typical examples of pre-RdLi-commutative triples arise when the intersection of the cuspidal supports of $\sigma$ and $\sigma'$ is non-empty. Another family of examples comes from when the integral preserves the level of $\pi$, see Corollary \ref{cor level preserving integrals comm}.

The version switching the right derivative to a left derivative and switching the left integral to a right integral can be formulated analogously, and we shall call those triples to be {\it pre-LdRi-commutative triples} and {\it strongly LdRi-commutative triples} respectively. For example, in the LdRi-version, (\ref{eqn embedding for strong commut}) becomes of the following form:
\begin{align}
\quad \sigma \boxtimes D^L_{\sigma}\circ I^R_{\sigma'}(\pi) \hookrightarrow I^R_{\sigma'}(\pi)_{N_{r,n(\pi)}} \hookrightarrow (\pi \times \sigma')_{N_{r,n(\pi)}} \stackrel{s}{\rightarrow} \pi_{N_{n(\sigma)-r,r}} \ddot{\times } \sigma ,
\end{align}
where $\pi_{N_{n(\sigma)-r,r}} \ddot{\times } \sigma$ is defined analogously as the $\dot{\times}$ by switching left and right versions, and $\pi_{N_{n(\sigma)-r,r}} \ddot{\times}\sigma$ is a $G_{n(\pi)-r}\times G_{r+n(\sigma)}$-representation. We shall not use the notion $\ddot{\times}$ anymore explicitly.

Indeed, from Section \ref{s prelim}, one can see that the two notions of strong commutativity can be related by the Gelfand-Kazhdan involution, and are a key to the formulation of the main duality result (Theorem \ref{thm dual strong commutation} below). Rather than using the Gelfand-Kazhdan involution, one may also define by using the smooth dual functor, while such formulation has a slightly subtle issue from dualizing the geometric lemma. We shall explain such issue in more details in \cite{Ch24+}.

We make the following conjecture:
\begin{conjecture} \label{conj pre implies strong}
Let $\sigma, \sigma' \in \mathrm{Irr}^{\square}$ and let $\pi \in \mathrm{Irr}$. If a triple $(\sigma, \sigma', \pi)$ is pre-RdLi-commutative, then $(\sigma, \sigma', \pi)$ is strongly RdLi-commutative.
\end{conjecture}

It is well-known from the work of Zelevinsky \cite{Ze80} and Tadi\'c \cite{Ta90} that essentially square-integrable representations are $\square$-irreducible.

\begin{theorem} (=Theorem \ref{thm pre imply strong}) \label{thm pre implies strong ds}
Conjecture \ref{conj pre implies strong} holds if both $\sigma$ and $\sigma'$ are essentially square-integrable.
\end{theorem}



\subsection{Combinatorially commutative triples} \label{ss combin commut triples}

For applications on branching laws, we are interested in $\sigma$ and $\sigma'$ to be essentially square-integrable representations. We shall give a combinatorial criteria for those strongly commutative triples, and we need some combinatorial objects. 


Let $|.|$ be the absolute value of $F$ and let $\mathrm{Nrd}: G_n \rightarrow F^{\times}$ be the reduced norm. To each cuspidal representation $\rho$ of $G_n$, we associate a unique character $\nu_{\rho}(g)=|\mathrm{Nrd}(g)|^{s_{\rho}}$ of $G_n$ with $s_{\rho}>0$ so that $\rho \times \nu_{\rho}^{\pm 1}\rho$ is reducible. For $a,b \in \mathbb Z$ with $b-a \geq 0$ and a cuspidal representation $\rho$, we call $[a,b]_{\rho}$ to be a {\it segment}, following \cite{Ze80}. We consider two segments $[a,b]_{\rho}$ and $[a',b']_{\rho'}$ to be equal if $\nu_{\rho}^a\rho\cong \nu_{\rho}^{a'}\rho'$ and $\nu_{\rho}^b\rho \cong \nu_{\rho}^{b'}\rho'$. For each segment $\Delta$, let $\mathrm{St}(\Delta)$ be the corresponding essentially square-integrable representations i.e. the unique simple quotient of $\nu_{\rho}^a\rho\times \ldots \times \nu_{\rho}^b\rho$. 

For a segment $\Delta$, we set 
\[D^R_{\Delta}(\pi)=D^R_{\mathrm{St}(\Delta)}(\pi),\ D^L_{\Delta}(\pi)=D^L_{\mathrm{St}(\Delta)}(\pi),\ I^R_{\Delta}(\pi)=I^R_{\mathrm{St}(\Delta)}(\pi),\ I^L_{\Delta}(\pi)=I^L_{\mathrm{St}(\Delta)}(\pi) .\]
We define more combinatorial invariants:
\begin{itemize}
\item  Define $\varepsilon_{\Delta}(\pi):=\varepsilon_{\Delta}^R(\pi)$ (resp. $\varepsilon_{\Delta}^L(\pi)$) to be the largest non-negative integer $k$ such that $D^k_{\Delta}(\pi)\neq 0$ (resp. $(D^L_{\Delta})^k(\pi)\neq 0$). Here the power $k$ means the composition of $D_{\Delta}$ for $k$-times and when $k=0$, $D^0_{\Delta}$ is regarded as the identity operator.
\item  Define 
\[ \eta_{\Delta}(\pi):=\eta^R_{\Delta}(\pi):=(\varepsilon_{[a,b]_{\rho}}(\pi), \varepsilon_{[a+1,b]_{\rho}}(\pi), \ldots, \varepsilon_{[b,b]_{\rho}}(\pi));\]
\[  (resp. \quad \eta_{\Delta}^L(\pi):=(\varepsilon^L_{[a,b]_{\rho}}(\pi), \varepsilon^L_{[a,b-1]_{\rho}}(\pi), \ldots, \varepsilon^L_{[a,a]_{\rho}}(\pi)) ).
\] 
\end{itemize}
When $\Delta=\left\{ \rho \right\}$, it is a more standard case in \cite{Ja07, Mi09}. 

\begin{definition} \label{def combinatorial comm triple}
Let $\Delta, \Delta'$ be segments. Let $\pi \in \mathrm{Irr}$. We say that $(\Delta, \Delta', \pi)$ is a {\it combinatorially RdLi-commutative triple} if $\varepsilon_{\Delta}(\pi)\neq 0$ (equivalently $D^R_{\Delta}(\pi)\neq 0$) and
\[ \eta_{\Delta}(I^L_{\Delta'}(\pi)) = \eta_{\Delta}(\pi) . 
\]
\end{definition}

We have the following criteria for strong commutation, which also uses Theorem \ref{thm pre implies strong ds} in the proof:

\begin{theorem} (=Part of Theorem \ref{thm combinatorial def})  \label{thm combin and strong}
 A triple $(\Delta, \Delta', \pi)$ is combinatorially RdLi-commutative if and only if $(\mathrm{St}(\Delta), \mathrm{St}(\Delta'), \pi)$ is strongly RdLi-commutative.
\end{theorem}

There is a dual formulation in Definition \ref{def dual com commutative}.

\subsection{Where does the commutation appear in branching laws?} \label{ss commutate from branching law}

Those integrals are a generalization of Kashiwara's crystal operators, as suggested by Leclerc \cite{Le03}, which is related to $p$-adic groups via the quantum Schur-Weyl duality of Chari-Pressley \cite{CP96} and the Bernstein theory. It is also known that there are tight links of the branching laws of type A algebras to crystal graphs, and so one may regard our work in \cite{Ch22+c} is to search for such link to $p$-adic groups. A notable connection is the Lascoux-Leclerc-Thibon algorithm \cite{LLT96} in finding the decomposition numbers for Hecke algebras at the root of unity, proved by Ariki \cite{Ar96}. There are many subsequent developments after that e.g. about canonical bases \cite{VV99, Sc00}, about categorifications \cite{CR08, KL09} as well as block structures \cite{Fa10}.

We now explain more precisely how the commutation comes into the play of branching laws. Let $D=F$. Let $\pi_1 \in \mathrm{Irr}(G_{n+1})$ and let $\pi_2 \in \mathrm{Irr}(G_n)$. We refer the reader to Section \ref{s branching laws} for a notion of derivatives $D^R_{\mathfrak n}$ and integral $I^L_{\mathfrak m}$ for multisegments $\mathfrak m$ and $\mathfrak n$, which extend the segment cases $D^R_{\Delta}$ and $I^L_{\Delta}$ respectively. We shall not introduce what the Bernstein-Zelevinsky derivatives are (see e.g. \cite{Ch21, CS21} for details), but studies in \cite{Ch22+, Ch22+b, Ch22+e} suggest that $D^R_{\mathfrak n}(\pi)$ and $D^L_{\mathfrak n}(\pi)$ are weak replacements of right and left Bernstein-Zelevinsky derivatives. 

Now suppose $\mathrm{Hom}_{G_n}(\pi_1, \pi_2) \neq 0$. Then the Bernstein-Zelevinsky theory suggests that $I^L_{\mathfrak m}\circ D^R_{\mathfrak n}(\nu^{1/2}\pi) \cong \pi'$ for some multisegments $\mathfrak m$ and $\mathfrak n$. By the left Bernstein-Zelevinsky theory \cite{CS21, Ch21}, we also have $I_{\mathfrak m'}^R \circ D^L_{\mathfrak n'}(\nu^{-1/2}\pi)\cong \pi'$. This imposes a condition on those $\mathfrak m, \mathfrak m', \mathfrak n, \mathfrak n'$:
\[    I^L_{\mathfrak m}\circ D^R_{\mathfrak n}(\nu^{1/2}\pi) \cong  I_{\mathfrak m'}^R \circ D^L_{\mathfrak n'}(\nu^{-1/2}\pi)  .
\] 
This gives that
\[   D^L_{\mathfrak n'}(\nu^{-1/2}\pi) \boxtimes \mathrm{St}(\mathfrak n') \hookrightarrow  I^L_{\mathfrak m}\circ D^R_{\mathfrak n}(\nu^{1/2}\pi)_{N_l} \hookrightarrow (\mathrm{St}(\mathfrak m) \times (D^R_{\mathfrak n}(\nu^{1/2}\pi)) )_{N_l} ,
\]
where $l=l_{abs}(\mathfrak m)$. In a special case that the Zelevinsky multisegments parametrizing $\pi$ and $\pi'$ has all segments of relative length at least $2$, combinatorics forces that $(\mathrm{St}(\mathfrak n'), \mathrm{St}(\mathfrak m), D^R_{\mathfrak n}(\nu^{1/2}\pi))$ is a pre-RdLi-commutative triple. This is a starting point of this article.

The goal of this article is to study in details for the segment case. Extending the segment case would be better to be dealt with the concept of minimal multisegments \cite{Ch22+, Ch22+a} and so will be postponed to \cite{Ch22+c}. 

\subsection{Duality} \label{ss duality triples}


One of our main results in this article is the following duality:

\begin{theorem} (=Theorem \ref{thm pre imply strong}+Corollary \ref{cor dual strong commutative triple}) \label{thm dual strong commutation}
Let $\sigma, \sigma'$ be essentially square-integrable representations in $\mathrm{Irr}$. A triple $(\sigma, \sigma', \pi)$ is strongly RdLi-commutative (pre-RdLi-commutative) if and only if $(\sigma', \sigma, I^L_{\sigma'}\circ D^R_{\sigma}(\pi))$ is strongly LdRi-commutative (resp. pre-LdRi-commutative).
\end{theorem}

Theorem \ref{thm dual strong commutation} will be important in the study of quotient branching laws, for which we deduce a duality  for the relevant pairs (see Corollary \ref{cor dual on relevant pair}). This corresponds to the following dual restriction in branching law:
\begin{proposition} (see \cite[Proposition 4.1]{Ch22})
Let $\pi \in \mathrm{Irr}(\mathrm{GL}_{n+1}(F))$ and let $\pi' \in \mathrm{Irr}(\mathrm{GL}_n(F))$. Then there exists a cuspidal representation $\sigma$ of $\mathrm{GL}_2(F)$ such that $\tau \times \pi'{}^{\vee}$ is irreducible and 
\[  \mathrm{Hom}_{\mathrm{GL}_n(F)}(\pi, \pi') \cong \mathrm{Hom}_{\mathrm{GL}_{n+1}(F)}(\tau \times \pi'{}^{\vee}, \pi^{\vee}) .\]
\end{proposition}
One can utilize these two dualities with the Bernstein-Zelevinsky machinery to give some inductive proofs in \cite{Ch22+c}.

\subsection{Key ideas in the proof of Theorem \ref{thm pre implies strong ds}. }

The proof relies on a special structure from the Jacquet functor, which we refer to an irreducible pair. Roughly speaking, the structure is a direct summand in a Jacquet module so that the followings work:
\begin{itemize}
\item one can deduce the strong commutativity from the pre-commutativity involving an irreducible pair (Proposition \ref{thm pre implies strong}(2)); 
\item the irreducibility pair property is preserved under the integral in a strongly commutative triple (Proposition \ref{prop strong irr imply irr}).
\end{itemize}

In order to apply the first bullet, we need to construct some irreducible pairs, that is Proposition \ref{prop completing pre comm}. The proof is basically an analysis on the terms in the geometric lemma, but a key to rule out a possibility of a structure is a structure of a big derivative studied in \cite{Ch22+b} (see Lemma \ref{lem big derivative compared with jacquet}).

These two properties are used to prove preservation of $\eta$-invariants under integrals and derivatives involved in some strongly commutative triples. Now, for a pre-RdLi-commutative triple $(\sigma_1, \sigma_2, \pi)$, one analyses the constituents on $\sigma_1 \dot{\times}^1 \pi_{N_{n(\sigma_1)}}$ and use the invariants to force the required embedding in the strong embedding has to happen.

\subsection{Key ideas in the proof of Theorem \ref{thm dual strong commutation}}

Theorem \ref{thm pre implies strong ds} reduces proving Theorem \ref{thm dual strong commutation} to proving the corresponding statement for pre-commutativity. The main idea of the proof is first constructing a certain irreducible pair to obtain some new pre-commutativity (Proposition \ref{prop completing pre comm}), then apply the dual theory for such pre-commutativity (Proposition \ref{prop commutative triple}) and finally use a transitivity property (Proposition \ref{prop commutative triple induct}) to obtain the desired pre-commutativity.

\subsection{Key ideas in the proof of Theorem \ref{thm combin and strong}}

For the if direction of Theorem \ref{thm combin and strong}, it is proved along the proof of Theorem \ref{thm combinatorial def} as well. The main technical ingredient is an analysis of the terms in the geometric lemma in Proposition \ref{prop completing pre comm}.

For the only if direction of Theorem \ref{thm combin and strong}, the combinatorial criteria can be used to produce a strongly RdLi-commutative triple and so one applies Proposition \ref{prop trivial orbit in irreducible case}  and the geometry in the geometric lemma to obtain the pre-commutativity. Now, one concludes the pre-commutativity of the original case by using the transitivity in Proposition \ref{prop transtivity pre commute} and then obtain the strong commutativity by Theorem \ref{thm pre implies strong ds}.

\subsection{Structure of arguments}

We think it is helpful to give links on main results. Let $\Delta, \Delta'$ be segments. Let $\pi \in \mathrm{Irr}$. Let $\mathfrak p=\mathfrak{mx}_{\Delta}(\pi)$ (see Section \ref{ss multisegment for eta}). 

\begin{itemize}
\item[(PCSeg)] Pre-commutativity for $(\mathrm{St}(\Delta), \mathrm{St}(\Delta'), \pi)$
\item[(PCMulti)] Pre-commutativity for $(\mathrm{St}(\mathfrak p), \mathrm{St}(\Delta'), \pi)$
\item[(SCSeg)] Strong commutativity for $(\mathrm{St}(\Delta), \mathrm{St}(\Delta'), \pi)$  
\item[(SCMulti)] Strong commutativity for $(\mathrm{St}(\mathfrak p), \mathrm{St}(\Delta'), \pi)$ 
\item[(CC)] Combinatorial commutativity for $(\Delta, \Delta', \pi)$
\item[(DualPCMulti)] Dual for pre-commutativity for $(\mathrm{St}(\Delta'), \mathrm{St}(\mathfrak p), I_{\Delta'}\circ D_{\mathfrak p}(\pi))$
\item[(DualPCSeg)] Dual for pre-commutativity for $(\mathrm{St}(\Delta'), \mathrm{St}(\Delta), I_{\Delta'}\circ D_{\Delta}(\pi))$
\item[(DualSC)] Dual for strong commutativity
\item[(DualCC)] Dual for combinatorial commutativity
\end{itemize}

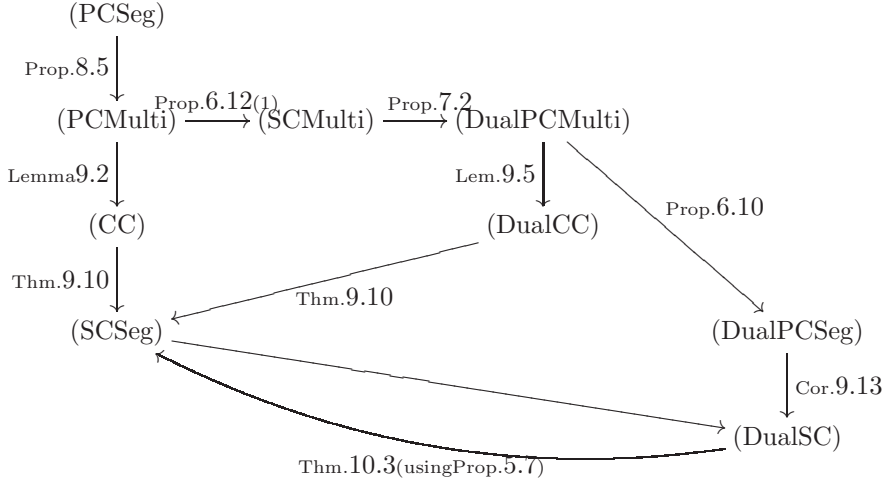
\[ \xymatrix{  \mathrm{(PCSeg)} \ar[d]_{\mathrm{Prop. \ref{prop completing pre comm}}}  \\   \mathrm{(PCMulti)}   \ar[d]_{\mathrm{Lemma \ref{lem eta unchange}}} \ar[r]^{\mathrm{Prop. \ref{thm pre implies strong}(1)}}  &  \mathrm{(SCMulti)} \ar[r]^{\mathrm{Prop. \ref{prop commutative triple}}} & \mathrm{(DualPCMulti)}\ar[ddr]^{\mathrm{Prop. \ref{prop commutative triple induct}}} \ar[d]_{\mathrm{Lem. \ref{lem eta for  another side}}}  \\   \mathrm{(CC)} \ar[d]_{\mathrm{Thm. \ref{thm pre imply strong}}}  &  & \mathrm{(DualCC)} \ar[dll]^{\mathrm{Thm. \ref{thm pre imply strong}}} & \\
   \mathrm{(SCSeg)} \ar[rrrd] &  &  &  \mathrm{(DualPCSeg)} \ar[d]^{\mathrm{Cor. \ref{cor dual strong commutative triple}}}  \\
	 & && \mathrm{(DualSC) \ar@/^2.0pc/[lllu]^{\mathrm{Thm. \ref{thm combinatorial def} (using Prop.\ref{prop trivial orbit in irreducible case})}} }
}
\]


\subsection{Applications}
Other than the duality above, the commutation (Proposition \ref{prop strong commute imply commute}) is also useful and is straightforward from Definition \ref{def strong comm} (from representation-theoretic perspective). On the other hand, Definition \ref{def combinatorial comm triple} (from combinatorial perspective) is useful in proving properties for the generalized GGP relevance in \cite{Ch22+c}. Further applications need to combine with the study in \cite{Ch22+, Ch22+a}. For example, combining with results in a series of articles \cite{Ch22+, Ch22+a}, one can prove the uniqueness of the relevant pairs under some minimal condition in \cite{Ch22+a}.

\subsection{Organizations}

Section \ref{s unique maps} studies some uniqueness of simple submodules/quotients from $\square$-irreducible representations. Section \ref{s geometric lemma geometry} gives some preparation on the geometric lemma. Section \ref{s irr pair der} studies a structure arising from the Jacquet functor, which we refer to an irreducible pair. Section \ref{s strong commutation for der and int} studies our main notions: pre-commutativity and strong commutativity, and we show a transitivity property for strong commutativity. Section \ref{s dual strong commut} studies a duality for strong commutativity. Section \ref{s construct precommut} constructs some pre-commutativity involving an irreducible pair. Section \ref{s pre implies strong in sq} proves that pre-commutativity implies strong commutativity for essentially square-integrable representations, using results in previous sections. Section \ref{s combinatorial commutation} shows our main result on equivalent definitions in terms of duality in Section \ref{s dual strong commut} as well as a combinatorial formulation. Section \ref{s application on strong commut triples} studies some consequences from our results. Section \ref{s branching laws} defines a notion of generalized GGP relevant pairs and gives some examples. 

\subsection{Acknowledgment} 
Part of results in this article was announced in the workshop of minimal representations and theta correspondence in the ESI at Vienna in April 2022. The author would like to thank the organizers, Wee Teck Gan, Marcela Hanzer, Alberto M\'inguez, Goran Mui\'c and Martin Weissman, for their kind invitation for giving a talk related to this work. This project is supported in part by the Research Grants Council of the Hong Kong Special Administrative Region, China (Project No: 17305223) and  the National Natural Science Foundation of China (Project No. 12322120).

\section{Preliminaries} \label{s prelim}




For results involving derivatives and integrals, most of time, we shall only prove results involving $D_{\sigma}=D^R_{\sigma}$ and $I_{\sigma}=I^L_{\sigma}$ for $\sigma \in \mathrm{Irr}^{\square}$. The analogous result that switches between $D^R_{\sigma}$ and $D^L_{\sigma}$, and between $I^R_{\sigma}$ and $I^L_{\sigma}$ can be proved similarly. From Sections \ref{ss left right switch} and \ref{ss switch left right jacquet functor}, we give some results that one can deduce results from one version to another.

\subsection{Second adjointness}

Let $J$ be the matrix in $G_n$ with $1$ in the anti-diagonal entries and $0$ in other entries. Let $\theta=\theta_n: G_n \rightarrow G_n$ given by $\theta(g)=Jg^{-T}J$. Then $\theta$ also deduces an auto-equivalence, still denoted by $\theta$, on the category of smooth representations of $G_n$.

We shall use the following standard fact. The proof is similar to the one in \cite[Proposition 2.1]{Ch22+} and we only sketch it.

\begin{proposition} \label{prop second adjoint map}
Let $\pi_1$ and $\pi_2$ be smooth representations of $G_{n_1}$ and $G_{n_2}$ respectively. Let $\pi$ be a smooth representation of $G_{n_1+n_2}$. Then 
\[  \mathrm{Hom}_{G_{n_1+n_2}}(\pi_1 \times \pi_2, \pi) \cong \mathrm{Hom}_{G_{n_2}\times G_{n_1}}(\pi_2 \boxtimes \pi_1, \pi_{N_{n(\pi_1)}}) .
\]
\end{proposition}

\begin{proof}
Define $\theta'(g_1, g_2)=(g_2, g_1)$. It follows from definitions that $\theta'(\pi_{N^-})\cong (\theta(\pi))_{N_{n(\sigma_2), n(\sigma_1)}}$, where $N^-$ is the unipotent subgroup in the parabolic subgroup opposite to $P_{n_1, n-2}$. Then 
\begin{align*}
   \mathrm{Hom}_{G_{n_1+n_2}}(\theta(\pi_1)\times \theta(\pi_2), \theta(\pi)) & \cong \mathrm{Hom}_{G_{n_1}\times G_{n_2}}(\theta(\pi_1)\boxtimes \theta(\pi_2), \theta(\pi)_{N^-}) \\
	        &\cong \mathrm{Hom}_{G_{n_2}\times G_{n_1}}(\theta(\pi_2)\boxtimes \theta(\pi_1), \theta(\pi)_{N_{n(\pi_2), n(\pi_1)}}) .
\end{align*}
Since $\pi_1, \pi_2$ and $\pi$ are arbitrary and $\theta$ is a bijection, the above isomorphism implies the proposition.
\end{proof}

One consequence on the structure of Jacquet functors is as follows:

\begin{corollary}
Let $\sigma \in \mathrm{Irr}^{\square}$ and let $\pi \in \mathrm{Irr}$. Suppose $D_{\sigma}(\pi)\neq 0$. Then $D_{\sigma}(\pi)\boxtimes \sigma$ appears in both submodule and quotient of $\pi_{N_{n(\sigma)}}$. 
\end{corollary}

\begin{proof}
By definition, $\pi$ appears as a submodule in $D_{\sigma}(\pi) \times \sigma$. Then, by \cite[Corollary 2.4]{LM19}, $\pi$ also apperas as a quotient of $\sigma \times D_{\sigma}(\pi)$. Now the statement follows from Proposition \ref{prop second adjoint map}.
\end{proof}

\subsection{Switching between left and right versions (parabolic inductions)} \label{ss left right switch}

 By (\ref{eqn square irreducible}) and the fact that $\theta$ preserves the irreducibility, $\pi$ is $\square$-irreducible if and only if $\theta(\pi)$ is $\square$-irreducible. Moreover, it is a simple fact that for a cuspidal representation (resp. essentially square-integrable representation) $\rho$ of $G_n$, $\theta(\rho)$ is still cuspidal (resp. essentially square-integrable).

Using $\theta(P_{n_1, n_2})=P_{n_2,n_1}$, we deduce that
\begin{align} \label{eqn square irreducible}
  \theta(\pi_1 \times \pi_2) \cong \theta(\pi_2) \times \theta(\pi_1).
\end{align}

\begin{lemma} \label{lem derivative integral under theta}
Let $\sigma \in \mathrm{Irr}^{\square}$. Let $\pi \in \mathrm{Irr}$. Then 
\[ \theta(I^L_{\sigma}(\pi))\cong I^R_{\theta(\sigma)}(\theta(\pi)), \quad \theta(I^R_{\sigma}(\pi)) \cong I^L_{\theta(\sigma)}(\theta(\pi)) .
\]
\end{lemma}

\begin{proof}
By definition, we have $I^L_{\sigma}(\pi) \hookrightarrow \sigma \times \pi$. Hence, by applying $\theta$, we have:
\[   \theta(I^L_{\sigma}(\pi)) \hookrightarrow \theta(\sigma \times \pi) \cong \theta(\pi) \times \theta(\sigma) .
\] 
Hence, by definition, we have:
\[  \theta(I^L_{\sigma}(\pi)) \cong I^R_{\theta(\sigma)}(\theta(\pi)) .
\]
\end{proof}

\subsection{Switching between left and right versions (Jacquet functors)} \label{ss switch left right jacquet functor}

\begin{lemma} \label{lem switch by theta}
Let $\pi_1$ and $\pi_2$ be smooth representations of $G_{n_1}$ and $G_{n_2}$ respectively. Let $\pi$ be a smooth representation of $G_{n_1+n_2}$. Then
\[   \mathrm{Hom}_{G_{n_1}\times G_{n_2}}(\pi_1 \boxtimes \pi_2, \pi_{N_{n_2}}) \cong \mathrm{Hom}_{G_{n_2}\times G_{n_1}}(\theta(\pi_2)\boxtimes \theta(\pi_1), \theta(\pi)_{N_{n_1}}) .
\]
\end{lemma}

\begin{proof}
Let $w_1$ and $w_2$ be elements in $G_{n_1}$ and $G_{n_2}$ with $1$ in anti-diagonal entries and $0$ elsewhere. Given a map $f \in \mathrm{Hom}_{G_{n_2}\times G_{n_1}}(\theta(\pi_2)\boxtimes \theta(\pi_1), \theta(\pi)_{N_{n_1}})$, we define a map $f'$ in $\mathrm{Hom}_{G_{n_1}\times G_{n_2}}(\pi_1\boxtimes \pi_2, \pi_{N_{n_2}})$ determined by: for $x_1 \in \pi_1$ and $x_2 \in \pi_2$,
\[    f'( x_1 \otimes x_2 ) =f((w_2.x_2)\otimes (w_1.x_1))  .
\]
Then, it is straightforward to check that $f'$ is $G_{n_1} \times G_{n_2}$-equivariant. The inverse map can be similarly defined.
\end{proof}

\begin{lemma}
Let $\sigma \in \mathrm{Irr}^{\square}$. Let $\pi \in \mathrm{Irr}(G_n)$. Then 
\[   \theta(D^L_{\sigma}(\pi)) \cong D^R_{\theta(\sigma)}(\theta(\pi)), \quad \theta(D^R_{\sigma}(\pi)) \cong D^L_{\theta(\sigma)}(\theta(\pi)) .
\]
\end{lemma}

\begin{proof}
Let $m=n(\sigma)$. By definition, we have:
\[    D^R_{\sigma}(\pi) \boxtimes \sigma \hookrightarrow \pi_{N_m} .
\]
Then, by Lemma \ref{lem switch by theta}, we have:
\[  \theta(\sigma) \boxtimes \theta(D^R_{\sigma}(\pi))  \hookrightarrow \theta(\pi)_{N_{n-m}} .
\]
Then, by definition, we have $\theta(D^R_{\sigma}(\pi)) \cong D_{\theta(\sigma)}^L(\theta(\pi))$. 
\end{proof}

\subsection{Switching left and right versions by using duals} \label{ss switch left right by dual}




We again have the standard fact that an irreducible representation $\sigma$ of $G_n$ is $\square$-irreducible if and only if $\sigma^{\vee}$ is $\square$-irreducible. Hence, for $\sigma \in \mathrm{Irr}^{\square}$, $I_{\sigma^{\vee}}^R$ and $I_{\sigma^{\vee}}^L$ are well-defined.

\begin{lemma} \label{lem dual on integrals}
Let $\sigma \in \mathrm{Irr}^{\square}$. Let $\pi \in \mathrm{Irr}$. Then $I^L_{\sigma}(\pi)^{\vee} \cong I^R_{\sigma^{\vee}}(\pi^{\vee})$.
\end{lemma}

\begin{proof}
By definition, we have:
\[   I^L_{\sigma}(\pi) \hookrightarrow \sigma \times \pi .
\]
Hence, we have a surjection from $\sigma^{\vee}\times \pi^{\vee} \cong (\sigma \times \pi)^{\vee}$ to $I^L_{\sigma}(\pi)^{\vee}$. By \cite[Corollary 2.4]{LM19}, $I^L_{\sigma}(\pi)^{\vee}$ embeds to $\pi^{\vee} \times \sigma^{\vee}$. Thus, by definition, we have $I^L_{\sigma}(\pi)^{\vee} \cong I^R_{\sigma^{\vee}}(\pi^{\vee})$. 
\end{proof}

\begin{corollary} \label{cor dual on derivatives}
Let $\sigma \in \mathrm{Irr}^{\square}$. Let $\pi \in \mathrm{Irr}$. Suppose $D^L_{\sigma}(\pi) \neq 0$. Then $D^L_{\sigma}(\pi)^{\vee} \cong D^R_{\sigma^{\vee}}(\pi^{\vee})$.
\end{corollary}

\begin{proof}
Let $\tau =D^L_{\sigma}(\pi)$. Then $I_{\sigma}^L(\tau) \cong \pi$. Hence, $(I_{\sigma}^L(\tau))^{\vee} \cong I_{\sigma^{\vee}}^R(\tau^{\vee})$. Hence, we have $\pi^{\vee} \cong I_{\sigma^{\vee}}^R(\tau^{\vee})$. Applying $D^L_{\sigma^{\vee}}$, we have $D_{\sigma^{\vee}}^R(\pi^{\vee}) \cong (D^L_{\sigma}(\pi))^{\vee}$. 
\end{proof}

\section{Some uniqueness from $\square$-irreducible representations} \label{s unique maps}

In this section, we shall first deduce some uniqueness property from \cite{KKKO15} or \cite{LM19}. For $\pi \in \mathrm{Alg}(G_n)$, let $n(\pi)=n$.

\subsection{Uniqueness}

We remark that via Frobenius reciprocity, we have $\pi \hookrightarrow D_{\sigma}(\pi)\times \sigma$ (resp. $\pi \hookrightarrow \sigma \times D^L_{\sigma}(\pi)$) when $D_{\sigma}(\pi)\neq 0$ (resp. $D^L_{\sigma}(\pi)\neq 0$). We shall frequently use these facts.

Earlier forms of derivatives for cuspidal representations go back to the work of Jantzen \cite{Ja07} and M\'inguez \cite{Mi09}.

\begin{lemma} \label{lem unique embedding}
Let $\sigma_1, \sigma_2, \ldots, \sigma_r \in \mathrm{Irr}^{\square}$ such that $\sigma_1 \times \ldots \times \sigma_r$ is still $\square$-irreducible. Let $\pi \in \mathrm{Irr}$. Then there exists at most one $\omega \in \mathrm{Irr}$ such that 
\[  \omega \boxtimes \sigma_r \boxtimes \ldots \boxtimes \sigma_1 \hookrightarrow \pi_N,
\]
where $N=N_{n(\pi)-n(\sigma_1)-\ldots -n(\sigma_r), n(\sigma_r), \ldots, n(\sigma_1)}$. Moreover, if such $\omega$ exists, the embedding is unique.
\end{lemma}
\begin{proof}
Let $n'=n(\sigma_1)+\ldots+n(\sigma_r)$. By Frobenius reciprocity and Proposition \ref{prop second adjoint map}, we have a natural isomorphism:
\[ \mathrm{Hom}_{G'}(\omega \boxtimes \sigma_r \boxtimes \ldots \boxtimes \sigma_1, \pi_N) \cong \mathrm{Hom}_{G''}(\omega \boxtimes (\sigma_1 \times \ldots \times \sigma_r), \pi_{N_{n'}}) ,
\]
where $G'=G_{n-n'}\times G_{n(\sigma_1)}\times \ldots \times G_{n(\sigma_r)}$ and $G''=G_{n-n'}\times G_{n'}$. Note that the RHS term above is further isomorphic to 
\[ \mathrm{Hom}_{G_n}((\sigma_r\times \ldots \times \sigma_1)\times \omega, \pi) 
\]
via Frobenius reciprocity and so has the multiplicity at most one by \cite[Corollary 2.4 and Lemma 2.8]{LM19}. Thus we also have the LHS term above has the multiplicity at most one, as desired.
\end{proof}

\begin{lemma} \label{lem uniqueness of products}
Let $\pi \in \mathrm{Irr}(G_n)$. Let $\sigma_1 \in \mathrm{Irr}^{\square}(G_k)$ and let $\sigma_2 \in \mathrm{Irr}^{\square}(G_l)$. Let $\tau=I_{\sigma_1}^R\circ I_{\sigma_2}^L(\pi)$ or $I_{\sigma_2}^L\circ I_{\sigma_1}^R(\pi)$. Then 
\[  \mathrm{dim}~ \mathrm{Hom}_{G_{n+k+l}}( \tau, \sigma_2 \times \pi \times \sigma_1) =1 .
\]
\end{lemma}

\begin{proof}
We only consider $\tau=I_{\sigma_2}^L\circ I_{\sigma_1}^R(\pi)$ and the other one is similar. Note that $\tau$ has to be an irreducible submodule of $\sigma_2 \times \tau'$ for a simple composition factor $\tau'$ in $\pi \times \sigma_1$. But $\tau\hookrightarrow \sigma_2 \times \tau'$ implies $\tau \cong I_{\sigma_2}^L(\tau')$ and so $\tau'=I_{\sigma_1}^R(\pi)$. Since $I_{\sigma_1}^R(\pi)$ appears with multiplicity one in $\pi \times \sigma_1$ (\cite[Theorem 3.2]{KKKO15}, \cite[Lemma 2.8]{LM19}), we then have the unique embedding from $I_{\sigma_2}^L\circ I_{\sigma_1}^R(\pi)$ to $\sigma_2 \times \pi \times \sigma_1$ by \cite[Lemma 2.8]{LM19}. 
\end{proof}

\begin{lemma} \label{lem unique embedding left right}
Let $\sigma_1, \sigma_2 \in \mathrm{Irr}^{\square}$. Let $\pi\in \mathrm{Irr}(G_n)$. Let $\tau=D^L_{\sigma_2}\circ D^R_{\sigma_1}(\pi)$ or $D^R_{\sigma_1}\circ D^L_{\sigma_2}(\pi)$, and assume it is non-zero. Then 
\[ \mathrm{dim}~ \mathrm{Hom}_{G_{n_2}\times G_{n'}\times G_{n_1}}(\sigma_2\boxtimes \tau\boxtimes \sigma_1, \pi_N) =1 ,
\]
where $n_1=n(\sigma_1)$, $n_2=n(\sigma_2)$, $n'=n-n_1-n_2$ and $N=N_{n_2,n', n_1}$.
\end{lemma}

\begin{proof}
By applying Bernstein's second adjointness theorem, we have:
\[  \mathrm{Hom}_{G_{n_2}\times G_{n'}\times G_{n_1}}(\sigma_2\boxtimes \tau \boxtimes \sigma_1, \pi_N) \cong \mathrm{Hom}_{G_n}(\sigma_1\times \tau \times \sigma_2, \pi) .
\]
By taking the dual, it is equivalent to show:
\[ \mathrm{dim} \mathrm{Hom}_{G_n}(\pi^{\vee}, \sigma_1^{\vee} \times \tau^{\vee} \times \sigma_2^{\vee}) =1
\]
Suppose we are in the case that $\tau=D^R_{\sigma_1}\circ D^L_{\sigma_2}(\pi)$. By Lemma \ref{lem dual on integrals} and Corollary \ref{cor dual on derivatives}, we have that $I_{\sigma_1^{\vee}}^L(\tau^{\vee})\cong I_{\sigma_1}^R(\tau)^{\vee} \cong D^L_{\sigma_2}(\pi)^{\vee} \cong D^R_{\sigma_2^{\vee}}(\pi^{\vee})$. Hence, $I^R_{\sigma_2^{\vee}}\circ I^L_{\sigma_1^{\vee}}(\tau^{\vee}) \cong \pi^{\vee}$. Then the dimension follows from Lemma \ref{lem uniqueness of products}.
\end{proof}

\section{Some geometry of orbits in the geometric lemma} \label{s geometric lemma geometry}


Note that the uniqueness statements in Section \ref{s unique maps} are up to a scalar. This is good for our purposes since we only concern about the image of the maps (see e.g. Definition \ref{def strong comm}). By abuse of terminology, we say that a diagram is commutative if the images of any composition of maps starting from the same representation and ending at the same representation coincide.

\subsection{General notions for supporting orbits} \label{ss supporting orbit}

Let $G$ be a connected reductive group over a non-Archimedean local field. Although we mainly consider $G_n$, we need slightly more general setting such as $G_{n_1}\times G_{n_2}$ in our discussions and so we consider a general setting.

We fix a minimal parabolic subgroup of $G$ and so fix root system for $G$. For a standard parabolic subgroup $P$ in $G$, denote by $M_PN_P$ the Levi decomposition of $P$, with $M_P$ to be Levi and $N_P$ to be unipotent, and let $\pi$ be a smooth representation of $M_P$, inflated to a $P$-representation. Let $W$ be the Weyl group of $G$. Let $l: W\rightarrow \mathbb Z_{\geq 0}$ be the length function. Denote by $\mathrm{Ind}_P^G\pi$ the normalized parabolic induction from $\pi$. Recall that the underlying space for $\mathrm{Ind}_P^G\pi$, which we shall also denote by $C^{\infty}(P\setminus G, \pi)$, is the space of smooth functions from $G$ to $\pi$ satisfying 
\[ f(pg)=\delta_P(p)^{1/2}p.f(g) \]
for any $p \in P$, where $\delta_P$ is the modular character of $P$. Let $Q$ be another parabolic subgroup of $G$ with the Levi decomposition $M_QN_Q$. Let $W(P)$ and $W(Q)$ be the associated Weyl groups for $P$ and $Q$ respectively. Let $W_{P,Q}=W_{P,Q}(G)$ be the set of minimal representatives for the double cosets in $W(P)\setminus W/W(Q)$. 

We shall enumerate elements in $W_{P,Q}$ labelled as $w_1, \ldots, w_r$ such that $i<j$ implies $w_i \not\geq w_j$. In particular, $w_1$ is the trivial element. The geometric lemma (\cite[Remarks 5.5]{BZ77} and \cite[Proposition 5.11]{BZ77}) asserts that $C^{\infty}(P\setminus G, \pi)_{N_Q}$ admits a filtration  
\[  0 \subset \mathcal J_r  \subset  \ldots \subset \mathcal J_1 =C^{\infty}(P\setminus G, \pi)_{N_Q}
\]
such that 
\begin{align} \label{eqn geometric lemma isom}
  \mathcal J_i/\mathcal J_{i+1} \cong \mathrm{Ind}_{P^{w_i}\cap M_Q}^{M_Q} ((\pi_{M_P\cap N_Q^{w_i}})^{w_i}),
\end{align}
where 
\begin{itemize}
\item $P^w=\dot{w}_i^{-1}P\dot{w}_i$, $N_Q^w=\dot{w}_i^{-1}N_Q\dot{w}_i$ for some representative $\dot{w}_i$ of $w_i$ in $G$;
\item $(\pi_{M_P\cap N_Q^w})^w$ is a $M_P^{w^{-1}}\cap M_Q$-representation, which $m$ acts on the space by the action $\dot{w}_i^{-1}m\dot{w}_i$ for $m \in M_P^{w^{-1}}\cap M_Q$.
\end{itemize}
The isomorphism  in (\ref{eqn geometric lemma isom}), denoted by $\Phi_i$, is given as follows. For $f \in \mathcal J_i$ and $m \in M_Q$,
\[  \Phi_w(f)(m)=\int_{N'} (\mathrm{pr}\circ \widetilde{f})(wnm)~dn ,
\]
where $\widetilde{f}$ is a representative of $f$ in $C^{\infty}(P\setminus G, \pi)$; $N'={}^{w_i}(N^-)\cap N_Q$; $dn$ is a Haar measure on $N'$ and $\mathrm{pr}$ is the projection from $\pi$ to $\pi_{M_P \cap N_Q^{w_i}}$. 



We say that the embedding $\lambda \hookrightarrow (\mathrm{Ind}_P^G\pi)_{N_Q}$ has the supporting orbit of the form $P\dot{w}_iQ$ if $\lambda \cap \mathcal J_{i+1} =0 , \quad \lambda \cap \mathcal J_i \neq 0$. We say that the embedding $\lambda \hookrightarrow (\mathrm{Ind}_P^G\pi)_{N_Q}$ has the {\it trivial supporting orbit}  if 
\[  \lambda \cap \mathcal J_2 =0 .
\]
The trivial supporting orbit case is the most interesting case for us and is independent of a choice of an enumeration. The notion depends on the term $(\mathrm{Ind}_P^G\pi)_{N_Q}$ which the geometric lemma carries for, and it should be clear for the context. For example, for $\pi_1 \in \mathrm{Irr}(G_{n_1})$, $\pi_2 \in \mathrm{Irr}(G_{n_2})$, $(\sigma_1 \times \sigma_2)_{N_m}$ means to consider the orbits $P_{n_1, n_2}\setminus G_{n_1+n_2}/Q_{n_1+n_2-m,m}$.

\subsection{Geometric lemma} \label{ss geo lem}
We shall now introduce more product notations. For a $G_{n_1}\times G_{n_2}$-representation $\pi$ and a $G_m$-representation $\tau$, inflate $\tau \boxtimes \pi$ to a $P_{m,n_1}\times G_{n_2}$-representation. Define 
\begin{equation} \label{eqn embedding 1}
 \tau \dot{\times}^1 \pi := \mathrm{Ind}_{P_{m, n_1}\times G_{n_2}}^{G_{m+n_1}\times G_{n_2}} \tau \boxtimes \pi .
\end{equation}
Similarly, define
\begin{equation} \label{eqn embedding 2}
 \tau \dot{\times}^2 \pi := \mathrm{Ind}_{G_{n_1}\times P_{m,n_2}}^{G_{n_1} \times G_{m+n_2}} \tau \boxtimes \pi ,
\end{equation}
where the Levi part $G_{n_1}\times G_m \times G_{n_2}\subset G_{n_1}\times P_{m,n_2} \subset G_{n_1}\times G_{m+n_2}$ acts on $\tau \boxtimes \pi$ by 
\[ (g_1, g_2, g_3).(v_1\boxtimes v_2)=g_2.v_1 \boxtimes (g_1,g_3).v_2    \]
and the unipotent part acts trivially. 


Let $\pi$ and $\pi'$ be representations of $G_{n_1}$ and $G_{n_2}$ respectively. Then the geometric lemma asserts that $(\pi \times \pi')_{N_i}$ admits a filtration with layers of the form:
\[   \mathrm{Ind}_{P_{n_1-i_1, n_2-i_2}\times P_{i_1, i_2}}^{G_{n-i}\times G_i} (\pi_{N_{i_1}} \boxtimes \pi'_{N_{i_2}})^{\phi},
\]
where $\phi$ is a natural twist sending $G_{n_1-i_1}\times G_{i_1}\times G_{n_2-i_2}\times G_{i_2}$-representations to $G_{n_1-i_1}\times G_{n_2-i_2}\times G_{i_1}\times G_{i_2}$-representations.When $i \leq n_2$, the top layer is isomorphic to $\pi \dot{\times}^1 (\pi'_{N_i})$, which will be used often when involving Definition \ref{def strong comm}.





\section{Irreducible pairs under derivatives} \label{s irr pair der} 








\subsection{$\eta$-invariants}

To discuss an important class of Rd-irreducible pairs in Definition \ref{def irr pair}, we need more notations. Recall that $\varepsilon$ and $\eta$ are defined in Section \ref{ss combin commut triples}.

\begin{definition} \label{def comb commute}
Let $\Delta=[a,b]_{\rho}$ be a segment. Let $\pi \in \mathrm{Irr}$.
\begin{itemize}
\item We write $\eta_{\Delta}(\pi)=0$ if $\varepsilon_{[c,b]_{\rho}}(\pi)=0$ for all $c=a, \ldots, b$. Similarly, we write $\eta_{\Delta}(\pi)\neq 0$ if $\varepsilon_{[c,b]_{\rho}}(\pi)\neq 0$ for some $c=a, \ldots, b$. 
\item We write $\eta_{\Delta}(\pi) ? \ \eta_{\Delta}(\pi')$ for $?\in \left\{ =,  \leq , <, \geq, > \right\}$ if $\varepsilon_{[c,b]_{\rho}}(\pi) ? \varepsilon_{[c,b]_{\rho}}(\pi')$ for all $c$ satisfying $a\leq c\leq b$. 
\item We similarly define the left version terminologies for $\varepsilon^L_{\Delta}$ and $\eta^L_{\Delta}$ if one uses the left derivatives $D^L_{\Delta}$ (also see (\ref{def max multisegment 2}) below). 
\end{itemize}
\end{definition}

\subsection{Multisegment counterpart of the $\eta$-invariant} \label{ss multisegment for eta}

Following \cite{Ze80}, a {\it multisegment} is a multiset of non-empty segments. The irreducible representations of $G_n$ are parametrized by multisegments, see work of Zelevinsky, Tadi\'c and M\'inguez-S\'echerre \cite{Ze80, Ta90, MS14}. 

Two segments $\Delta_1, \Delta_2$ are {\it linked} if $\Delta_1 \not\subset \Delta_2$ and $\Delta_2 \not\subset \Delta_1$ and $\Delta_1\cup\Delta_2$ is still a segment. Otherwise, we say that $\Delta_1$ and $\Delta_2$ are unlinked. A multisegment $\mathfrak m$ is said to be {\it pairwise unlinked} if any two segments $\Delta_1$ and $\Delta_2$ in $\mathfrak m$ are unlinked. A main property for unlinked segments $\Delta_1$ and $\Delta_2$ is the following \cite{Ze80, Ta90}:
\begin{align} \label{eqn comm unlinked segments}
  \mathrm{St}(\Delta_1)\times \mathrm{St}(\Delta_2) \cong \mathrm{St}(\Delta_2)\times \mathrm{St}(\Delta_1) .
\end{align}
For a pairwise unlinked multisegment $\mathfrak m=\left\{ \Delta_1, \ldots, \Delta_r \right\}$, let 
\[ \mathrm{St}(\mathfrak m) = \mathrm{St}(\Delta_1)\times \ldots \times \mathrm{St}(\Delta_r) ,
\]
which is irreducible and $\square$-irreducible, and is independent of the ordering by \cite{Ze80, Ta90}, and we write:
\[ D_{\mathfrak m}(\pi)= D_{\mathrm{St}(\mathfrak m)}(\pi).
\]

For a segment $\Delta=[a,b]_{\rho}$, a segment $[a',b']_{\rho}$ is said to be {\it $\Delta$-saturated} if $a'\leq a$ and $b'=b$. A multisegment $\mathfrak m$ is said to be {\it $\Delta$-saturated} if all segments in $\mathfrak m$ are $\Delta$-saturated. As we shall see later, $\Delta$-saturated segments can be used to produce more strongly commutative triples (see Proposition \ref{prop completing pre comm} and Corollary \ref{cor satruated commute triple revise}).

For $\pi \in \mathrm{Irr}$ and a segment $\Delta=[a,b]_{\rho}$, define 
\begin{align} \label{def max multisegment}
  \mathfrak{mx}_{ \Delta}(\pi):=\mathfrak{mx}^R_{\Delta}(\pi):=\sum_{k=0}^{b-a}\varepsilon_{[a+k,b]_{\rho}}(\pi)\cdot [a+k,b]_{\rho} ,
\end{align}
\begin{align} \label{def max multisegment 2}
  \mathfrak{mx}^L_{\Delta}(\pi):=\sum_{k=0}^{b-a}\varepsilon^L_{[a,b-k]_{\rho}}(\pi)\cdot [a,b-k]_{\rho} ,
\end{align}
where the numbers $\varepsilon_{[a+k,b]_{\rho}}(\pi)$ and $\varepsilon^L_{[a,b-k]_{\rho}}(\pi)$ give the multiplicities of those segments. 



\subsection{Examples of Rd-irreducible pairs}

\begin{definition} \label{def irr pair}
Let $\sigma \in \mathrm{Irr}^{\square}(G_r)$. Let $\pi \in \mathrm{Irr}(G_n)$ such that $D_{\sigma}(\pi) \neq 0$. We say that $(\sigma, \pi)$ is a {\it Rd-irreducible pair} (resp. {\it Ld-irreducible pair}) if $\pi_{N_r}$ (resp. $\pi_{N_{n-r}}$) has a direct summand of the form $\omega  \boxtimes \sigma$ (resp. $\sigma \boxtimes \omega$) for some $\omega \in \mathrm{Irr}(G_{n-r})$. 
\end{definition}

We now give an important class of Rd-irreducible pairs. For a segment $\Delta=[a,b]_{\rho}$, let $l_{abs}(\Delta)=(b-a+1)n(\rho)$. For a multisegment $\mathfrak m$, let $l_{abs}(\mathfrak m)=\sum_{\Delta \in \mathfrak m}l_{abs}(\Delta)$.


\begin{proposition} \label{prop mx rd irr pair}
Let $\pi \in \mathrm{Irr}$. Let $\mathfrak q$ be a submultisegment of $\mathfrak{mx}_{\Delta}(\pi)$ such that $D_{\mathfrak q}(\pi)\neq 0$. Then $(\mathrm{St}(\mathfrak q), \pi)$ is a Rd-irreducible pair if and only if $\mathfrak q=\mathfrak{mx}_{\Delta}(\pi)$. 
\end{proposition}

\begin{proof}
For the if direction, it is a simple application of the geometric lemma, which is shown in \cite[Proposition 11.1]{Ch22+b}.

We now consider the only if direction. Suppose $\mathfrak q \subsetneq \mathfrak{mx}_{\Delta}(\pi)$. Let $\mathfrak q'=\mathfrak{mx}_{\Delta}(\pi)-\mathfrak q$. Let $m=l_{abs}(\mathfrak q)$ and let $m'=l_{abs}(\mathfrak q')$. Note that $\mathfrak{mx}_{\Delta}(D_{\mathfrak q}(\pi))=\mathfrak q'$. If $(\mathrm{St}(\mathfrak q), \pi)$ is still a Rd-irreducible pair, then $D_{\mathfrak q}(\pi) \boxtimes \mathrm{St}(\mathfrak q)$ is an indecomposable summand in $\pi_{N_{m}}$. But, from the if part, 
$D_{\mathfrak q'}\circ D_{\mathfrak q}(\pi) \boxtimes \sigma'$ is an indecomposable summand in $D_{\sigma}(\pi)_{N_{m'}}$. From the transitivity of Jacquet functors, we then have that 
\[   D_{\mathfrak q'}\circ D_{\mathfrak q}(\pi)\boxtimes \mathrm{St}(\mathfrak q')\boxtimes \mathrm{St}(\mathfrak q)
\]
is a direct summand in $\pi_{N_{n-m'-m, m', m}}$. On the other hand, let $\widetilde{\sigma}=\mathrm{St}(\mathfrak{mx}_{\Delta}(\pi))$. By the if direction which is just proved, 
\[   D_{\widetilde{\sigma}}(\pi)\boxtimes \widetilde{\sigma}
\]
is a direct summand in $\pi_{N_{n(\widetilde{\sigma})}}$. Note that $\mathrm{St}(\mathfrak q')\boxtimes \mathrm{St}(\mathfrak q)$ is not a direct summand in $\widetilde{\sigma}_{N_{n(\sigma)}}$. Hence, the transitivity of Jacquet functors implies that the unique submodule (see Lemma \ref{lem unique embedding})
\[   D_{\widetilde{\sigma}}(\pi) \boxtimes \mathrm{St}(\mathfrak q')\boxtimes \mathrm{St}(\mathfrak q)
\]
does not form an indecomposable summand in $\pi_{N_{n-m-m', m', m}}$. Thus we obtain a contradiction.
\end{proof}

For applications, we shall use the following reformulation:

\begin{proposition} \label{prop reformulate rd irr}
Let $\sigma \in \mathrm{Irr}^{\square}(G_{n_1})$ and let $\pi \in \mathrm{Irr}(G_{n_2})$. Let $i_1: \sigma \boxtimes \pi \rightarrow I_{\sigma}(\pi)_{N_{n_2}}$ be the unique non-zero map and let $i_2: I_{\sigma}(\pi)_{N_{n_2}} \rightarrow \sigma \boxtimes \pi$ be the unique non-zero map. Then $(\sigma, \pi)$ is Rd-irreducible if and only if $i_2\circ i_1\neq 0$. 
\end{proposition}

\begin{proof}
Suppose $i_2 \circ i_1 \neq 0$. Then the map $i_1$ splits and so $\sigma \boxtimes \pi$ is a direct summand in $I_{\sigma}(\pi)_{N_{n_2}}$. 

We now prove the converse direction. Suppose $\sigma \boxtimes \pi$ is a direct summand in $I_{\sigma}(\pi)_{N_{n_2}}$. Then  it follows from the uniqueness that the embedding from the summand to $I_{\sigma}(\pi)_{N_{n_2}}$ coincides with $i_1$ and the surjection from $I_{\sigma}(\pi)_{N_{n_2}}$ to the summand coincides with $i_1$. Hence, $i_2\circ i_1$ is a map multiplied by a non-zero scalar and in particular is non-zero.
\end{proof}

\subsection{Trivial supporting orbit for $I_{\sigma}(\pi) \hookrightarrow \sigma \times \pi$}

The following Proposition \ref{prop trivial orbit and irred} is one important property for irreducible pairs and also is one key in proving equivalent definitions in Theorem \ref{thm combinatorial def}.

\begin{lemma} \label{lem factor through lemma}
Let $\sigma$ be a smooth representation of $G_{n_1}$ and let $\pi$ be a smooth representation of $G_{n_2}$. Let $q: \sigma \times \pi=C^{\infty}(P_{n_1,n_2}\setminus G_{n_1+n_2}, \sigma \boxtimes \pi) \rightarrow C^{\infty}(P_{n_1,n_2}\setminus P_{n_1,n_2}, \sigma \boxtimes \pi)$ be the natural projection map. Let $\omega$ be a smooth representation of $G_{n_1+n_2}$ such that there is a non-zero map $i$ from $\omega$ to $\sigma \times \pi$. Then $q \circ i\neq 0$.
\end{lemma}

\begin{proof}
Let $n=n_1+n_2$. By applying Frobenius reciprocity on the map $\omega \stackrel{i}{\rightarrow} \sigma \times \pi$, we also have a non-zero map $\omega_{N_{n_2}} \stackrel{i'}{\rightarrow} \sigma \boxtimes \pi $. Then $i(x)(g)=i'(g.x)$ for $x \in \omega$ and $g \in G_n$, where we regard $g.x$ as an element in $\omega_{N_{n_2}}$ via the natural projection. We can choose $x^* \in \omega$ such that $i'(x^*)\neq 0$. Then $i(x^*)$ (viewed as a function from $G_n$ to $\sigma\boxtimes \pi$) has a support containing $1$. This implies that $q \circ i \neq 0$. 
\end{proof}

\begin{proposition} \label{prop trivial orbit and irred}
Let $\sigma \in \mathrm{Irr}^{\square}$ and let $\pi \in \mathrm{Irr}$. Then the embedding 
\[ \sigma \boxtimes \pi \hookrightarrow I_{\sigma}(\pi)_{N_{n(\pi)}} \hookrightarrow (\sigma \times \pi)_{N_{n(\pi)}}  \] 
has the trivial supporting orbit if and only if $(\sigma, I_{\sigma}(\pi))$ is a Ld-irreducible pair.
\end{proposition}

\begin{proof}
Let $P=P_{n(\sigma), n(\pi)}$ and let $N=N_{n(\sigma), n(\pi)}$. By Lemma \ref{lem factor through lemma}, the unqiue map from $I_{\sigma}(\pi)_N$ to $\sigma \boxtimes \pi$ (arising from Frobenius reciprocity) factors through the quotient map $(\sigma \times \pi)_N \twoheadrightarrow \sigma \boxtimes \pi$ in the geometric lemma.


Suppose $(\sigma, I_{\sigma}(\pi))$ is a Ld-irreducible pair. By Proposition \ref{prop reformulate rd irr}, the composition of the unique embedding $\sigma \boxtimes \pi \rightarrow I_{\sigma}(\pi)_N$ and the quotient map above is non-zero, giving the trivial supporting orbit statement.


For the converse, the definition of the trivial supporting orbit implies that we have a non-zero composition:
\[ \sigma \boxtimes \pi \hookrightarrow I_{\sigma}(\pi)_N \hookrightarrow (\sigma \times \pi)_N \twoheadrightarrow \sigma \boxtimes \pi .
\]
Then, the Rd-irreducibility condition follows from Proposition \ref{prop reformulate rd irr} again.
\end{proof}

\subsection{Some variants}

We shall also need the following variation:

\begin{proposition} \label{prop trivial orbit in irreducible case}
Let $\sigma \in \mathrm{Irr}^{\square}(G_n)$. Let $\omega \in \mathrm{Irr}$ such that $(\sigma, \omega)$ is Ld-irreducible. Suppose $\omega \hookrightarrow \sigma \times \pi$ for some smooth (not necessarily irreducible) representation $\pi$. Then, the embedding
\[   \sigma \boxtimes D^L_{\sigma}(\omega) \hookrightarrow \omega_N \hookrightarrow (\sigma\times \pi)_N ,
\]
where $N=N_{n(\sigma),n-n(\sigma)}$, has the trivial supporting orbit.
\end{proposition}

\begin{proof}
It follows from the argument in Proposition \ref{prop trivial orbit and irred} that the composition
\[   \omega_N \hookrightarrow (\sigma \times \pi)_N \twoheadrightarrow \sigma \boxtimes \pi
\]
is non-zero (but it is not necessarily surjective since $\pi$ is not necessarily irreducible). Thus, $\omega_N$ has a quotient of the form $\sigma \boxtimes \widetilde{\pi}$ for some submodule $\widetilde{\pi}$ of $\pi$. By the condition of Ld-irreducibility, the unique simple submodule $\sigma \boxtimes D^L_{\sigma}(\omega)$ coincides with the unique simple quotient $\sigma \boxtimes D^L_{\sigma}(\omega)$. Thus, we must have that $\widetilde{\pi}$ coincides with $\pi$ and the above compositions are non-zero. Thus the supporting orbit for the embedding in the lemma has the trivial supporting orbit.
\end{proof}

We prove a situation of the trivial supporting orbit that will be used later. 

\begin{corollary} \label{example on trivial supporting orbit by irreducible}
Let $\Delta=[a,b]_{\rho}$ and let $\widetilde{\Delta}=[a',b]_{\rho}$ for some $ a'>a$.  Suppose we have
\[ \mathrm{St}(\Delta) \hookrightarrow \mathrm{St}(\widetilde{\Delta}) \times \pi. \] 
Then the induced embedding $\mathrm{St}(\widetilde{\Delta})\boxtimes \mathrm{St}([a,a'-1]_{\rho})$ to $(\mathrm{St}(\widetilde{\Delta})\times \pi)_N$ has the trivial supporting orbit. 
\end{corollary}

\begin{proof}
The Jacquet functor on $\mathrm{St}(\widetilde{\Delta})$ is semisimple and in particular this implies that $(\mathrm{St}(\widetilde{\Delta}), \mathrm{St}(\Delta))$ is Ld-irreducible. Now the trivial supporting orbit statement follows from Proposition \ref{prop trivial orbit in irreducible case}. 
\end{proof}

\subsection{Converse of Proposition \ref{prop trivial orbit in irreducible case} }\label{ss converse and big derivatives}

We prove the converse of Proposition \ref{prop trivial orbit in irreducible case} for a special case in Lemma \ref{lem indecomp trivial embedd}, which relies on a property shown in Lemma \ref{lem generic indecomp} below. Some arguments appear in \cite[Section 9.4]{Ch22+b} and we reproduce and slightly generalize for the convenience of the reader. Lemma \ref{lem indecomp trivial embedd} will be used to show certain embedding cannot have the trivial supporting orbit in Lemma \ref{lem impossible trivial supp orbit} below.

For $\pi \in \mathrm{Irr}$, there exists a unique multiset of cuspidal representations $\rho_1, \ldots, \rho_k$ such that $\pi$ is a simple composition factor in $\rho_1 \times \ldots \times \rho_k$. We denote such multiset by $\mathrm{csupp}(\pi)$, called the {\it cuspidal support} of $\pi$.



The Krull-Schmidt theorem asserts that any smooth representation of $G_n$ of finite length can be uniquely written as a direct sum of indecomposable representations. For $\pi \in \mathrm{Irr}$ and $\sigma \in \mathrm{Irr}^{\square}$ with $D_{\sigma}(\pi)\neq 0$, since 
\[   \mathrm{Hom}_{G_{n(\pi)-n(\sigma)}\times G_{n(\sigma)}}(D_{\sigma}(\pi)\boxtimes \sigma, \pi_{N_{n(\sigma)}}) \cong \mathbb C,\]
there is a unique indecomposable summand $\kappa$ in $\pi_{N_{n(\sigma)}}$ containing $D_{\sigma}(\pi)\boxtimes \sigma$ as a submodule. We remark that we do not know whether the embedding $\kappa$ to $\pi_{N_{n(\sigma)}}$ is unique in general.

\begin{lemma} \label{lem generic indecomp}
Let $\omega=\mathrm{St}(\mathfrak m)$ for some pairwise unlinked multisegment $\mathfrak m$. Let $\sigma=\mathrm{St}(\mathfrak m') \in \mathrm{Irr}$ for some pairwise unlinked multisegment $\mathfrak m'$ such that $D^L_{\sigma}(\omega)\neq 0$. Let $\kappa$ be the indecomposable component in $\omega_{N_l}$ ($l=l_{abs}(\mathfrak m)-l_{abs}(\mathfrak m')$) containing $\sigma \boxtimes D^L_{\sigma}(\omega)$ as a submodule. Then $\kappa$ has unique irreducible submodule and unique irreducible quotient and both are isomorphic to $\sigma \boxtimes D^L_{\sigma}(\omega)$. Furthermore, $D^L_{\sigma}(\omega) \cong \mathrm{St}(\mathfrak n)$ for some pairwise unlinked multisegment $\mathfrak n$.
\end{lemma}

\begin{proof}
We consider the component $\kappa$ of $\pi_{N_{n(\sigma)}}$, which has cuspidal support $( \mathrm{csupp}(\sigma), \mathrm{csupp}(\pi)-\mathrm{csupp}(\sigma))$. As shown in \cite[Lemma 10.14]{Ch22+} (which is stated for $D=F$ case, but the argument works for general $D$, also see \cite[Corollary 2.6]{Ch21}), any simple submodule of $\kappa$ is isomorphic to $\mathrm{St}(\mathfrak n') \boxtimes \mathrm{St}(\mathfrak n'')$ for some pairwise unlinked multisegment $\mathfrak n', \mathfrak n''$ and so $\mathrm{dim}~\mathrm{Hom}_{G_{n-n(\sigma)}\times G_{n(\sigma)}}( \sigma \boxtimes D^L_{\sigma}(\omega), \pi_{N_{n(\sigma)}}) \leq 1$ implies that $\kappa$ has unique submodule. Similar argument (or by taking dual) gives that $\kappa$ has a unique simple quotient, and the simple quotient is of the form. This implies the lemma.
\end{proof}

For $\pi \in \mathrm{Alg}(G_n)$ and $\sigma \in \mathrm{Irr}(G_k)$, define:
\[  \mathbb D_{\sigma}(\pi) = \mathrm{Hom}_{G_k}(\sigma, \pi_{N_k}),
\]
where $\pi_{N_k}$ is regarded as a $G_k$-module via the embedding to the second factor in $G_{n-k}\times G_k$. Moreover, $\mathbb D_{\sigma}(\pi)$ is viewed as a $G_{n-k}$-module via the map 
\[  (g.f)(v)=\mathrm{diag}(g, I_k).f(v) .
\]
It is called a {\it big derivative} in \cite{Ch22+b}. For a segment $\Delta$, we usually write $\mathbb D_{\Delta}(\pi)$ for $\mathbb D_{\mathrm{St}(\Delta)}(\pi)$. The left version of the big derivative $\mathbb D^L_{\sigma}(\pi)$ is similarly defined by using 
\[     \mathbb D^L_{\sigma}(\pi) = \mathrm{Hom}_{G_k}(\sigma, \pi_{N_{n-k}}),
\]
where $\pi_{N_{n-k}}$ is regarded as a $G_k$-module via the embedding to the first factor in $G_k \times G_{n-k}$. Again $\mathbb D^L_{\sigma}(\pi)$ has a natural $G_{n-k}$-module structure.

\begin{lemma} \label{lem SI property of big derivative} \cite{Ch22+}
Let $\sigma =\mathrm{St}(\mathfrak m) \in \mathrm{Irr}(G_k)$ for some pairwise unlinked multisegment $\mathfrak m$ and let $\pi=\mathrm{St}(\mathfrak n)$ for some pairwise unlinked multisegment $\mathfrak n$. Suppose $D^L_{\sigma}(\pi)\neq 0$. Then $\mathbb D^L_{\sigma}(\pi)$ satisfies the SI property i.e. the socle of $\mathbb D^L_{\sigma}(\pi)$ is irreducible and appears with multiplicity one in the Jordan-H\"older series of $\mathbb D^L_{\sigma}(\pi)$. 
\end{lemma}

\begin{lemma} \label{lem big derivative compared with jacquet} (c.f. \cite[Section 9.4]{Ch22+b})
We use the notations in Lemma \ref{lem SI property of big derivative}. Suppose $D^L_{\sigma}(\pi)\neq 0$. Let $\kappa$ be the indecomposable component in $\pi_{N_{n-k}}$ that has an irreducible submodule isomorphic to $\sigma \boxtimes D^L_{\sigma}(\pi)$. If $(\sigma, \pi)$ is not Ld-irreducible, then $\kappa \not\cong \sigma \boxtimes \mathbb D^L_{\sigma}(\pi)$. 
\end{lemma}

\begin{proof}
In the left version of \cite[Theorem 9.14]{Ch22+b}, we have shown the socle irreducible property for $\mathbb D^L_{\sigma}(\mathrm{St}(\mathfrak m))$ (i.e. the socle of $\mathbb D^L_{\sigma}(\mathrm{St}(\mathfrak m))$ is irreducible and appears with multiplicity one in the Jordan-H\"older series of $\mathbb D^L_{\sigma}(\mathrm{St}(\mathfrak m))$). On the other hand, the hypothesis and Lemma \ref{lem generic indecomp} imply that $\kappa$ contains the factor $\sigma \boxtimes D^L_{\sigma}(\mathrm{St}(\mathfrak m))$ with multiplicity at least $2$. This implies the non-isomorphism.
\end{proof}

\begin{lemma}  \label{lem indecomp trivial embedd} (c.f. \cite[Section 9.4]{Ch22+b})
Let $\sigma=\mathrm{St}(\mathfrak m) \in \mathrm{Irr}$ for some pairwise unlinked multisegment $\mathfrak m$. Let $\omega$ be an irreducible submodule of $\sigma \times \pi$ for some $\pi \in \mathrm{Alg}(G_n)$. Suppose $\omega \cong \mathrm{St}(\mathfrak n)$ for some pairwise unlinked multisegment $\mathfrak n$. Suppose $(\sigma, \pi)$ is not Ld-irreducible. Then the embedding  
\[  \sigma\boxtimes D_{\sigma}^L(\omega) \hookrightarrow  \omega_{N_n}\hookrightarrow (\sigma \times \pi)_{N_n} \]
does not have the trivial supporting orbit.
\end{lemma}

\begin{proof}
Let $\kappa$ be the indecomposable component of $\pi_{N_{n(\sigma)}}$ which contains $\sigma \boxtimes D^L_{\sigma}(\pi)$ as a submodule. Suppose $\pi \boxtimes \sigma$ has the trivial supporting orbit via the embedding, and we shall derive a contradiction. Then by Lemma \ref{lem generic indecomp}, $\kappa \hookrightarrow \sigma\boxtimes \tau $. Thus, $\kappa \cong \sigma\boxtimes \tau'$ for some submodule $\tau'$ of $\tau$. This contradicts to Lemma \ref{lem big derivative compared with jacquet}.
\end{proof}

\section{Strong commutation between derivatives and integrals} \label{s strong commutation for der and int}

Recall that the notions of pre-commutativity and strong commutativity are defined in Definition \ref{def strong comm}. Section \ref{ss commutate triples} will prove the commutativity property and Section \ref{ss transitive property} will study compositions of strong commutativity. The composition is useful when $\sigma, \sigma'$ are $\square$-irreducible and $\sigma \times \sigma'$ is also $\square$-irreducible, and one would like to show the strong commutation involving $\sigma\times \sigma'$  from the strong commutation of $\sigma$ and $\sigma'$.

\subsection{Pre-commutativity and trivial supporting orbits}

Pre-commutativity and trivial supporting orbits are two related concepts. The pre-commutativity emphasizes on the commutativity between integrals and derivatives while the trivial supporting orbits emphasizes on the geometry in the geometric lemma.

\begin{lemma}
Let $\sigma_1 \in \mathrm{Irr}^{\square}(G_{n_1})$ and let $\sigma_2 \in \mathrm{Irr}^{\square}(G_{n_2})$. Let $\pi$ be a smooth representation of $G_{n}$. Then $(\sigma_1, \sigma_2, \pi)$ is a pre-RdLi-commutative triple if and only if the embedding 
\[ D_{\sigma_1}\circ I_{\sigma_2}(\pi)\boxtimes \sigma_1 \hookrightarrow (\sigma_2 \times \pi)_{N_{n_1}} 
\]
also has the trivial supporting orbit.
\end{lemma}

\begin{proof}
Let $P=P_{n_2, n}$, $Q=P_{n+n_2-n_1, n_1}$. Then the embedding can be rewritten as:
\[   D_{\sigma_1}\circ I_{\sigma_2}(\pi) \boxtimes \sigma_1 \hookrightarrow C^{\infty}(P\setminus G_{n+n_2}, \sigma_2\boxtimes \pi)_{N_{n_1}}
\]
The surjection $(\sigma_2 \times \pi)_{N_{n_1}}\twoheadrightarrow \sigma_2\dot{\times}^1 (\pi_{N_{n_1}})$ can be rewritten as the surjection from $(\mathrm{Ind}_P^G\sigma_2\boxtimes \pi)_{N_{n_1}}$ to $C^{\infty}(P\setminus PQ, \sigma_2\boxtimes \pi)_{N_{n_1}}\cong \pi \dot{\times}^1 \pi_{N_{n_1}}$, where $C^{\infty}(P\setminus PQ, \sigma_2\boxtimes \pi)$ is the space of locally constant functions from $P\setminus PQ$ to $\sigma_2\boxtimes \pi$. Then the composition of the embedding and the surjection coincides with the one in the definition of the pre-commutativity as well as the one in the definition of the supporting orbit. This implies the lemma.
\end{proof}

\subsection{Commutative triples} \label{ss commutate triples}

 The following proposition suggests the use of 'commutative' in our terminologies:

\begin{proposition} \label{prop strong commute imply commute}
Let $(\sigma', \sigma, \pi)$ be a strongly RdLi-commutative triple. Then 
\[   D^R_{\sigma'} \circ I^L_{\sigma}(\pi) \cong I^L_{\sigma}\circ D^R_{\sigma'}(\pi) .
\]
\end{proposition}

\begin{proof}
By the factoring through condition, we have that 
\[ D^R_{\sigma'}\circ I^L_{\sigma}(\pi) \boxtimes \sigma' \hookrightarrow \sigma \dot{\times}^1 (D_{\sigma'}(\pi)\boxtimes \sigma') =(\sigma \times D_{\sigma'}(\pi))\boxtimes \sigma'. \]
Hence, $D^R_{\sigma'}\circ I^L_{\sigma}(\pi) \hookrightarrow \sigma \times D^R_{\sigma'}(\pi)$ by K\"unneth formula. Thus, by the uniqueness of the submodule in $\sigma \times D^R_{\sigma'}(\pi)$, we have that:
\[   D^R_{\sigma'} \circ I^L_{\sigma}(\pi) \cong I^L_{\sigma}\circ D^R_{\sigma'}(\pi) .
\]
\end{proof}

The following two lemmas give examples of pre-commutativity by using the geometric lemma.

\begin{lemma} \label{lem pre commutative a preceding case}
Let $\Delta_1=[a_1, b_1]_{\rho}$ and $\Delta_2=[a_2, b_2]_{\rho}$ be two segments with $a_1>a_2$ or $b_1>b_2$. Then $(\mathrm{St}(\Delta_1), \mathrm{St}(\Delta_2), \pi)$ is pre-RdLi-commutative. 
\end{lemma}

\begin{proof}
We consider the case that $b_1>b_2$ and the another case is similar. Suppose not to derive a contradiction. We consider the embedding:
\[       I_{\Delta_2}(\pi) \hookrightarrow \mathrm{St}(\Delta_2) \times \pi .\]
Let $l=l_{abs}(\Delta_1)$. Then we have:
\[  D_{\Delta_1}\circ I_{\Delta_2}(\pi) \boxtimes \mathrm{St}(\Delta_1) \hookrightarrow (I_{\Delta_2}(\pi))_{N_l} \hookrightarrow (\mathrm{St}(\Delta_2)\times \pi)_N .
\]
Then one applies the geometric lemma on $(\mathrm{St}(\Delta_2) \times \pi)_{N_l}$ and it admits a filtration with subquotients taking the form:
\[ (*) \quad \mathrm{St}([b',b_2]_{\rho}) \times \tau_1 \boxtimes \mathrm{St}([a_2,b'-1]_{\rho}) \times \tau_2,\]
where $\tau_1 \boxtimes \tau_2$ is an irreducible composition factor in $\pi_{N'}$ for some suitable unipotent radical $N'$. Thus, $D_{\Delta_1}\circ I_{\Delta_2}(\pi) \boxtimes \mathrm{St}(\Delta_1)$ embeds to one of such subquotients. However, since $b_2<b_1$, applying Frobenius reciprocity on the second factor in (*) gives that $D_{\Delta_1}\circ I_{\Delta_2}(\pi) \boxtimes \mathrm{St}(\Delta_1)$ cannot embed to one of above layers. This gives a contradiction. Thus, $(\mathrm{St}(\Delta_1), \mathrm{St}(\Delta_2), \pi)$ is pre-RdLi-commutative triple.
\end{proof}

Following application of the geometric lemma as in the previous lemma, we have more examples on pre-RdLi-commutative triples:

\begin{example} \label{ex examples of st comm triples}
\begin{enumerate}
\item Let $\sigma_1, \sigma_2 \in \mathrm{Irr}^{\square}$. Suppose $\mathrm{csupp}(\sigma_1)\cap \mathrm{csupp}(\sigma_2)=\emptyset$. Then, for any $\pi \in \mathrm{Irr}$ with $D_{\sigma_1}(\pi)\neq 0$, $(\sigma_1, \sigma_2, \pi)$ is a pre-RdLi-commutative triple. 
\item Let $\sigma \in \mathrm{Irr}^{\square}$. Let $\Delta=[a,b]_{\rho}$ be a segment. Suppose $\nu_{\rho}^a\rho \notin \mathrm{csupp}(\sigma)$. Then, for any $\pi \in \mathrm{Irr}$, $(\sigma, \mathrm{St}(\Delta), \pi)$ is a pre-RdLi-commutative triple.
\item Let $\sigma \in \mathrm{Irr}^{\square}$. Let $\Delta=[a,b]_{\rho}$ be a segment. Suppose $\nu_{\rho}^b\rho \notin \mathrm{csupp}(\sigma)$. Then, for any $\pi \in \mathrm{Irr}$, $(\mathrm{St}(\Delta), \sigma, \pi)$ is a pre-RdLi-commutative triple.
\end{enumerate}
\end{example}

In general, commutativity does not imply strong commutativity and we have the following example:

\begin{example}
Let $\pi=\rho$ be an irreducible cuspidal representation. We have that $I_{\rho}\circ D_{\rho}(\pi)\cong D_{\rho}\circ I_{\rho}(\pi)\cong \rho$. However, $(\left\{ \rho \right\}, \left\{ \rho \right\}, \pi)$ is not a strongly RdLi-commutative triple. 
\end{example}

\subsection{Composition for strong commutations} \label{ss transitive property}

\begin{proposition} \label{prop strong commutation}
Let $\sigma_1', \sigma_2', \sigma \in \mathrm{Irr}^{\square}$. Let $\pi \in \mathrm{Irr}$. Suppose $D_{\sigma}(\pi)\neq 0$. Let $\omega=I_{\sigma_2'}\circ I_{\sigma_1'}(\pi)$. Let $l=n(\sigma)$. If $(\sigma, \sigma'_1, \pi)$ and $(\sigma, \sigma'_2, I_{\sigma_1'}(\pi))$ are strongly RdLi-commutative, then the composition of natural maps
\[    D_{\sigma}(\omega) \boxtimes \sigma \hookrightarrow \omega_{N_l}  \hookrightarrow (\sigma_2'\times \sigma_1'\times \pi)_{N_l} \twoheadrightarrow \sigma'_2 \dot{\times}^1 (\sigma'_1 \dot{\times}^1 (\pi_{N_l}))
\]
are non-zero, and factors through the natural embedding:
\[  (\sigma'_2 \times \sigma'_1 \times D_{\sigma}(\pi)) \boxtimes \sigma \hookrightarrow\sigma'_2 \dot{\times}^1( \sigma'_1 \dot{\times}^1 (\pi_{N_l}))
\]
\end{proposition}

\begin{proof}
We have that $I_{\sigma_2'}\circ I_{\sigma_1'}(\pi)\cong I_{\sigma_2'\times \sigma_1'}(\pi)$. We consider the following commutative diagram:
\[\xymatrix{ D_{\sigma}(\omega)\boxtimes \sigma \ar@{^{(}->}[r]^j \ar@{-->}[dr]^{f_1}  & \omega_{N_l} \ar@{^{(}->}[r]^{i_1} &  (\sigma_2'\times I_{\sigma_1'}(\pi))_{N_l} \ar@{^{(}->}[r]^{i_2} \ar@{->>}[d]^{q_1} & (\sigma_2'\times \sigma_1'\times \pi)_{N_l} \ar@{->>}[d]^{q_2}  \\
 & (\sigma_2' \times D_{\sigma}\circ I_{\sigma_1'}(\pi))\boxtimes \sigma \ar@{^{(}->}[r] \ar@{-->}[dr]^{f_2} & \sigma_2' \dot{\times}^1 I_{\sigma_1'}(\pi)_{N_l}  \ar@{^{(}->}[r] &  \sigma_2' \dot{\times}^1 (\sigma_1'\times\pi)_{N_l} \ar@{->>}[d]^{q_3} \\
&	& (\sigma_2'\times\sigma_1'\times D_{\sigma}(\pi))\boxtimes \sigma \ar@{^{(}->}[r]   &  \sigma_2'\dot{\times}^1 (\sigma_1' \dot{\times}^1 (\pi_{N_l}))  
},
\]
where the top vertical maps $q_1$ and $q_2$ are the surjections onto the toppest layer in the geometric lemma; $q_3$ is induced from the surjection from $(\sigma_1' \times \pi)_{N_l}$ to $\sigma_1'\dot{\times}^1(\pi_{N_l})$; $j$ is the unique embedding from the derivative; $i_1$ is the unique embedding from the integral of $I_{\sigma_2'}(I_{\sigma_1'}(\pi))$ and $i_2$ is induced from the unique embedding from $I_{\sigma_1'}(\pi)$ to $\sigma_1'\times \pi$. 

 The existence of $f_1$ and $f_2$ follows from strong commutativity. Then the composition $f_2\circ f_1$ gives the required map for factoring through the map in the proposition. 
\end{proof}

\begin{corollary}
We use the notations in Proposition \ref{prop strong commutation}. Suppose further that $\sigma_1' \times \sigma_2'$ is irreducible. Then $(\sigma, \sigma_1'\times \sigma_2', \pi)$ is also a strongly RdLi-commutative triple.
\end{corollary}

\begin{proof}
Note that $\sigma_2' \dot{\times}^1 (\sigma_1' \dot{\times}^1 \pi_{N_1}) \cong (\sigma_2' \times \sigma_1')\dot{\times}^1 (\pi_{N_1})$. Then the corollary follows from Proposition \ref{prop strong commutation} and the definition of a strongly RdLi-commutative triple.
\end{proof}

We also have the following dual version:

\begin{proposition} \label{prop strong commutation der}
Let $\sigma_1, \sigma_2, \sigma' \in \mathrm{Irr}^{\square}$. Let $\pi \in \mathrm{Irr}$. Let $n=n(\pi)$. Let $l_1=n(\sigma_1)$ and let $l_2=n(\sigma_2)$.  Suppose $D_{\sigma_1}(\pi)\neq 0$. If $(\sigma_1, \sigma', \pi)$ and $(\sigma_2, \sigma', D_{\sigma_1}(\pi))$ are strongly RdLi-commutative triples, then the composition of the following natural maps:
\[  D_{\sigma_2}\circ D_{\sigma_1}\circ I_{\sigma'}(\pi)\boxtimes \sigma_2\boxtimes \sigma_1 \hookrightarrow (\sigma' \times \pi)_{N_{n+n(\sigma')-l_1-l_2,l_2,l_1}} \twoheadrightarrow \sigma' \dot{\times}^1(\pi_{N_{n-l_1-l_2, l_2,l_1}}) 
\]
is non-zero, and factors through the natural map:
\[   (\sigma' \times D_{\sigma_2}\circ D_{\sigma_1}(\pi))\boxtimes \sigma_2\boxtimes \sigma_1 \hookrightarrow \sigma' \dot{\times}^1\pi_{N_{n-l_1-l_2, l_2,l_1}} .
\]
Here $\dot{\times}^1$ is defined in a similar manner to (\ref{eqn embedding 1}) to obtain a $G_{n+n(\sigma)-l_1-l_2}\times G_{l_2}\times G_{l_1}$-representation, and the natural maps are described in detail in (\ref{eqn comp derivative diagram}) below.

\end{proposition}

\begin{proof}
Let $\omega= D_{\sigma_2}\circ D_{\sigma_1}\circ I_{\sigma'}(\pi)$ and $\omega'=D_{\sigma_2}\circ D_{\sigma_1}(\pi)$. Let $N=N_{n+n(\sigma')-l_1-l_2, l_2, l_1}$ and let $N'=N_{n-l_1-l_2,l_2,l_1}$. This follows from the following commutative diagram:
\begin{align} \label{eqn comp derivative diagram}
 \xymatrix{ \omega\boxtimes \sigma_2\boxtimes \sigma_1 \ar@{^{(}->}[r]^{i_1} & D_{\sigma_1}\circ I_{\sigma'}(\pi)_{N_{l_2}}\boxtimes \sigma_1 \ar@{^{(}->}[r]^{i_2} \ar@{=}[d] \ar@{-->}[drr]^{f_1}& (\sigma' \times \pi)_N \ar@{->>}[r] & (\sigma' \dot{\times}^1 \pi_{N_{l_1}})_{N_{l_2}} \ar@{->>}[r]  & \sigma_1 \dot{\times}^1 \pi_{N'}  \\
        &  I_{\sigma'}\circ D_{\sigma_1}(\pi)_{N_{l_2}}\boxtimes \sigma_1 \ar@{^{(}->}[rr]^{i_3}  &    &  (\sigma' \times D_{\sigma_1}(\pi))_{N_{l_2}} \boxtimes \sigma_1 \ar@{->>}[r] \ar@{^{(}->}[u]^{j_1} \ar@{-->}[dr]^{f_2} & \sigma' \dot{\times}^1 (D_{\sigma_1}(\pi))_{N_{l_2}} \boxtimes \sigma_1 \ar@{^{(}->}[u]^{j_2} \\
		&  & & &	\sigma' \times \omega'\boxtimes \sigma_2\boxtimes \sigma_1	\ar@{^{(}->}[u]^{j_3}
}
\end{align}

The embedding $i_1$ comes from taking a derivative, the embedding $i_2$ is induced from the composition of 
\[    D_{\sigma_1}\circ I_{\sigma'}(\pi)\boxtimes \sigma_1 \hookrightarrow I_{\sigma'}(\pi)_{N_{l_1}} \hookrightarrow (\sigma' \times \pi)_{N_{l_1}}
\] 
and the embedding $i_3$ is induced from the embedding $I_{\sigma'}\circ D_{\sigma_1}(\pi) \hookrightarrow \sigma' \times D_{\sigma_1}(\pi)$. The maps $j_1$ and $j_2$ are induced from the natural embeddings $D_{\sigma_1}(\pi)\boxtimes \sigma_1 \hookrightarrow \pi_{N_{l_1}}$, and $j_3$ is induced from the natural embedding $\omega' \boxtimes \sigma_2 \hookrightarrow D_{\sigma_1}(\pi)_{N_{l_2}}$. The leftmost vertical equality follows from Proposition \ref{prop strong commute imply commute}

The surjections are the natural ones from the geometric lemma, with furthermore first taking the Jacquet functor $N_{n_1}$ and then taking the Jacquet functor $N_{n_2}$ on the first factors. 

Now the existence of $f_1$ comes from the strong commutativity for $(\sigma_1, \sigma', \pi)$ and the existence of $f_2$ comes from the strong commutativity for $(\sigma_2, \sigma', D_{\sigma_1}(\pi))$. The map $f_2\circ f_1\circ i$ gives the required map for the factoring through condition.
\end{proof}

\begin{corollary}
We use the notations in Proposition \ref{prop strong commutation der}. Suppose further that $\sigma_1 \times \sigma_2$ is also $\square$-irreducible. Then $(\sigma_1\times \sigma_2, \sigma', \pi)$ is also a strongly RdLi-commutative triple.
\end{corollary}





\subsection{Transitivity of pre-commutative triples}

We prove a weaker converse of Proposition \ref{prop strong commutation}. 
\begin{proposition} \label{prop commutative triple induct} 
Let $\sigma_2, \sigma_2' \in \mathrm{Irr}^{\square}$ such that $\sigma_2 \times \sigma_2'$ is still $\square$-irreducible. Let $\sigma_1 \in \mathrm{Irr}^{\square}$. Let $\pi \in \mathrm{Irr}$. Suppose $(\sigma_1, \sigma_2 \times \sigma_2', \pi)$ is pre-RdLi-commutative. Then $(\sigma_1, \sigma_2, I_{\sigma_2'}(\pi))$ is also a pre-RdLi-commutative triple. 

\end{proposition}

\begin{proof}

Let $n_1=n(\sigma_1)$. Let $\omega=I_{\sigma_2\times \sigma_2'}(\pi)$ and let $\omega'=I_{\sigma_2}(\pi)$. We have the following diagram
\[ \xymatrix{ D_{\sigma_1}(\omega)\boxtimes \sigma_1 \ar@{^{(}->}[r]^{i_1} &  \omega_{N_{n_1}} \ar@{^{(}->}[r]^{i_2} & (\sigma_2\times \omega')_{N_{n_1}} \ar@{^{(}->}[r]^{i_3} \ar@{->>}[d]^{q_1} & (\sigma_2 \times \sigma_2' \times \pi)_{N_{n_1}}  \ar@{->>}[d]^{q_2} \\ 
      &    &       \sigma_2\dot{\times}^1 (\omega'_{N_{n_1}}) \ar@{^{(}->}[r]^{i_4} & \sigma_2\dot{\times}^1 (\sigma_2' \times \pi)_{N_{n_1}}   \ar@{->>}[d]^{q_3} \\ 
			&   &  &  \sigma_2 \dot{\times}^1 (\sigma_2' \dot{\times}^1 (\pi_{N_{n_1}})) } ,
\]
where the top horizontal maps are induced from  natural embeddings arising from derivatives and integrals, and the right vertical maps from the geometric lemma. 

The diagram is commutative by the functoriality of the geometric lemma. By uniqueness of the embedding $\omega \hookrightarrow (\sigma_2\times \sigma_2')\times \pi$, $i_3\circ i_2$ agrees with the map induced from the unique embedding $\pi$ to $\sigma_2\times \sigma_2'\times \pi$. The strongly commutative triple $(\sigma_1, \sigma_2\times \sigma_2', \pi)$ implies that $q_3\circ q_2 \circ i_3\circ i_2 \circ i_1 \neq 0$. By the commutative diagram, we then have that $q_1 \circ i_2\circ i_1\neq 0$, giving pre-commutativity for $(\sigma_1, \sigma_2, \omega')$.
\end{proof}

\begin{proposition} \label{prop transtivity pre commute}
Let $\sigma_1, \sigma'_1 \in \mathrm{Irr}^{\square}$ such that $\sigma_1 \times \sigma_1'$ is still $\square$-irreducible. Let $\sigma_2 \in \mathrm{Irr}^{\square}$. Suppose $(\sigma_1\times \sigma_1', \sigma_2, \pi)$ is pre-RdLi-commutative. Then $(\sigma_1, \sigma_2, \pi)$ is also pre-RdLi-commutative.
\end{proposition}

\begin{proof}
Let $\tau=D_{\sigma_1\times \sigma_1'}\circ I_{\sigma_2}(\pi)$. Let $n_1=n(\sigma_1)$, $n_1'=n(\sigma_1')$, $n_2=n(\sigma_2)$ and $n=n(\pi)$.
We first have the following diagram:
\[  \xymatrix{ \tau \boxtimes \sigma_1' \boxtimes \sigma_1 \ar[r] \ar[rd]  & \tau \boxtimes (\sigma_1\times \sigma_1')_{N_{n_1,n_1'}} \ar[r] &  I_{\sigma_2}(\pi)_{N} \ar[r]^{q}  & (\sigma_2\times \pi)_{N} \ar[r] \ar[rd]^{q_1} &  \sigma_2 \dot{\times}^1 \pi_{N'}  \\      
  & D_{\sigma_1}\circ I_{\sigma_2}(\pi)_{N_{n_1}}\boxtimes \sigma_1 \ar[ru]_i &    &   &   (\sigma_2 \dot{\times}^1 (\pi_{N_{n_1'}}))_{N''}    \ar[u]_{q_2}    
 },
\]
where $N=N_{n_2+n-n_1-n_1',n_1,n_1'}$ in $G_{n+n_2}$, $N'=N_{n-n_1-n_1',n_1,n_1'}$ in $G_n$, and $N''=N_{n+n_2-n_1-n_1',n_1}$ in $G_{n+n_2-n_1'}$ regarded as a subgroup of $G_{n+n_2-n_1'}\times G_{n_1'}$ via embedding to the first factor. The maps $q, q_1, q_3$ are the projections from the geometric lemma, while other maps in the diagram come from natural embeddings arising form derivatives and integrals.

The left triangle in the diagram is commutative by Lemma \ref{lem unique embedding left right}, while the right triangle is commutative by the naturality of the geometric lemma.

Now the composition of the top maps is non-zero by the pre-commutativity of $(\sigma_1\times \sigma_1', \sigma_2, \pi)$. This, in particular, implies that $q_1\circ q \circ i \neq 0$. Since $q_1\circ q\circ i$ is obtained from taking the Jacquet functor (with respect to $N''$) on the required composition for the pre-RdLi-commutativity of $(\sigma_1 \times \sigma_1', \sigma_2, \pi)$, the required composition is also non-zero. In other words, $(\sigma_1', \sigma_2, \pi)$ is pre-RdLi-commutative.
\end{proof}

\subsection{Producing some strongly RdLi-commutative triples}

\begin{proposition} \label{thm pre implies strong}
\begin{enumerate}
\item Let $\sigma, \sigma' \in \mathrm{Irr}^{\square}$. Suppose $(\sigma, \pi)$ is a Rd-irreducible pair and $(\sigma, \sigma', \pi)$ is a pre-RdLi-commutative triple.  Then $(\sigma, \sigma', \pi)$ is a strongly RdLi-commutative triple.
\item Let $\sigma, \sigma' \in \mathrm{Irr}^{\square}$. Suppose $D_{\sigma}\circ I_{\sigma'}(\pi)$ does not embed to $(\sigma'\times \mathbb D_{\sigma}(\pi))/(\sigma' \times D_{\sigma}(\pi))$. (Here $\sigma' \times D_{\sigma}(\pi)$ is viewed as a submodule $\sigma' \times \mathbb D_{\sigma}(\pi)$ via inducing from the natural map from $D_{\sigma}(\pi)$ to $\mathbb D_{\sigma}(\pi)$.) If $(\sigma, \sigma', \pi)$ is a pre-RdLi-commutative triple, then $(\sigma, \sigma', \pi)$ is a strongly RdLi-commutative triple.
\end{enumerate}

\end{proposition}

\begin{proof}
We first prove (1). Since $(\sigma, \sigma', \pi)$ is a pre-RdLi-commutative triple, there is an embedding 
\[   D_{\sigma}(\pi)\boxtimes \sigma \hookrightarrow \sigma' \dot{\times}^1 \pi_{N_i},
\]
where $i=n(\sigma)$. On the other hand, $(\sigma, \pi)$ is a Rd-irreducible pair, and so 
\[  \pi_{N_i} = D_{\sigma}(\pi)\boxtimes \sigma +\tau .
\]
Moreover, by the uniqueness of submodule of Jacquet module under $\square$-irreducible representations, $\tau$ does not have a submodule of the form $\omega \boxtimes \sigma_2$ and hence
\[  \mathrm{Hom}_{G_s\times G_t}(D_{\sigma}(\pi)\boxtimes \sigma, \sigma' \dot{\times}^1 \tau) \cong \mathrm{Hom}_{G_s}(D_{\sigma}(\pi), \mathrm{Hom}_{G_t}(\sigma, \sigma' \dot{\times}^1 \tau))\cong \mathrm{Hom}(D_{\sigma}(\pi), \sigma'\times \mathbb D_{\sigma}(\tau))=0 ,
\]
where $s=n(\pi)-n(\sigma)$ and $t=n(\sigma)$. Here $\mathrm{Hom}_{G_t}(\sigma, \sigma'\dot{\times}^1\tau)$ is regarded as a $G_s$-representation in a similar manner to the big derivative in Section \ref{ss converse and big derivatives}. This implies that the above embedding factors through the map $(\sigma' \times D_{\sigma}(\pi))\boxtimes \sigma$ to $\sigma'\dot{\times}^1\pi_{N_i}$. This implies the strong commutativity for $(\sigma, \sigma', \pi)$.

For (2), by definition, we have the following commutative diagram:
\[ \xymatrix{   D_{\sigma}(\pi)\circ I_{\sigma'}(\pi) \boxtimes \sigma \ar@{^{(}->}[r] \ar@{^{(}->}[dr]^u \ar@{-->}[ddr] &  \sigma' \dot{\times}^1 \pi_{N_{n(\sigma)}} \\        &    \sigma' \times \mathbb D_{\sigma}(\pi) \boxtimes \sigma \ar[u] \\ 
           &    \sigma' \times D_{\sigma}(\pi) \boxtimes \sigma   \ar[u]^i }                                           
\]
It is straightforward that the assumption in (2) implies that the map $u$ factors through $i$.
\end{proof}

\begin{remark}
We have not discussed an integral version of an irreducible pair, but one possible way to define is the following. For $\pi \in \mathrm{Irr}$ and $\sigma \in \mathrm{Irr}^{\square}$, we say that $(\sigma, \pi)$ is a {\it Li-irreducible pair} if $(\sigma, I_{\sigma}(\pi))$ is a Ld-irreducible pair. We shall not use this anyway and shall not explore its properties further.
\end{remark}

\section{Dual strongly commutative triples} \label{s dual strong commut}

We now show a duality for strongly commutative triples switching left and right versions. Similar to previous situations, it will work better (at least technically) if we impose some conditions of irreducible pairs.

\subsection{Strongly commutative triples and irreducible pairs}

The following proposition is one important and interesting result that can switch left and right versions of irreducible pairs. A key idea of the proof is the commutative diagram (\ref{eqn commutative diagram}) below and one first figures out the orbit that contributes an embedding from top maps and then deduces the orbit from the bottom maps in order to apply Proposition \ref{prop trivial orbit and irred}.

\begin{proposition} \label{prop strong irr imply irr}
Let $\sigma_1, \sigma_2 \in \mathrm{Irr}^{\square}$. Let $\pi \in \mathrm{Irr}$. Suppose $(\sigma_1, \sigma_2, \pi)$ is a strongly RdLi-commutative triple and $(\sigma_1, \pi)$ is a Rd-irreducible pair. Then $(\sigma_1, I_{\sigma_2}(\pi))$ is also a Rd-irreducible pair.
\end{proposition}

\begin{proof}
\noindent
{\bf Step 1: Establish a commutative diagram.}
Let $\tau =D_{\sigma_1}\circ I_{\sigma_2}(\pi)$. Let $n_1=n(\sigma_1)$ and let $n_2=n(\sigma_2)$. By Lemma \ref{lem uniqueness of products} and Proposition \ref{prop strong commute imply commute}, the following diagram:
\begin{align} \label{eqn commutative diagram}
 \xymatrix{    &  \sigma_2 \times \pi \ar[rd] &  \\
       I_{\sigma_2}(\pi)\cong  I^R_{\sigma_1}( \tau) \ar[ru] \ar[rd] &                       &   \sigma_2 \times D_{\sigma_1}(\pi) \times \sigma_1  \\   
					           &  \tau \times \sigma_1 \ar[ru] &
 }
\end{align}
is commutative, where the left upward and downward maps are unique embeddings; the right downward map is induced from the embedding $\pi \hookrightarrow D_{\sigma_1}(\pi)\times \sigma_1$; and the right upward map is induced from the embedding $\tau \hookrightarrow \sigma_2 \times D_{\sigma_1}(\pi)$. \\


\noindent
{\bf Step 2: Show the triviality of the supporting orbit for the composition of top maps in (\ref{eqn commutative diagram}).}
Now we apply the Jacquet functor $N=N_{n_1}$ on the top maps, and obtain induced maps:
\begin{align} \label{eqn dual commut diagram 1}
 \xymatrix{ \tau \boxtimes \sigma_1 \ar@{^{(}->}[r]^{j} \ar[dr] &  I_{\sigma_2}(\pi)_{N_{n_1}} \ar@{^{(}->}[r]^{i_1} & (\sigma_2\times \pi)_{N_{n_1}} \ar@{^{(}->}[r]^{i_2} \ar@{->>}[d]^{q_1} & (\sigma_2 \times D_{\sigma_1}(\pi)\times \sigma_1)_{N_{n_1}} \ar@{->>}[d]^{q_2}  \\  
              & \sigma_2\dot{\times}^1(D_{\sigma_1}(\pi)\boxtimes \sigma_1) \ar@{^{(}->}[r]^{i_1'}     & \sigma_2 \dot{\times} (\pi_{N_{n_1}})  \ar@{->>}[rd]_{q'\circ  i_2'} \ar@{^{(}->}[r]^{i_2'} &   \sigma_2 \dot{\times}^1 (D_{\sigma_1}(\pi)\times \sigma_1)_{N_{n_1}} \ar@{->>}[d]^{q'}   \\ &  &    & \sigma_2\dot{\times}^1(D_{\sigma_1}(\pi)\boxtimes \sigma_1) },
\end{align}
where $j, i_1, i_2$ are the natural embeddings as before and $q_1, q_2, q'$ are quotient maps from the geometric lemma. (The left top trapezium commutes from strong commutations, and the square follows from the functoriality of the geometric lemma and definitions, and the triangle follows from the Rd-irreducibility and Proposition \ref{prop trivial orbit and irred}.)

 By the definition of strong commutation, we have that $q_1\circ i_1 \circ j \neq 0$. The irreducible pair of $(\sigma_1, \pi)$ also implies that $q'\circ i_2'\circ i_1'$ is an isomorphism by definition, and $(q' \circ i_2')\circ ( q_1\circ i_1)\neq 0$. Thus we have that the supporting orbit for the embedding $\tau \boxtimes \sigma_1 \hookrightarrow (\sigma_2 \times D_{\sigma_1}(\pi)\times \sigma_1)_N$ is trivial. \\

\noindent
{\bf Step 3: Deduce the Rd-irreducibility using bottom maps in the commutative diagram in (\ref{eqn commutative diagram}).}
Now the commutative diagram in the beginning implies that in the following diagram
\[
  \xymatrix{  \tau \boxtimes \sigma_1 \ar@{^{(}->}[r] &  I_{\sigma_2}(\pi)_{N_{n_1}} \ar@{^{(}->}[r]^k & ( \tau \times \sigma_1)_{N_{n_1}} \ar[r] \ar@{->>}[d]^l & (\sigma_2\times D_{\sigma_1}(\pi)\times \sigma_1)_{N_{n_1}} \ar@{->>}[d] \\
                           &  &  \tau\boxtimes \sigma_1 \ar@{^{(}->}[r] &  \sigma_2\dot{\times}^1( D_{\sigma_1}(\pi))\boxtimes \sigma_1) 
 },
\]
the composition of the top maps coincides with the composition of the top maps in (\ref{eqn dual commut diagram 1}). Thus, the composition of the top horizontal maps with the rightmost vertical map is nonzero, and so $l\circ k\neq 0$. In other words, the supporting orbit for the embedding $\tau \boxtimes \sigma_1 \hookrightarrow (\tau \times \sigma_1)_N$ is also trivial. By Proposition \ref{prop trivial orbit and irred}, $(\sigma_1, I_{\sigma_2}(\pi))$ is an Rd-irreducible pair.
\end{proof}

\subsection{Dual strongly commutative triples}

We now prove a duality on pre-commutativity. The key idea is that if we start from a pre-RdLi-commutative triple $(\sigma_1, \sigma_2, \pi)$ with the condition that $(\sigma_1, \pi)$ is a Rd-irreducible pair, then we can deduce the following commutative diagram from Lemma \ref{lem unique embedding left right}: for $n_1=n(\sigma_1)$, $n_2=n(\sigma_2)$ and $n=n(\pi)$, and let $\tau= D_{\sigma_1}\circ I_{\sigma_2}(\pi)$,
\[  \xymatrix{          & \tau_{N_{n-n_1}}\boxtimes \sigma_1  \ar[d]^{q_1} &           &   \\
         \sigma_2\boxtimes D_{\sigma_1}(\pi)\boxtimes \sigma_1 \ar[ur] \ar[dr] &      I_{\sigma_2}(\pi)_{N_{n_2,n-n_1,n_1}}  \ar[r]^i                                    &        ( \tau\times \sigma_1)_{N_{n_2,n-n_1,n_1}}   \\
				                                                       &   \sigma_2\boxtimes \pi_{N_{n_1}} \ar[u]_{q_2}     &    
					  }，
\]
where the maps are the natural maps from derivatives and integrals. Then one determines the supporting orbit arising from the embedding $i\circ q_1$ via the property of a Rd-irreducible pair and then deduce the trivial supporting orbit arising from $i \circ q_2$ via the commutative diagram.

\begin{proposition} \label{prop commutative triple}
Let $(\sigma_1, \sigma_2, \pi)$ be a pre-RdLi-commutative triple. Suppose $(\sigma_1, \pi)$ is a Rd-irreducible pair. Then $(\sigma_2, \sigma_1, I_{\sigma_2} \circ D_{\sigma_1}(\pi))$ is a pre-LdRi-commutative triple.
\end{proposition}

\begin{proof}

Let $n_1=n(\sigma_1)$, $n_2=n(\sigma_2)$ and $n=n(\pi)$. Let $\tau=I_{\sigma_2} \circ D_{\sigma_1}(\pi)$. By Proposition \ref{thm pre implies strong}, $(\sigma_1, \sigma_2, \pi)$ is a strongly RdLi-commutative triple. By Proposition \ref{prop strong irr imply irr}, $(I_{\sigma_2}(\pi), \sigma_1)$ is also a Ld-irreducible pair. By Proposition \ref{prop trivial orbit and irred}, the embedding 
\begin{align} \label{eqn embedding in dual pre}
  \tau \boxtimes \sigma_1 \hookrightarrow  I^R_{\sigma_1}(\tau)_{N_{n_1}} \hookrightarrow (\tau \times \sigma_1)_{N_{n_1}} 
\end{align}
again has the trivial supporting orbit. Let $m=n-n_1$. Since $P_{m+n_2, n_1}P_{n_2,m, n_1}=P_{m+n_2,n_1}P_{m+n_2,n_1}$, 
 the induced embedding in (\ref{eqn embedding in dual pre}),  via taking the Jacquet functor $N_m$ on the first factor of $G_{n(\tau)}\times G_{n_1}$,
\[   \sigma_2\boxtimes D^L_{\sigma_2}(\tau)\boxtimes \sigma_1 \hookrightarrow \tau_{N_m}\boxtimes \sigma_1 \hookrightarrow I_{\sigma_2}(\pi)_{N'} \hookrightarrow (\tau \times \sigma_1)_{N'},
\]
 still has the trivial supporting orbit. Here $N'=N_{n_2,m,n_1}$.

By Proposition \ref{prop strong commute imply commute}, $D_{\sigma_2}^L\circ I_{\sigma_1}^R(\tau)\cong \pi$. On the other hand, we can take the Jacquet functor $N_{n_2, m+n_1}$ first to have:
\begin{align} \label{eqn embedding for pre-dual}
  \sigma_2\boxtimes \pi \hookrightarrow I^R_{\sigma_1}(\tau)_{N_{m}} \hookrightarrow (\tau \times \sigma_1)_{N_{m}} .
\end{align}
Then, we further take the Jacquet functor $N_{m,n_1}$ on the second factor and we have an induced embedding:
\[  \sigma_2 \boxtimes D_{\sigma_1}(\pi)\boxtimes \sigma_1 \hookrightarrow \sigma_2 \boxtimes \pi_{N_{n_1}} \hookrightarrow I^R_{\sigma_1}(\tau)_{N'} \hookrightarrow (\tau \times \sigma_1)_{N'} .
\]
By the uniqueness in Lemma \ref{lem unique embedding left right}, this embedding agrees with the previous one. Hence, the embedding has the trivial supporting orbit, and the corresponding orbit takes the form $P_{n_2+m, n_1}P_{n_2,m,n_1}$. Since $P_{n_2+m, n_1}P_{n_2,m,n_1} \subset P_{n_2+m,n_1}P_{n_2,n}$, the supporting orbit for the embedding 
\[   \sigma_2 \boxtimes \pi \hookrightarrow I_{\sigma_2}(\pi)_{N_n} \rightarrow (\tau \times \sigma_1)_{N_n}
\]
takes the form $ P_{n_2+m,n_1}P_{n_2,n}$ since $\sigma_2\boxtimes D_{\sigma_1}(\pi)\boxtimes \sigma_1$ is a submodule of $\sigma_2\boxtimes \pi_{N_{n_1}}$, and so also has the trivial supporting orbit.
\end{proof}

\section{Constructing some pre-commutativity} \label{s construct precommut}

The main goal of this section is to prove Proposition \ref{prop completing pre comm}. We first illustrate some basic arguments in Lemma \ref{lem completing pre comm} and then extend to the full case of Proposition \ref{prop completing pre comm}. The key idea of proving Lemma \ref{lem completing pre comm} is to first use induction and some simple application on the geometric lemma to obtain control on the structure in Lemmas \ref{lem impossible trivial supp orbit} and \ref{lem trivial orbit of times steinberg two cases}, and one then incorporates with another induction to prove Proposition  \ref{prop completing pre comm}.

\subsection{Completing to a Rd-irreducible pair} \label{ss complete to Rd irr pair}

\begin{lemma} \label{lem non ld irr pair}
We fix a cuspidal representation $\rho$. Let $\mathfrak m$ be a multisegment such that for any segment $\Delta$ in $\mathfrak m$, $b(\Delta)=\rho$. Let $\Delta'$ be another segment such that $b(\Delta')=\rho$ and $\Delta' \subset \Delta$. If $\mathfrak m$ has more than one segment, then $(\mathrm{St}(\Delta'), \mathrm{St}(\mathfrak m))$ is not a Ld-irreducible pair.
\end{lemma}

\begin{proof}
Let $n=l_{abs}(\mathfrak m)$. We write all the segments in $\mathfrak m$ as $\Delta_1=[a_1,0]_{\rho}, \ldots, \Delta_r=[a_r,0]_{\rho}$. Relabel the segments in $\mathfrak m$ such that $a_r \leq a_{r-1} \leq \ldots \leq a_1$. Write $\Delta'=[a',0]_{\rho}$. 

The representation $(\mathrm{St}(\Delta_1) \times \ldots \times \mathrm{St}(\Delta_r))_{N_r}$ is a $G_{n-l}\times G_l$-representation and we only consider the summand with the cuspidal support same as $\mathrm{St}(\Delta)$ in the $G_l$-factor. Then one applies the geometric lemma on 
\[   (\mathrm{St}(\Delta_1) \times \ldots \times \mathrm{St}(\Delta_r))_{N_l} 
\]
and those layers that contribute to the support is of the form:
\[ (\mathrm{St}(\Delta_1)\times \ldots \times \mathrm{St}(\Delta_{i-1})\times \mathrm{St}(\Delta_i') \times \mathrm{St}(\Delta_{i+1})\times \ldots \times \mathrm{St}(\Delta_r))\boxtimes \mathrm{St}(\Delta'),
\]
where $\Delta_i'=[a_i, a'-1]_{\rho}$. For simplicity, let
\[ \omega_i =(\mathrm{St}(\Delta_1)\times \ldots \times \mathrm{St}(\Delta_{i-1})\times \mathrm{St}(\Delta_i') \times \mathrm{St}(\Delta_{i+1})\times \ldots \times \mathrm{St}(\Delta_r)) .
\]
Note that $\omega_1$ is irreducible and generic.

We have that $\omega_i$ has unique submodule which is isomorphic to $\omega_1$. On the other hand, there is only one indecomposable direct summand that contains $\omega_1 \boxtimes \mathrm{St}(\Delta')$. Hence, all those layers have to be in that indecomposable direct summand. Now by our assumption that $\mathfrak m$ has more than one segment (i.e. $i>1$), we must have more than one layer and so the direct summand is not irreducible. In other words, $(\mathrm{St}(\Delta'), \mathrm{St}(\mathfrak m))$ is not a Ld-irreducible pair.
\end{proof}


For a segment $\Delta=[a,b]_{\rho}$, set $b(\Delta)=\nu_{\rho}^b\rho$.

\begin{lemma} \label{lem impossible trivial supp orbit}
Let $\tau$ be a smooth representation of $G_t$. Let $\widetilde{\Delta}$ be a segment and write $\widetilde{\Delta}=[\widetilde{a}, \widetilde{b}]_{\rho}$. Suppose $\mathfrak m$ satisfies the following properties:
\begin{itemize}
\item $D_{\mathfrak m}(\tau)\neq 0$;
\item any segment in $\mathfrak m$ takes the form $[a, \widetilde{b}]_{\rho}$ for some $a\leq \widetilde{a}$;
\item $\mathfrak m$ contains more than one segment.
\end{itemize}
Suppose there exists an embedding:
\[  \mathrm{St}(\mathfrak m) \hookrightarrow \mathrm{St}(\widetilde{\Delta}) \times \tau .
\]
 Let $\Delta$ be a shortest segment in $\mathfrak m$. Then the induced embedding
\[  \mathrm{St}(\Delta)\boxtimes \mathrm{St}(\mathfrak m-\Delta) \hookrightarrow  \mathrm{St}(\mathfrak m)_N \hookrightarrow (\mathrm{St}(\widetilde{\Delta}) \times \tau)_N
\] 
cannot have the trivial supporting orbit. Here $N=N_{l_1, l_2}$ for $l_1=l_{abs}(\Delta)$ and $l_2=l_{abs}(\mathfrak m-\Delta)$.
 
\end{lemma}

\begin{proof}

 If the embedding has the trivial supporting orbit, then, by definition, we have:
\[    \mathrm{St}(\Delta)\boxtimes \mathrm{St}(\mathfrak m-\Delta) \hookrightarrow  \mathrm{St}(\widetilde{\Delta}) \times (\tau_{N'}) ,
\]
where $N'=N_{n(\tau)-l_{abs}(\mathfrak m), l_{abs}(\mathfrak m)}$. We further take the Jacquet functor on $\mathrm{St}(\Delta)$ to give
\[   \mathrm{St}(\widetilde{\Delta}) \boxtimes \mathrm{St}(\Delta\setminus \widetilde{\Delta}),
\]
where $\Delta\setminus \widetilde{\Delta}$ is the set-theoretic subtraction. Then the embedding
\[   \mathrm{St}(\widetilde{\Delta}) \boxtimes \mathrm{St}(\Delta\setminus \widetilde{\Delta}) \boxtimes \mathrm{St}(\mathfrak m-\Delta) \hookrightarrow (\mathrm{St}(\widetilde{\Delta}) \times \tau_{N'})_{N''},
\]
where $N''=N_{l_a(\widetilde{\Delta}\setminus \Delta)}$ is the Jacquet functor taking on the first factor, also has the trivial supporting orbit by Corollary \ref{example on trivial supporting orbit by irreducible}, and the supporting orbit takes the form $P_{a,b+c}P_{a,b,c}$ ($a=l_{abs}(\widetilde{\Delta})$, $b=l_{abs}(\Delta\setminus \widetilde{\Delta})$ and $c=l_{abs}(\mathfrak m-\Delta)$. Since the above embedding is obtained by taking a Jacquet functor on the following embedding
\begin{align} \label{eqn embedding in lem}
  \mathrm{St}(\widetilde{\Delta}) \boxtimes \mathrm{St}(\mathfrak m-\Delta+\Delta\setminus \widetilde{\Delta}) \hookrightarrow (\mathrm{St}(\widetilde{\Delta}) \times \tau)_N ,
\end{align}
we also have that (\ref{eqn embedding in lem}) also has the trivial supporting orbit and the supporting orbit takes the form $P_{a,b+c}P_{a,b+c}$. However, since $\mathfrak m$ contains more than one segment, Lemma \ref{lem non ld irr pair} implies that $(\mathrm{St}(\widetilde{\Delta}), \mathrm{St}(\mathfrak m))$ is not a Ld-irreducible pair. This contradicts to Lemma \ref{lem indecomp trivial embedd}.
\end{proof}

\begin{lemma} \label{lem trivial orbit of times steinberg two cases}
Fix a cuspidal representation $\rho$ and an integer $b$. Let $\mathfrak m$ be a multisegment such that any segment in $\mathfrak m$ takes the form $[a,b]_{\rho}$ for some integer $a$. Let $\widetilde{\Delta}=[\widetilde{a}, \widetilde{b}]_{\rho}$ for some $\widetilde{b} \geq b$ and $\widetilde{b} \geq \widetilde{a} \geq b$. Let $l=l_{abs}(\mathfrak m)$. Let $\pi \in \mathrm{Irr}(G_n)$. If $\mathrm{St}(\mathfrak m)$ embeds to $\mathrm{St}(\widetilde{\Delta})\times \pi$, then the supporting orbit for the embedding:
\[    D_{\mathfrak m}(\pi) \boxtimes \mathrm{St}(\mathfrak m) \hookrightarrow  (\mathrm{St}(\widetilde{\Delta})\times \pi)_{N_l} 
\]
can be either one of the followings:
\begin{itemize}
\item the supporting orbit is trivial; or
\item the supporting orbit is determined by the minimal representative $w$ in $(S_s \times S_n)\setminus S_{s+n}/ (S_{s+n-l}\times S_l)$ given by:
\[  w^{-1}(p)=p \quad \mbox{ for $p=1, \ldots, s-s'$ } 
\]
\[  w^{-1}(p)=(s+n-l)+(p-(s-s'))\quad \mbox{ for $p=s-s'+1, \ldots, s$ } .
\]
and, furthermore, 
$D_{\mathfrak m}(\pi) \boxtimes \mathrm{St}(\mathfrak m) \hookrightarrow \mathrm{St}([\widetilde{a}, b]_{\rho}) \dot{\times}^2(\mathrm{St}([b+1, \widetilde{b}]_{\rho}) \dot{\times}^1 \pi_{N_{l-l'}})$, where $s=l_{abs}([\widetilde{a},b]_{\rho})$ and $s'=l_{abs}([b+1, \widetilde{b}]_{\rho})$.
\end{itemize}
\end{lemma}

\begin{proof}
We apply the geometric lemma on $(\mathrm{St}(\widetilde{\Delta})\times \pi)_{N_l}$ and the layers take the form:
\[   \mathrm{St}(\Delta_i) \dot{\times}^2(\mathrm{St}(\overline{\Delta}_i) \dot{\times}^1 \pi_{N_{l_i}}) ,
\]
where $\Delta_i=[\widetilde{a} ,i]_{\rho}$, $\overline{\Delta}_i=[i+1, \widetilde{b}]_{\rho}$ and $l_i=l-l_{abs}(\overline{\Delta}_i)$. We shall denote such layer by $\lambda_i$. 

Then, we must have that 
\[  D_{\mathfrak m}(\pi) \boxtimes \mathrm{St}(\mathfrak m) \hookrightarrow \lambda_{i^*}
\]
for some $i^*$. By comparing cuspidal support on the term $\mathrm{St}(\mathfrak m)$, we must then have that $\widetilde{a} \leq i^* \leq b$. It remains to show that when $i^* =b$. To this end, let $\tau=\mathrm{St}(\overline{\Delta}_i) \dot{\times}^1 \pi_{N_{l_i}}$ and we apply Frobenius reciprocity on
\[  \mathrm{Hom}(D_{\mathfrak m}(\pi)\boxtimes \mathrm{St}(\mathfrak m)_{N_{l-s_i}},   \mathrm{St}(\Delta_i) \dot{\times}^2 \tau )\cong \mathrm{Hom}_{G_{n-l}\times G_{s_i}\times G_{l_i}}(D_{\mathfrak m}(\pi) \boxtimes \mathrm{St}(\mathfrak m)_{N_{l-s_i}}, (\mathrm{St}(\Delta_i)\boxtimes \tau)^{\phi} ),
\]
where $s_i=l_{abs}(\overline{\Delta}_i)=l-l_i$, and $\phi$ is a natural twist that switches from $G_{s_i} \times G_{n-l}\times G_{l_i}$-representations to 
$G_{n-l}\times G_{s_i}\times G_{l_i}$-representations. 

Now, one studies the composition factor of $\mathrm{St}(\mathfrak m)_{N_{l-s_i}}$. A standard argument of using the geometric lemma and the Jacquet functors on generalized Steinberg modules, one has that any composition factor of $\mathrm{St}(\mathfrak m)_{N_{l-s_i}}$ has the form $\tau_1\boxtimes \tau_2$ with $\nu_{\rho}^b\rho \in \mathrm{supp}(\tau_1)$. This forces that $\nu_{\rho}^b \rho \in \overline{\Delta}_{i^*}$ and so, with $i^*\leq b$, we can only have $i^*=b$, as desired. 
\end{proof}

In the following proposition, we show that the pre-commutativity for an essentially square-integrable representation gives that of a larger generic representation, and we may first note that an easier case below is when both $\Delta=\widetilde{\Delta}$ (in the notation of Proposition \ref{prop completing pre comm}) are singletons. The proof for the general case requires deeper structure in Lemma \ref{lem impossible trivial supp orbit}.

\begin{lemma} \label{lem completing pre comm}
Let $\Delta=[a,b]_{\rho}, \widetilde{\Delta}=[\widetilde{a},\widetilde{b}]_{\rho}$ be segments. Let $\mathfrak m$ be a multisegment such that any segment in $\mathfrak m$ takes the form $[a', b]_{\rho}$ for some $a'\geq a$. Let $\pi \in \mathrm{Irr}$. If $(\mathrm{St}(\Delta), \mathrm{St}(\widetilde{\Delta}), \pi)$ is a pre-RdLi-commutative triple and $D_{\mathfrak m}(\pi)\neq 0$, then $(\mathrm{St}(\mathfrak m), \mathrm{St}(\widetilde{\Delta}), \pi)$ is a pre-RdLi-commutative triple.

\end{lemma}

\begin{proof}
We write the segments $\Delta_i=[a_i,b]_{\rho}$ in $\mathfrak m$. We arrange the segments in $\mathfrak m$ such that:
\[   a_r\leq \ldots \leq a_1.
\]
Let $\mathfrak m_k=\left\{ \Delta_1, \ldots, \Delta_k \right\}$ and let $\omega=I_{\widetilde{\Delta}}(\pi)$. We shall inductively show that $(\mathrm{St}(\mathfrak m_k), \mathrm{St}(\widetilde{\Delta}), \pi)$ is a pre-RdLi-commutative triple. When $k=1$, it follows from the given hypothesis. We now assume $k \geq 2$. 




Suppose $(\mathrm{St}(\mathfrak m_{k+1}), \mathrm{St}(\widetilde{\Delta}), \pi)$ is not a pre-RdLi-commutative triple to derive a contradiction. Let $l_1'=l_{abs}([b+1,\widetilde{b}]_{\rho})$, $l_2=l_{abs}(\mathfrak m_k)$ and $l_3=l_{abs}(\mathfrak m_{k+1})$. Then, by the second bullet of Lemma \ref{lem trivial orbit of times steinberg two cases}, we have an embedding of the form:
\begin{align}\label{eqn embedding tech 0}
   D_{\mathfrak m_{k+1}}(\omega) \boxtimes \mathrm{St}(\mathfrak m_{k+1}) \hookrightarrow \mathrm{St}([\widetilde{a}, b]_{\rho})\dot{\times}^2 \tau ,
\end{align}
where $\tau =\mathrm{St}([b+1, \widetilde{b}]_{\rho}) \dot{\times}^1 (\pi_N)$, where $N=N_{l_3-l_1'}$ is the unipotent radical in $G_{n(\pi)}$.

Now $\mathrm{St}(\Delta_{k+1})\boxtimes \mathrm{St}(\mathfrak m_k) \hookrightarrow \mathrm{St}(\mathfrak m_{k+1})_{N_{l_2}}$ will give an embedding:
\begin{align} \label{eqn embedding tech 1}
  D_{\mathfrak m_{k+1}}(\omega) \boxtimes \mathrm{St}(\Delta_{k+1}) \boxtimes \mathrm{St}(\mathfrak m_k) \hookrightarrow (\mathrm{St}([\widetilde{a}, b]_{\rho})\dot{\times}^2\tau)_{N'}
\end{align}
By Lemma \ref{eqn embedding tech 1}, there is another way to obtain the embedding via:
\begin{align} \label{eqn embedding tech 2}
  D_{\mathfrak m_k}(\omega)\boxtimes \mathrm{St}(\mathfrak m_k) \hookrightarrow (\mathrm{St}(\widetilde{\Delta})\times \pi)_{N_{l_2}} 
\end{align}
and $D_{\mathfrak m_{k+1}}(\omega)\boxtimes \mathrm{St}(\Delta_{k+1})\hookrightarrow D_{\mathfrak m_k}(\omega)_{N_{l_2}}$.

The inductive hypothesis implies (\ref{eqn embedding tech 2}) has the trivial supporting orbit. Now if we simply consider (\ref{eqn embedding tech 0}) as a map 
\[    \mathrm{St}(\mathfrak m_{k+1}) \hookrightarrow (\mathrm{St}([\widetilde{a}, b]_{\rho})\dot{\times}^2\tau)_{N'} 
\]
(via any natural embedding from $\mathrm{St}(\mathfrak m_{k+1})$ to $\mathrm{St}(\Delta_{k+1})\boxtimes \mathrm{St}(\mathfrak m_k)$), Proposition  \ref{prop composition supp orbit} in Appendix implies that the embedding:
\begin{align} \label{eqn embedding tech 3}
    \mathrm{St}(\Delta_{k+1}) \boxtimes \mathrm{St}(\mathfrak m_k) \hookrightarrow (\mathrm{St}([\widetilde{a}, b]_{\rho})\dot{\times}^2\tau)_{N_{l_2}},
\end{align}
where $N_{l_2}$ is regarded as a subgroup of the second factor, also has the trivial supporting orbit.

However, it follows from Lemma \ref{lem impossible trivial supp orbit} that (\ref{eqn embedding tech 3}) cannot have the trivial supporting orbit and so we arrive a contradiction.
\end{proof}

\begin{proposition} \label{prop completing pre comm}
We use the notations in Lemma \ref{lem completing pre comm}. Let $\mathfrak n$ be a multisegment with all segments $\Delta'$ satisfying $a(\Delta')=a(\widetilde{\Delta})$ and $\Delta'\subset \widetilde{\Delta}$. Let $\mathfrak m$ be a $\Delta$-saturated multisegment with $D_{\mathfrak m}(\pi)\neq 0$. If $(\mathrm{St}(\Delta), \mathrm{St}(\mathfrak n), \pi)$ is a pre-RdLi-commutative triple, then $(\mathrm{St}(\mathfrak m), \mathrm{St}(\mathfrak n), \pi)$ is a pre-RdLi-commutative triple.
\end{proposition}

\begin{proof}
The proof is similar to Lemma \ref{lem completing pre comm}. We explain necessary modifications. 

Let $\omega=I_{\mathfrak n}(\pi)$. We write the segments in $\mathfrak n$ as $\Delta_1', \ldots, \Delta_s'$. Let $\mathfrak n_i=\left\{ \Delta_1', \ldots, \Delta_i' \right\}$ and let $\bar{\mathfrak n}_i=\mathfrak n-\mathfrak n_i$. We write the segments in $\mathfrak m$ as $\Delta_1=\Delta, \Delta_2, \ldots, \Delta_r$. We shall prove inductively on $r$ that $(\mathrm{St}(\mathfrak n_i), \mathrm{St}(\mathfrak m), \pi)$ is a pre-RdLi-commutative triple. 

Let $\tau_i=\mathrm{St}(\bar{\mathfrak n}_i)\times \pi$. We shall prove by an induction on $r$ that the embedding
\[  (*)\quad    D_{\mathfrak m}(\omega)\boxtimes \mathrm{St}(\mathfrak m) \hookrightarrow \omega_{N_l} \hookrightarrow (\mathrm{St}(\mathfrak n_i)\times \tau_i)_{N_l} 
\]
has the trivial supporting orbit. For $r=1$, the inductive hypothesis implies that the supporting orbit for the embedding:
\[         D_{\mathfrak m}(\omega)\boxtimes \mathrm{St}(\mathfrak m) \hookrightarrow \omega_{N_l} \hookrightarrow         (\mathrm{St}(\mathfrak n) \times \pi)_{N_l} 
\]
is trivial and so the corresponding orbit (in the sense of Section \ref{ss supporting orbit}) takes the form $P_{p,n}P_{p+n-l,l}$, where $p=l_{abs}(\mathfrak n)$, $n=n(\pi)$ and $l=l_{abs}(\mathfrak m)$. Since $P_{p,n}P_{p+n-l,l}=P_{p_i,\bar{p}_i,n}P_{p_i+n-l,l}$ (where $p_i=l_{abs}(\mathfrak n_i)$ and $\bar{p}_i=p-p_i=l_{abs}(\bar{\mathfrak n}_i)$), the embedding
\[ D_{\mathfrak m}(\omega)\boxtimes \mathrm{St}(\mathfrak m) \hookrightarrow \omega_{N_l} \hookrightarrow         (\mathrm{St}(\mathfrak n_i) \times \mathrm{St}(\bar{\mathfrak n}_i) \times \pi)_{N_l} 
\]
has the trivial supporting orbit and the corresponding orbit takes the form $P_{p_i,\bar{p}_i,n}P_{p_i+n-l,l}$. Now, since $P_{p_i,\bar{p}_i,n}P_{p_i+n-l,l} \subset P_{p_i,\bar{p}_i+n}P_{p_i+n-l,l}$, the embedding (*) also has the trivial supporting orbit. This proves the case of $r=1$.

We now consider $r \geq 2$. We have
\begin{align*} 
   \pi \hookrightarrow \mathrm{St}(\mathfrak n_i)\times \tau_i \cong \mathrm{St}(\mathfrak n_{i-1})\times \mathrm{St}(\Delta_i)\times \tau_i =\mathrm{St}(\mathfrak n_{i-1})\times \tau_{i-1} .
\end{align*}

\noindent
{\it Claim 1:} The embedding:
\[   D_{\mathfrak m}(\omega) \boxtimes \mathrm{St}(\mathfrak m) \hookrightarrow \omega_{N_l}   \hookrightarrow (\mathrm{St}(\mathfrak n_{i-1})\times \mathrm{St}(\Delta_i)\times \tau_{i})_{N_l}
\]
has the trivial supporting orbit. Here we regard the parabolic induction is from $G_a\times G_b\times G_c$ to $G_{a+b+c}$, where $a=l_{abs}(\mathfrak n_{i-1})$, $b=l_{abs}(\Delta_i)$ and $c=n(\tau_{i})$. \\

Since $P_{a,b,c}P_{a+b+c-l,l}$ is a closed subspace in $P_{a+b,c}P_{a+b+c-l,l}$, the claim indeed implies the embedding 
\[  D_{\mathfrak m}(\omega) \boxtimes \mathrm{St}(\mathfrak m) \hookrightarrow \omega_{N_l}   \hookrightarrow (\mathrm{St}(\mathfrak n_i)\times \tau_i)_{N_l} 
\]
has the trivial supporting orbit, in which we regard the parabolic induction is from $G_{a+b}\times G_c$ to $G_{a+b+c}$. \\

It remains to prove Claim 1. Before that, we first prove another useful claim: \\
\noindent
{\it Claim 2:} Let $\mathfrak m'=\left\{ \Delta_1, \ldots, \Delta_{r-1}\right\}$ and let $l'=l_{abs}(\mathfrak m')$. The embedding:
\[  D_{\mathfrak m'}(\omega) \boxtimes \mathrm{St}(\mathfrak m')  \hookrightarrow \omega_{N_{l'}} \hookrightarrow (\mathrm{St}(\mathfrak n_{i-1})\times \mathrm{St}(\Delta_i') \times  \tau_{i})_{N_{l'}}
\]
has the trivial supporting orbit and the orbit takes the form $P_{a,b,c}P_{a+b+c-l',l'}$. \\

\noindent
{\it Proof Claim 2:} By induction hypothesis on $r$, we have that the embedding:
\[   D_{\mathfrak m'}(\omega) \boxtimes \mathrm{St}(\mathfrak m') \hookrightarrow \omega_{N_{l'}} \hookrightarrow (\mathrm{St}(\mathfrak n_{i})\times \tau_i)_{N_{l'}} ,
\]
has the trivial supporting orbit and the orbit takes the form $P_{a+b,c}P_{a+b+c-l',l'}$. Since $P_{a,b,c}P_{a+b+c-l',l'} \subset P_{a+b,c}P_{a+b+c-l',l'}$, we then obtain the claim. \\

Now we go to prove Claim 1. \\
{\it Proof of Claim 1:} We shall now prove by induction on $i$. When $i=0$, there is nothing to prove. Now assume $i \geq 1$. We shall consider the composition

\begin{align} \label{eqn composition of maps in inductive condition}
   D_{\mathfrak m}(\omega) \boxtimes \mathrm{St}(\mathfrak m)  \hookrightarrow \omega_{N_l}   \hookrightarrow (\mathrm{St}(\mathfrak n_{i-1})\times \tau_{i-1})_{N_l} \twoheadrightarrow \mathrm{St}(\mathfrak n_{i-1})\dot{\times}^1 (\tau_{i-1})_{N_l}
\end{align}
which is non-zero by induction on $i$. Here the maps again are the natural maps from derivatives, integrals and the geometric lemma. 

Now we write $(\tau_{i-1})_{N_l}=(\mathrm{St}(\Delta_i')\times \tau_{i})_{N_l}$. Now we have two possibilities:

\noindent
{\bf Case 1:} Suppose the embedding 
\[     D_{\mathfrak m}(\omega) \boxtimes \mathrm{St}(\mathfrak m)  \hookrightarrow \omega_{N_l}   \hookrightarrow (\mathrm{St}(\mathfrak n_{i-1})\times \tau_{i-1})_{N_l} \twoheadrightarrow \mathrm{St}(\mathfrak n_{i-1})\dot{\times}^1 (\tau_{i-1})_{N_l} \twoheadrightarrow \mathrm{St}(\mathfrak n_{i-1})\dot{\times}^1 (\mathrm{St}(\Delta_i)\dot{\times}^1 (\tau_i)_{N_l})
\]
is non-zero. In other words, the embedding has the trivial supporting orbit and the orbit takes the form $P_{a,b,c}P_{a+b+c-l,l}$. Again, using $P_{a,b,c}P_{a+b+c-l,l}\subset P_{a+b,c}P_{a+b+c-l,l}$, we obtain Claim 1. \\

\noindent
{\bf Case 2:} Suppose the embedding in Claim 1 does not hold. Then, by Lemma \ref{lem trivial orbit of times steinberg two cases}, 
\begin{align} \label{eqn embedding hard case}
  D_{\mathfrak m}(\omega) \boxtimes \mathrm{St}(\mathfrak m) \hookrightarrow  \mathrm{St}(\mathfrak n_{i-1})\dot{\times}^1(\mathrm{St}([\widetilde{a},b]_{\rho})\dot{\times}^2 \kappa ),
\end{align}
where $\kappa=\mathrm{St}(\Delta_i-[\widetilde{a}, b]_{\rho}) \dot{\times}^1 \tau_{N_s}$ for $s=l-l_{abas}([\widetilde{a},b]_{\rho})$.

Set $\widetilde{\kappa}=\mathrm{St}(\mathfrak n_{i-1}) \dot{\times}^1\kappa$. Note that $\mathrm{St}([\widetilde{a},b]_{\rho})\dot{\times}^2 \widetilde{\kappa}$. We now regard the embedding (\ref{eqn embedding hard case}) as $G_l$-representations via the second factor in $G_{a+b+c-l}\times G_l$. Now Claim 2 (with Proposition \ref{prop composition supp orbit} in the Appendix) implies that the embedding
\[  \mathrm{St}(\Delta_r)\boxtimes \mathrm{St}(\mathfrak m')\hookrightarrow (\mathrm{St}([\widetilde{a},b]_{\rho})\dot{\times}^2 \widetilde{\kappa})_{N_{l'}}, 
\]
where $N_{l'}$ is in $G_{l}$, has trivial supporting orbit. But this then contradicts to Lemma \ref{lem impossible trivial supp orbit}. Hence, this case is not possible.
\end{proof}

\section{Pre-commutativity $\Rightarrow$ strong commutativity in square-integrable case} \label{s pre implies strong in sq}

Recall that for an essentially square-integrable representation $\sigma$, $\sigma \cong \mathrm{St}(\Delta)$ for some segment $\Delta$. For segments $\Delta$ and $\Delta'$, and $\pi \in \mathrm{Irr}$, we sometimes also say that $(\Delta, \Delta', \pi)$ is pre-RdLi-commutative (resp. strongly RdLi-commutative) if $(\mathrm{St}(\Delta), \mathrm{St}(\Delta'), \pi)$ is pre-RdLi-commutative (resp. strongly RdLi-commutative). Similar terminologies are also used for LdRi-versions.

The main part of this section is to study the effect of derivatives and integrals on $\eta_{\Delta}$-invariants. Such properties will be used to verify the conditions in Proposition \ref{thm pre implies strong}(2) to prove Theorem \ref{thm pre imply strong}.

\subsection{$\eta$-invariant under strong commutation}

\begin{lemma} \label{lem right multi submulti}
Let $\pi \in \mathrm{Irr}$. Let $\sigma \in \mathrm{Irr}^{\square}$. For any segment $\Delta$, $\mathfrak{mx}_{\Delta}(\pi)$ is a submultisegment of $\mathfrak{mx}_{\Delta}(I_{\sigma}(\pi))$, equivalently
\[   \eta_{\Delta}(\pi) \leq \eta_{\Delta}(I_{\sigma}(\pi)).
\]
\end{lemma}

\begin{proof}
 Let $\mathfrak m=\mathfrak{mx}_{\Delta}(\pi)$. We have the embedding:
\[   \pi \hookrightarrow D_{\mathfrak m}(\pi) \times \mathrm{St}(\mathfrak m) .
\]
Then 
\[  I_{\sigma}(\pi) \hookrightarrow \sigma \times \pi \hookrightarrow \sigma \times D_{\mathfrak m}(\pi) \times \mathrm{St}(\mathfrak m) .
\]
Thus $D_{\mathfrak m}(I_{\sigma}(\pi))\neq 0$. This implies the lemma.
\end{proof}

\begin{lemma} \label{lem eta unchange}
Let $\Delta_1, \Delta_2$ be segments. Let $\mathfrak n$ be a $\Delta_2$-saturated multisegment. Suppose $(\mathrm{St}(\Delta_1), \mathrm{St}(\mathfrak n), \pi)$ is a pre-RdLi-commutative triple. Then $\eta_{\Delta_1}(\pi)=\eta_{\Delta_1}(I_{\mathfrak n}(\pi))$. In particular, for any $\Delta_1$-saturated segment $\Delta$, if $(\Delta_1, \Delta_2, \pi)$ is a pre-RdLi-commutative triple, then $\eta_{\Delta}(\pi)=\eta_{\Delta}(I_{\mathfrak n}(\pi))$. 
\end{lemma}

\begin{proof}
Let $\mathfrak m=\mathfrak{mx}_{\Delta_1}(\pi)$. By Lemma \ref{lem completing pre comm}, $(\mathrm{St}(\mathfrak m), \mathrm{St}(\mathfrak n), \pi)$ is pre-RdLi-commutative and so, by Proposition \ref{thm pre implies strong}(1), is strongly commutative. By Proposition \ref{prop strong irr imply irr}, $(\mathrm{St}(\mathfrak m), I_{\mathfrak n}(\pi))$ is still Rd-irreducible. Hence, by Proposition \ref{prop mx rd irr pair}, $\mathfrak{mx}_{\Delta_1}(I_{\mathfrak n}(\pi))=\mathfrak m$, as desired.
\end{proof}

We now prove a preliminary commutativity result:

\begin{lemma} \label{lem commutate under irred pair}
Let $(\Delta_1, \Delta_2, \pi)$ be a pre-RdLi-commutative triple. Let $\mathfrak m=\mathfrak{mx}_{\Delta_1}(\pi)$. Then
\[   D_{\mathfrak m}\circ I_{\Delta_2}(\pi) \cong I_{\Delta_2}\circ D_{\mathfrak m}(\pi) .
\]
\end{lemma}

\begin{proof}
By Lemma \ref{lem completing pre comm}, $(\mathrm{St}(\mathfrak m), \mathrm{St}(\Delta_2), \pi)$ is pre-RdLi-commutative and so, by Proposition \ref{thm pre implies strong}(1), is strongly RdLi-commutative. Now the required commutativity follows from Proposition \ref{prop strong commute imply commute}.
\end{proof}

We now prove a dual version of Lemma \ref{lem eta unchange}. Thanks to a dual theory of Proposition \ref{prop commutative triple}, the proof is simpler and we do not have to prove a dual version of Proposition \ref{prop completing pre comm}:

\begin{lemma} \label{lem pre commute 1}
Let $(\Delta_1, \Delta_2, \pi)$ be a pre-RdLi-commutative triple. Let $\mathfrak m=\mathfrak{mx}_{ \Delta_1}(\pi)$. Then 
\[  \eta^L_{\Delta_2}( I_{\Delta_2} \circ D_{\mathfrak m}(\pi))=\eta^L_{\Delta_2}(I_{\Delta_2}(\pi))
\]
\end{lemma}

\begin{proof}
Since $(\Delta_1, \Delta_2, \pi)$ is pre-RdLi-commutative, Proposition \ref{prop completing pre comm} implies that $(\mathrm{St}(\mathfrak m), \mathrm{St}(\Delta_2), \pi)$ is pre-RdLi-commutative. Then, by Proposition \ref{prop commutative triple}, $(\mathrm{St}(\Delta_2), \mathrm{St}(\mathfrak m), I_{\Delta_2}\circ D_{\mathfrak m}(\pi))$ is a pre-LdRi-commutative triple. Now, by using the right version of Lemma \ref{lem commutate under irred pair}, we have
\[   \eta_{\Delta_2}^R(I^L_{\mathfrak m}\circ I_{\Delta_2}\circ D_{\mathfrak m}(\pi))=\eta_{\Delta_2}^R(I_{\Delta_2}(\pi)) .
\]
Combining with the right version of Lemma \ref{lem eta unchange}, we obtain the desired statement.
\end{proof}

\begin{lemma} \label{lem eta for  another side}
Let $(\Delta_1, \Delta_2, \pi)$ be a pre-RdLi-commutative triple. Let $\mathfrak m=\mathfrak{mx}_{\Delta_1}(\pi)$. Then $\eta^L_{\Delta_2}(D_{\mathfrak m}\circ I_{\Delta_2}(\pi))=\eta^L_{\Delta_2}(D_{\Delta_1}\circ I_{\Delta_2}(\pi))=\eta^L_{\Delta_2}(I_{\Delta_2}(\pi))$. 
\end{lemma}

\begin{proof}

Now, let $\tau=D_{\mathfrak m}\circ I_{\Delta_2}(\pi)$. Then 
\[  D_{\Delta_1}\circ I_{\Delta_2}(\pi) \hookrightarrow \tau \times \mathrm{St}(\mathfrak m-\Delta_1) ,
\]
and so the right version of Lemma \ref{lem right multi submulti} gives that 
\[  \eta_{\Delta_2}(D_{\Delta_1}\circ I_{\Delta_2}(\pi)) \geq \eta_{\Delta_2}(\tau) .
\]
Now, using $I_{\Delta_2}(\pi) \hookrightarrow (D_{\Delta_1}\circ I_{\Delta_2}(\pi))\times \mathrm{St}(\Delta_1)$ and the right version of Lemma \ref{lem right multi submulti} again, we have:
\[  \eta_{\Delta_2}(I_{\Delta_2}(\pi)) \geq \eta_{\Delta_2}(D_{\Delta_1}\circ I_{\Delta_2}(\pi) .
  \]
	By Lemmas \ref{lem commutate under irred pair} and \ref{lem pre commute 1},
\[    \eta^L_{\Delta_2}(D_{\mathfrak m}\circ I_{\Delta_2} (\pi))=\eta^L_{\Delta_2}(I_{\Delta_2}(\pi)) .
\]
and so the above inequalities are equalities as desired.
\end{proof}


\subsection{Some more properties of $\eta_{\Delta}$}

\begin{lemma}  \label{lem composition factor big derivative}
Let $\Delta$ be a segment and let $\pi \in \mathrm{Irr}$. Let $\tau$ be a composition factor in the Jordan-H\"older series of $\mathbb D_{\Delta}(\pi)$ with $\tau \not\cong D_{\Delta}(\pi)$. Then $\mathfrak{mx}_{\Delta}(\tau) \neq \mathfrak{mx}_{\Delta}(D_{\Delta}(\pi))=\mathfrak{mx}_{\Delta}(\pi)-\Delta$.
\end{lemma}

\begin{proof}
 Let $\mathfrak p=\mathfrak{mx}_{\Delta}(\pi)$. Then 
\[   \pi \hookrightarrow D_{\mathfrak p}(\pi) \times \mathrm{St}(\mathfrak p) .
\]
Write $\Delta=[a,b]_{\rho}$. Since $\mathfrak{mx}_{\Delta}(D_{\mathfrak p}(\pi))=\emptyset$, the $G_{k}$-factor (for any $k$) of any simple composition factor in $D_{\mathfrak p}(\pi)_{N_k}$ cannot have the cuspidal support of the form $\left\{ \nu^c \rho, \ldots, \nu^b\rho \right\}$ for some $a\leq c\leq b$. Since 
\[   D_{\Delta}(\pi)\boxtimes \mathrm{St}(\Delta) \hookrightarrow \pi_{N_l} \hookrightarrow (D_{\mathfrak p}(\pi)\times \mathrm{St}(\mathfrak p))_{N_l},
\]
where $l=l_{abs}(\Delta)$, the only layer in the geometric lemma of $(D_{\mathfrak p}(\pi)\times \mathrm{St}(\mathfrak p))_{N_l}$ that can contribute the above embedding is the toppest one i.e. $\mathbb D_{\mathfrak p}(\pi) \times (\mathrm{St}(\mathfrak p)_{N_l})$. Hence, $\mathbb D_{\Delta}(\pi)$ embeds to $D_{\mathfrak p}(\pi) \times \mathbb D_{\Delta}(\mathrm{St}(\mathfrak p))$. We also have that $\mathbb D_{\Delta}(\mathrm{St}(\mathfrak p)) \cong \mathrm{St}(\mathfrak p-\Delta)$ by \cite[Lemma 11.6]{Ch22+b}. Thus, we have:
\[   \mathbb D_{\Delta}(\pi) \hookrightarrow D_{\mathfrak p}(\pi)\times \mathrm{St}(\mathfrak p-\Delta) .
\]
Let $l'=l_{abs}(\mathfrak p-\Delta)$. Then we also have 
\[   \mathbb D_{\Delta}(\pi)_{N_{l'}} \hookrightarrow (D_{\mathfrak p}(\pi) \times \mathrm{St}(\mathfrak p-\Delta))_{N_{l'}} .
\]
Now we consider the component with the cuspidal support same as $\mathrm{csupp}(\mathrm{St}(\mathfrak p-\Delta))$ in the second factor of $G_{n-l-l'}\times G_{l'}$ ($n=n(\pi)$). We apply the geometric lemma and a standard computation shows that such component is of the form $D_{\mathfrak p}(\pi)\boxtimes \mathrm{St}(\mathfrak p-\Delta)$ with multiplicity one. Such form must come from $D_{\Delta}(\pi)_{N_{l'}}$ and so for other composition factor in $\tau$ of $\mathbb D_{\Delta}(\pi)$, $\tau_{N_{l'}}$ cannot have the factor of the form $\tau' \boxtimes \mathrm{St}(\mathfrak p-\Delta)$ for some $\tau' \in \mathrm{Irr}$. This implies the lemma.
\end{proof}

We also need a variant. 

\begin{lemma} \label{lem decrease mx after derivative}
We use the notations in Lemma \ref{lem composition factor big derivative}. Then 
\[  |\mathfrak{mx}_{\Delta}(\tau)| \leq |\mathfrak{mx}_{\Delta}(\pi)|-1 .
\]
\end{lemma}

\begin{proof}
Again write $\Delta=[a,b]_{\rho}$. We have shown in Lemma \ref{lem composition factor big derivative} that 
\[  \tau \hookrightarrow D_{\mathfrak p}(\pi) \times \mathrm{St}(\mathfrak p-\Delta) .\]
Let $\mathfrak q=\mathfrak{mx}_{\Delta}(\tau)$. Let $\tau_1 \boxtimes \tau_2$ be any composition factor in $D_{\Delta}(\pi)_{N_k}$. By using $\mathfrak{mx}_{\Delta}(D_{\mathfrak p}(\pi))=\emptyset$, if $\mathrm{csupp}(\tau_2)$ contains only cuspidal representations $\nu^a_{\rho}\rho, \ldots, \nu^{b}_{\rho}\rho$ (with certain multiplicities), then $\mathrm{csupp}(\tau_2)$ contains only cuspidal representations $\nu^a_{\rho}\rho, \ldots, \nu^b_{\rho}\rho$ (with certain multiplicities). 

Let $\tau_1' \boxtimes \tau_2'$ be a composition factor in $\mathrm{St}(\mathfrak p-\Delta)_{N_{k'}}$ for some $k'$. Then the cuspidal representation $\nu^b_{\rho}\rho$ can have at most multiplicity $|\mathfrak{mx}_{\Delta}(\pi)|-1$ in $\mathrm{csupp}(\tau_2')$. Thus, if the cuspidal representations in $\mathrm{csupp}(\tau_2 \times \tau_2')$ lie in $\nu^a_{\rho}\rho, \ldots, \nu^b_{\rho}\rho$, then $\nu_{\rho}^b\rho$ can appear with multiplicity at most $|\mathfrak{mx}_{\Delta}(\pi)|-1$ in $\mathrm{csupp}(\tau_2\times \tau_2')$. This proves the lemma.
\end{proof}

For convenience, for a segment $\Delta=[a,b]_{\rho}$, set $a(\Delta)=\nu_{\rho}^a\rho$.

\begin{lemma} \label{lem inequaltiy on left derivatives}
We use the notations in Lemma \ref{lem composition factor big derivative}. Let $\Delta'$ be any segment. Then $|\mathfrak{mx}^L_{\Delta'}(\tau)| \leq |\mathfrak{mx}^L_{\Delta'}(\pi)|$.
\end{lemma}

\begin{proof}
Let $\mathfrak q=\mathfrak{mx}^L_{\Delta'}(\pi)$. Then 
\[   \pi \hookrightarrow \mathrm{St}(\mathfrak q) \times D^L_{\mathfrak q}(\pi) .
\]
Now, we have:
\[   \mathbb D_{\Delta'}(\pi)\boxtimes \mathrm{St}(\Delta') \hookrightarrow \pi_{N_l} \hookrightarrow (\mathrm{St}(\mathfrak q) \times D^L_{\mathfrak q}(\pi))_{N_l}, 
\]
where $l=l_{abs}(\Delta')$. Then $\tau \boxtimes \mathrm{St}(\Delta')$ is still a simple composition factor in $\pi_{N_l}$. Then, by the geometric lemma on  $(\mathrm{St}(\mathfrak q) \times D^L_{\mathfrak q}(\pi))_{N_l}$, $\tau$ is a composition factor of $\tau_1 \times \tau_2$ for some $\tau_1\in \mathrm{Irr}$ and $\tau_2\in \mathrm{Irr}$ such that $\tau_1\boxtimes \tau_1'$ is a simple composition factor in $\mathrm{St}(\mathfrak q)_{N_k}$ for some $k$ and some $\tau_1' \in \mathrm{Irr}$ and $\tau_2 \boxtimes \tau_2'$ is a simple composition factor in $D^L_{\mathfrak q}(\pi)_{N_r}$ for some $r$ and some $\tau_2' \in \mathrm{Irr}$. The composition factors in $\mathrm{St}(\mathfrak q)_{N_k}$ can be computed from the geometric lemma again and so it is a straightforward computation to give $|\mathfrak{mx}^L_{\Delta'}(\tau_1)|\leq |\mathfrak q|$. Since $\mathfrak{mx}^L_{\Delta'}(\pi)=\emptyset$, $\mathfrak{mx}^L_{\Delta'}(\tau_2)=\emptyset$. Thus, we have that $\tau$ is a composition factor of $\mathrm{St}(\mathfrak q')\times \omega$ for some $\omega \in \mathrm{Irr}$ with $\mathfrak{mx}^L_{\Delta'}(\omega)=\emptyset$ and a left $\widetilde{\Delta}$-saturated multisegment $\mathfrak q'$ with $|\mathfrak q'|\leq |\mathfrak q|$. (Here a left $\widetilde{\Delta}$-satuarted multisegment $\mathfrak m$ is a multisegment with all segments $\widetilde{\Delta}$ in $\mathfrak m$ satisfy $\widetilde{\Delta}\subset \Delta$ and $a(\widetilde{\Delta})=a(\Delta)$.) The remaining proof is precisely the left version of the arguments in the proof of Lemma \ref{lem decrease mx after derivative}.
\end{proof}

\begin{lemma}
Let $\pi \in \mathrm{Irr}$. Let $\Delta$ be a segment. Let $\mathfrak p=\mathfrak{mx}^L_{\Delta}(\pi)$. Let $\tau_1\boxtimes \tau_2$ be an irreducible composition factor in $\pi_{N_k}$ (for some $k$). Then 
\[   |\mathfrak{mx}^L_{\Delta}(\tau_1)| \leq |\mathfrak{mx}^L_{\Delta}(\pi)| .
\]
\end{lemma}

\begin{proof}
Suppose $|\mathfrak{mx}^L_{\Delta}(\tau_1)|>|\mathfrak{mx}^L_{\Delta}(\pi)|$ to derive a contradiction. Let $\mathfrak q=\mathfrak{mx}^L_{\Delta}(\tau_1)$. Then $\mathrm{St}(\mathfrak q)\boxtimes \tau'$ is a composition factor in $(\tau_1)_{N_l}$ for some $l$ and some $\tau' \in \mathrm{Irr}$. Hence, $\mathrm{St}(\mathfrak q)\boxtimes \tau' \boxtimes \tau_2$ appears in a Jacquet functor of $\pi$. By comparing cuspidal support, the submodule takes the form $\omega \boxtimes \tau_1' \boxtimes \tau_2'$, where 
\[  \mathrm{csupp}(\omega)=\mathrm{csupp}(\mathrm{St}(\mathfrak q)), \quad \mathrm{csupp}(\tau_1)=\mathrm{csupp}(\tau_1'),\quad \mathrm{csupp}(\tau_2)=\mathrm{csupp}(\tau_2') . \]
Recall that any irreducible representation, particularly $\omega$, can be realized as a submodule of the dual of a standard module. This then gives that 
\[  \omega \hookrightarrow \mathrm{St}(\mathfrak q')\times \omega'
\]
for some left $\Delta$-saturated multisegment $\mathfrak q'$ and some irreducible module $\omega'$ with $a(\Delta)\notin \mathrm{csupp}(\omega')$. In particular, we also have $|\mathfrak q'|>|\mathfrak{q}|$. 

Now, by Frobenius reciprocity, we have $\mathrm{St}(\mathfrak q')\boxtimes \kappa$ is a simple submoulde of a Jacquet module of $\pi$ for some $\kappa \in \mathrm{Irr}$. This implies that $D_{\mathfrak q'}^L(\pi)\neq 0$. However, as we have shown that $|\mathfrak q'|>|\mathfrak q|$, this gives a contradiction to the maximality of $\mathfrak{mx}_{\Delta}^L(\pi)$, as desired.
\end{proof}

\subsection{Strong commutativity}

We now prove pre-commutativity implies strong commutativity for essentially square-integrable representations. 

\begin{theorem} \label{thm pre imply strong}
Let $\sigma_1, \sigma_2 \in \mathrm{Irr}$ be both essentially square-integrable. Let $\pi \in \mathrm{Irr}$. If $(\sigma_1, \sigma_2, \pi)$ is pre-RdLi-commutative, then $(\sigma_1, \sigma_2, \pi)$ is strongly RdLi-commutative. 
\end{theorem}

\begin{proof}
Write $\sigma_1=\mathrm{St}(\Delta_1)$ and $\sigma_2=\mathrm{St}(\Delta_2)$ for some segments $\Delta_1$ and $\Delta_2$. Let $\mathfrak p=\mathfrak{mx}_{\Delta_1}(\pi)$ and let $\widetilde{\sigma}=\mathrm{St}(\mathfrak p-\Delta_1)$. We also write $\Delta_1=[a_1,b_1]_{\rho}$ and $\Delta_2=[a_2,b_2]_{\rho}$. (If $\Delta_1$ and $\Delta_2$ are not in the same cuspidal line, the case is easy and we shall omit that.) Then $(\mathrm{St}(\mathfrak p), \pi)$ is a Rd-irreducible pair (Proposition \ref{prop mx rd irr pair}). By Proposition \ref{prop completing pre comm}, $( \mathrm{St}(\mathfrak p), \mathrm{St}(\Delta_2), \pi)$ is pre-RdLi-commutative. Hence, by Proposition \ref{thm pre implies strong}(1), $(\mathrm{St}(\mathfrak p), \mathrm{St}(\Delta_2),\pi)$ is a strongly RdLi-commutative triple. 

We now check the conditions in Proposition \ref{thm pre implies strong}(2). Recall that $D_{\Delta_1}(\pi)$ has only multiplicity one in $\mathbb D_{\Delta_1}(\pi)$ \cite[Proposition 11.5]{Ch22+b}. Hence, it suffices to show that $D_{\Delta_1}(I_{\Delta_2}(\pi))$ cannot embed to $\mathrm{St}(\Delta_2) \times \tau$ for any simple composition factor $\tau \not\cong D_{\Delta_1}(I_{\Delta_2}(\pi))$ of $\mathbb D_{\Delta_1}(\pi)$. Suppose not to obtain a contradiction. We then have an embedding:
\[  (*) \quad   D_{\Delta_1}(I_{\Delta_2}(\pi)) \hookrightarrow \mathrm{St}(\Delta_2) \times \tau .
\]
By Lemma \ref{lem decrease mx after derivative}, $|\mathfrak{mx}_{\Delta_1}(\tau)|\leq |\mathfrak{mx}_{\Delta_1}(\pi)|-1$ .

We now consider several cases: \\

{\bf Case 1:} $|\mathfrak{mx}_{ \Delta_1}(\tau)|=|\mathfrak{mx}_{\Delta_1}(\pi)|-1$.  Using (*) and $\mathrm{St}(\Delta_2) \times \tau \hookrightarrow \mathrm{St}(\Delta_2) \times D_{\mathfrak{mx}_{\Delta_1}(\tau)}( \tau) \times \mathrm{St}(\mathfrak{mx}_{\Delta_1}(\tau))$, we have 
\[   \mathfrak{mx}_{\Delta_1}(\tau) \subset \mathfrak{mx}_{\Delta_1}(D_{\Delta_1}(I_{\Delta_2}(\pi))) .
\]
But further looking at the number of segments in those $\mathfrak{mx}_{\Delta_1}$ from Lemmas \ref{lem eta unchange} and \ref{lem composition factor big derivative}, it must be an equality, giving a contradiction to Lemma \ref{lem composition factor big derivative}.

{\bf Case 2:} $|\mathfrak{mx}_{\Delta_1}(\tau)|<|\mathfrak{mx}_{\Delta_1}(\pi)|-2$. Again (*) implies that
\[  |\mathfrak{mx}_{\Delta_1}(D_{\Delta_1}(I_{\Delta_2}(\pi)))|\leq |\mathfrak{mx}_{\Delta_1}(\tau)|+|\mathfrak{mx}_{\Delta_1}(\mathrm{St}(\Delta_2))| ,
\]
but this is impossible from Lemmas \ref{lem eta unchange} and \ref{lem composition factor big derivative}.

{\bf Case 3:} $|\mathfrak{mx}_{\Delta_1}(\tau)|=|\mathfrak{mx}_{\Delta_1}(\pi)|-2$. Let $\mathfrak p=\mathfrak{mx}_{\Delta_1}(D_{\Delta_1}(\pi))$ and let $\mathfrak q=\mathfrak{mx}_{\Delta_1}(\tau)$. Let $l=l_{abs}(\mathfrak p)$. 

In such case, using (*), we have:
\[   D_{\mathfrak p}\circ D_{\Delta_1}(I_{\Delta_2}(\pi))\boxtimes \mathrm{St}(\mathfrak p) \hookrightarrow (\mathrm{St}(\Delta_2) \times \tau)_{N_{l}}  .
\]
Using the geometric lemma and comparing the number of segments in those $\mathfrak{mx}_{\Delta_1}$, the only possible layer that can contribute to the above embedding takes the form:
\[   (\mathrm{St}(\Delta') \times \tau_1) \boxtimes (\mathrm{St}(\Delta'')\times \tau_2) ,
\]
where $\Delta'=[b_1+1, b_2]_{\rho}$ and $\Delta''=[a_1,b_1]_{\rho}$, and $\tau_1\boxtimes \tau_2$ is a simple composition factor in $\tau_{N_{p}}$ ($p=l_{abs}(\mathfrak p)$). Then 
\begin{align} \label{eqn embedding under geo lem}
\quad   D_{\mathfrak p}\circ D_{\Delta_1}(I_{\Delta_2}(\pi)) \hookrightarrow \mathrm{St}(\Delta') \times \tau_1 
\end{align}
By Lemma \ref{lem inequaltiy on left derivatives}, we also have
\begin{align} \label{eqn first bound by pi}
|\mathfrak{mx}^L_{\Delta_2}(\tau)| \leq  |\mathfrak{mx}^L_{\Delta_2}(\pi)|
\end{align}
Note that, by using the commutativity of taking Jacquet functors on left and right, one deduces that 
\[ |\mathfrak{mx}^L_{\Delta_2}(\tau_1)| \leq |\mathfrak{mx}^L_{\Delta_2}(\tau)| 
\]
and so combines with (\ref{eqn first bound by pi})
\[  |\mathfrak{mx}^L_{\Delta_2}(\tau_1)|  \leq  |\mathfrak{mx}^L_{\Delta_2}(\pi)|.
\]
With $\mathfrak{mx}^L_{\Delta_2}(\mathrm{St}(\Delta'))=\emptyset$, (\ref{eqn embedding under geo lem}) implies that 
\[   |\mathfrak{mx}^L_{\Delta_2}(D_{\mathfrak p}\circ D_{\Delta_1}(I_{\Delta_2}(\pi)))| \leq  |\mathfrak{mx}^L_{\Delta_2}(\tau)| ,
\]
and so combining with (\ref{eqn first bound by pi}), we have
\[   |\mathfrak{mx}^L_{\Delta_2}(D_{\mathfrak p}\circ D_{\Delta_1}(I_{\Delta_2}(\pi)))| \leq |\mathfrak{mx}^L_{\Delta_2}(\pi)| .
\]
However, by Lemma \ref{lem eta for  another side}, 
\[ |\mathfrak{mx}_{\Delta_2}^L(D_{\mathfrak p}\circ D_{\Delta_1}(I_{\Delta_2}(\pi)))|=|\mathfrak{mx}_{\Delta_2}^L(I_{\Delta_2}(\pi))|=|\mathfrak{mx}_{\Delta_2}^L(\pi)|+1 , \]
giving a contradiction.

We have checked the conditions in Proposition \ref{thm pre implies strong}(2), and it follows from that proposition that $(\sigma_1, \sigma_2, \pi)$ is also strongly RdLi-commutative.
\end{proof}

\begin{corollary} \label{cor saturated commut triple}
Let  $(\mathrm{St}(\Delta_1), \mathrm{St}(\Delta_2), \pi)$ be a RdLi-pre-commutative triple. Let $\Delta'$ be a $\Delta_1$-saturated segment with $D_{\Delta'}(\pi)\neq 0$. Then $(\mathrm{St}(\Delta'), \mathrm{St}(\Delta_2) , \pi)$ is also a strongly RdLi-commutative triple.
\end{corollary}

\begin{proof}
Let $\sigma_1=\mathrm{St}(\Delta')$. Let $\mathfrak p=\mathfrak{mx}_{\Delta_1}(\pi)$ and let $\widetilde{\sigma}=\mathrm{St}(\mathfrak p-\Delta')$. The remaining follows the same argument as in the proof of Theorem \ref{thm pre imply strong}.
\end{proof}

\begin{remark}
Let $\Delta=[a,b]_{\rho}$ and $\Delta'=[a',b]_{\rho}$. Suppose $(\mathrm{St}(\Delta), \sigma_2, \pi)$ is a  pre-RdLi-commutative triple. If $a'\leq a$, the above lemma implies that $(\mathrm{St}(\Delta'), \sigma_2, \pi)$ is also a  pre-RdLi-commutative triple. However, it is not true in general if $a'>a$. For example, take $\pi=\mathrm{St}([0,1])\times \langle [0,2]\rangle$. Let $\sigma=\sigma'=\mathrm{St}([2])$. Let $\widetilde{\sigma}=\mathrm{St}([0,2])$. Then $(\sigma, \sigma', \pi)$ is a pre-RdLi-commutative triple. However, $(\widetilde{\sigma}, \sigma', \pi)$ is not a pre-RdLi-commutative triple.

\end{remark}

\subsection{Dual formulation for essentially square-integrable representations}

\begin{corollary} \label{cor dual strong commutative triple}
Let $\sigma_1, \sigma_2 \in \mathrm{Irr}$ be essentially square-integrable. Let $\pi \in \mathrm{Irr}$. Then $(\sigma_1, \sigma_2, \pi)$ is a strongly RdLi-commutative triple if and only if $(\sigma_2, \sigma_1, I_{\sigma_2}\circ D_{\sigma_1}(\pi))$ is a strongly LdRi-commutative triple.
\end{corollary}

\begin{proof}
We only prove the only if direction and a proof for the if direction is similar. By Proposition \ref{prop mx rd irr pair}, $\sigma' \in \mathrm{Irr}^{\square}$ such that $(\sigma_1\times \sigma', \pi)$ is Rd-irreducible pair (and $\sigma_1\times \sigma'$ is irreducible) and $(\sigma_1 \times \sigma', \sigma_2, \pi)$ is still a strongly RdLi-commutative triple by Propositions \ref{prop completing pre comm} and \ref{thm pre implies strong}(1). Thus, by Proposition \ref{prop commutative triple} (the version that switches between left and right), $(\sigma_2, \sigma_1 \times \sigma', I^R_{\sigma_2}\circ D_{\sigma_1 \times \sigma'}(\pi))$ is a pre-LdRi-commutative triple. Thus, by Proposition \ref{prop commutative triple induct}, $(\sigma_2, \sigma_1, I_{\sigma'}\circ I_{\sigma_2} \circ D_{\sigma' \times \sigma_1}(\pi))$ is a pre-LdRi-commutative triple. By Proposition \ref{prop strong commute imply commute} and $D_{\sigma_1\times \sigma'}=D_{\sigma'}\circ D_{\sigma'}$, we have that $ I^R_{\sigma'}\circ I_{\sigma_2} \circ D_{\sigma' \times \sigma_1}(\pi)\cong D_{\sigma_1}\circ I_{\sigma_2}(\pi)$ and so $(\sigma_2, \sigma_1,  D_{\sigma_1}\circ I_{\sigma_2}(\pi))$ is a pre-LdRi-commutative triple. Hence, by the LdRi-version of Theorem \ref{thm pre imply strong}, $(\sigma_2, \sigma_1, D_{\sigma_1}\circ I_{\sigma_2} (\pi))$ is a strongly LdRi-commutative triple. By Proposition \ref{prop strong commute imply commute}, we of course also have $D_{\sigma_1}\circ I_{\sigma_2}(\pi)\cong I_{\sigma_2}\circ D_{\sigma_1}(\pi)$ and this gives the corollary 
\end{proof}
 
Based on the square-integrable case, we raise the following question:

\begin{conjecture}
Corollary \ref{cor dual strong commutative triple} holds for any $\sigma_1, \sigma_2 \in \mathrm{Irr}^{\square}$.
\end{conjecture}

\section{Combinatorial commutations} \label{s combinatorial commutation}

\subsection{Combinatorially commutative triples}

Recall that strong commutation is defined in Definition \ref{def combinatorial comm triple}. We also have a dual formulation:

\begin{definition} \label{def dual com commutative}
Let $\Delta_1, \Delta_2$ be segments. Let $\pi \in \mathrm{Irr}$. We say that $(\Delta_1, \Delta_2, \pi)$ is a {\it dual combinatorially RdLi-commutative triple} if 
\[ \eta_{\Delta_2}^L(D_{\Delta_1}\circ I_{\Delta_2}(\pi))=\eta_{\Delta_2}^L(I_{\Delta_2}(\pi)). \]
\end{definition}



\begin{remark} \label{rmk on dual commut}
\begin{enumerate}
\item It is not a correct formulation (c.f. Theorem \ref{thm combinatorial def}) if one changes the condition in Definition \ref{def dual com commutative} to $\eta^L_{\Delta_2}(D_{\Delta_1}(\pi)) =\eta^L_{\Delta_2}(\pi)$ (while it is not far away, see the second remark). For example, if one takes $\pi=\langle [-1,0]\rangle$ and $\Delta_1=[0]$ and $\Delta_2=[0,1]$, then $(\Delta_1, \Delta_2, \pi)$ is not strongly RdLi-commutative, but we still have $\eta^L_{\Delta_2}(D_{\Delta_1}(\pi))=\eta^L_{\Delta_2}(\pi)=0$. 
\item On the other hand, it follows from Theorem \ref{thm combinatorial def} below that if $(\Delta_1, \Delta_2, \pi)$ is a dual strongly RdLi-commutative triple, then $D_{\Delta_1}\circ I_{\Delta_2}(\pi)\cong I_{\Delta_2}\circ D_{\Delta_1}(\pi)$ (Proposition \ref{prop strong commute imply commute}). Then, we also have
\[    \eta_{\Delta_2}^L(D_{\Delta_1}(\pi))=\eta_{\Delta_2}^L(\pi) .
\] 
\item Suppose $(\Delta_1, \Delta_2, \pi)$ is strongly RdLi-commutative. Note that if $D_{\Delta_1}^2(\pi)\neq 0$, then $(\Delta_1, \Delta_2, D_{\Delta_1}(\pi))$ is also strongly RdLi-commutative. However, the analogue is not quite right for the induction one i.e. $(\Delta_1, \Delta_2, I_{\Delta_2}(\pi))$ is not necessarily strongly RdLi-commutative.
\item One can also define a notion of combinatorially LdRi-commuative triple as follows. A triple $(\Delta_1, \Delta_2, \pi)$ is said to be combinatorially LdRi-commutative if $\eta^L_{\Delta_1}(I^R_{\Delta_2}(\pi))=\eta^L_{\Delta_1}(\pi)$. It follows from definitions that $(\Delta_1, \Delta_2, \pi)$ is dual RdLi-commutative if and only if $(\Delta_2, \Delta_1, D_{\Delta_1}\circ I_{\Delta_2}(\pi))$ is strongly LdRi-commutative. 
\end{enumerate}
\end{remark}



\subsection{Equivalent definitions}

\begin{theorem} \label{thm combinatorial def}
Let $\Delta_1, \Delta_2$ be segments. Let $\pi \in \mathrm{Irr}$. Then the followings are equivalent:
\begin{enumerate}
\item $(\mathrm{St}(\Delta_1), \mathrm{St}(\Delta_2), \pi)$ is a strongly  RdLi-commutative triple;
\item $(\mathrm{St}(\Delta_2), \mathrm{St}(\Delta_1), I_{\Delta_2}\circ D_{\Delta_1}(\pi))$ is a strongly LdRi-commutative triple;
\item $(\Delta_1, \Delta_2, \pi)$ is a combinatorially RdLi-commutative triple;
\item $(\Delta_1, \Delta_2, \pi)$ is a dual combinatorially RdLi-commutative triple.
\end{enumerate}
\end{theorem}

\begin{proof}

(1) $\Leftrightarrow$ (2) is Corollary \ref{cor dual strong commutative triple}. (1) $\Rightarrow$ (3) is proved in Lemma \ref{lem eta unchange}.

Let $\tau =I_{\Delta_2}(\pi)$. Write $\mathfrak p=\mathfrak{mx}(\pi, \Delta_1)$. Let $l=l_{abs}(\mathfrak p)$. We now prove (3) $\Rightarrow$ (1). We consider the following commutative diagram:
\[ \xymatrix{ D_{\mathfrak p}(\tau)\boxtimes \mathrm{St}(\mathfrak p) \ar@{^{(}->}[r]^{\iota} &  \tau_{N_l} \ar@{^{(}->}[r]^{\iota'}  & (\mathrm{St}(\Delta_2)\times\pi)_{N_l}  \ar@{^{(}->}[r] \ar[d]^s & (\mathrm{St}(\Delta_2)\times D_{\mathfrak p}(\pi)\times \mathrm{St}(\mathfrak p))_{N_l}\ar[d] \\
   &     &  \mathrm{St}(\Delta_2)\dot{\times}^1\pi_{N_l} \ar@{^{(}->}[r]   & \mathrm{St}(\Delta_2)\dot{\times}^1 (D_{\mathfrak p}(\pi)\times \mathrm{St}(\mathfrak p))_{N_l} \ar[d] \\
				& & &   (\mathrm{St}(\Delta_2) \times D_{\mathfrak p}(\pi))\boxtimes \mathrm{St}(\mathfrak p) 
                   }
\]
Here the vertical maps are from projecting to the top layers in the geometric lemma and the horizontal maps except the leftmost one come from the unique submodules of integrals. The leftmost map comes from the unique embedding of such submodule.

Suppose (3) holds. (3) implies that $(\mathrm{St}(\mathfrak p), \tau)$ is a Rd-irreducible pair by Proposition \ref{prop mx rd irr pair}. Now, we consider the composition of the following maps 
\[  \tau  \hookrightarrow \mathrm{St}(\Delta_2)\times \pi  \hookrightarrow \mathrm{St}(\Delta_2)\times D_{\mathfrak p}(\pi)\times \mathrm{St}(\mathfrak p) ,
\]
and by the Rd-irreducibility of $(\mathrm{St}(\mathfrak p), \tau)$ and Proposition \ref{prop trivial orbit in irreducible case}, the composition of the toppest horizontal maps and the rightmost vertical map, which gives a map from $D_{\mathfrak p}(\pi)\boxtimes \mathrm{St}(\mathfrak p)$ to $(\mathrm{St}(\Delta_2) \times D_{\mathfrak p}(\pi))\boxtimes \mathrm{St}(\mathfrak p)$, is non-zero. Thus, tracing the commutative diagram, the map $s\circ \iota' \circ \iota$ is non-zero. This implies the pre-commutativity of $(\mathrm{St}(\mathfrak p), \mathrm{St}(\Delta_2), \pi)$. Hence, we also have the pre-commutativity of $(\mathrm{St}(\Delta_1), \mathrm{St}(\Delta_2), \pi)$ by Proposition \ref{prop transtivity pre commute}. This implies the strong commutativity by Theorem \ref{thm pre imply strong} and so proves (1).

A proof for (2) $\Leftrightarrow$ (4) is similar to (1) $\Leftrightarrow$ (3).
\end{proof}

\section{Applications on constructing strongly commutative triples} \label{s application on strong commut triples}

\subsection{A consequence on producing more strongly commutative triples}

\begin{corollary} (c.f. Corollary \ref{cor saturated commut triple}) \label{cor satruated commute triple revise}
Let $(\Delta_1, \Delta_2, \pi)$ be a strongly RdLi-commutative triple. Let $\Delta'$ be a $\Delta_1$-saturated segment such that $D_{\Delta'}(\pi)\neq 0$. 
\begin{enumerate}
\item $(\Delta', \Delta_2, \pi)$ is also a strongly RdLi-commutative triple;
\item Let $\Delta''$ be another $\Delta_1$-saturated segment such that $D_{\Delta''}\circ D_{\Delta'}(\pi)\neq 0$. Then $(\Delta'', \Delta_2, D_{\Delta'}(\pi))$ is a strongly RdLi-commutative triple.
\end{enumerate} 
\end{corollary}

\begin{proof}
(1) is shown in Corollary \ref{cor saturated commut triple} but we use our new results here. By definitions and Theorem \ref{thm combinatorial def}, $\eta_{\Delta_1}(I_{\Delta_2}(\pi))=\eta_{\Delta_1}(\pi)$ automatically implies that $\eta_{\Delta'}(I_{\Delta_2}(\pi))=\eta_{\Delta'}(\pi)$. By Theorem \ref{thm combinatorial def}, we have (1).

We now consider (2). By (1) and Proposition \ref{prop strong commute imply commute}, we have $I_{\Delta_2}\circ D_{\Delta'}(\pi) \cong D_{\Delta'}\circ I_{\Delta_2}(\pi)$. Then $\eta_{\Delta_1}(I_{\Delta_2}\circ D_{\Delta'}(\pi))$ is obtained from $\eta_{\Delta_1}(I_{\Delta_2}(\pi))$ by decreasing $\varepsilon_{\Delta'}(I_{\Delta_2}(\pi))$ by $1$; and similarly $\eta_{\Delta_1}(D_{\Delta'}(\pi))$ is obtained from $\eta_{\Delta_1}(\pi)$ by decreasing $\varepsilon_{\Delta'}(\pi)$ by $1$. Now, the combinatorial commutation for $(\Delta_1, \Delta_2, \pi)$ implies $\eta_{\Delta_1}(I_{\Delta_2}\circ D_{\Delta'}(\pi))=\eta_{\Delta_1}(D_{\Delta'}(\pi))$. Thus $\eta_{\Delta''}(I_{\Delta_2}\circ D_{\Delta'}(\pi))=\eta_{\Delta''}(D_{\Delta'}(\pi))$. This implies (2) again by Theorem \ref{thm combinatorial def}.
\end{proof}

\subsection{A consequence on level preserving integrals}

For $\pi \in \mathrm{Irr}$, we denote by $\mathrm{lev}(\pi)$, the level of $\pi$ in the sense of Bernstein-Zelevinsky \cite{BZ77} i.e. the largest integer $i$ such that the $i$-th Bernstein-Zelevinsky derivative of $\pi$ is non-zero. Since this is the only part we need those notions, we refer the reader to \cite{Ch22+} for definitions. 

It is more convenient to use the machinery of the highest derivative multisegment in \cite{Ch22+}. We briefly recall the definition. For $\pi \in \mathrm{Irr}$ and a cuspidal representation $\rho$ and an integer $c$, let $\mathfrak{mx}^{\rho}(\pi)$ be the segment with largest $l_{abs}(\mathfrak{mx}^{\rho}(\pi))$ satisfying the following two properties: 
\begin{enumerate}
\item for any segment $\widetilde{\Delta}$ in $a(\widetilde{\Delta})=\rho$;
\item $D_{\mathfrak{mx}^{\rho}(\pi)}(\pi)\neq 0$.
\end{enumerate}
The {\it highest derivative multisegment} of $\pi$ is defined as:
\[  \mathfrak{hd}(\pi) := \sum_{\rho} \mathfrak{mx}^{\rho}(\pi), \]
where $\rho$ runs for all isomorphism classes of cuspidal representations.

Note that the multiplicity of the segment $[\rho, \nu^c\rho]$ in $\mathfrak{mx}^{\rho}(\pi)$ is equal to 
\[   \varepsilon_{[\rho, \nu^c\rho]}(\pi)- \varepsilon_{[\rho, \nu^{c+1}\rho]}(\pi) . \]
Hence, $\varepsilon_{\Delta}(\pi)$ determines $\mathfrak{mx}^{\rho}(\pi)$. Indeed, by a simple induction, one can recover all $\varepsilon_{\Delta}(\pi)$ from $\mathfrak{mx}^{\rho}(\pi)$.

\begin{corollary} \label{cor level preserving integrals comm}
Let $\Delta$ be a segment and let $\pi \in \mathrm{Irr}$. Suppose $\mathrm{lev}(I_{\Delta}(\pi))=\mathrm{lev}(\pi)$. Then, for any segment $\Delta'$ with $D_{\Delta'}(\pi)\neq 0$, $(\Delta', \Delta, \pi)$ is a strongly RdLi-commutative triple. In particular, $D_{\Delta'}\circ I_{\Delta}(\pi)\cong I_{\Delta}\circ D_{\Delta'}(\pi)$.
\end{corollary}

\begin{proof}
By Lemma \ref{lem right multi submulti}, $\eta_{\widetilde{\Delta}}(\pi)\leq \eta_{\widetilde{\Delta}}(I_{\Delta}(\pi))$ for any segment $\widetilde{\Delta}$. Let $\mathfrak h_1$ and $\mathfrak h_2$ be the highest derivative multisegments for $I_{\Delta}(\pi)$ and $\pi$ respectively. The equality on the level implies that $l_{abs}(\mathfrak h_1)=l_{abs}(\mathfrak h_2)$.

Since $\eta_{\widetilde{\Delta}}$ determines all $\varepsilon_{\Delta}$ as discussed above, we have that all the inequalities are equalities. Indeed, let $p_1(\rho)$ and $p_2(\rho)$ be the number of segments in $\mathfrak h_1$ and $\mathfrak h_2$ containing $\rho$. Then the above inequality forces that $p_1(\rho)\geq p_2(\rho)$ for all $\rho$. But since the sum of $p_i(\rho)$ is equal to $l_{abs}(\mathfrak h_i)$ ($i=1,2$), we must have that all inequalities and equalities. This then inductively recovers that all the inequalities for $\eta$ are equalities.

Now, in particular, $\eta_{\Delta'}(\pi)=\eta_{\Delta'}(I_{\Delta}(\pi))$. The strong commutation then follows from Theorem \ref{thm combinatorial def}. The commutation then follows from Proposition \ref{prop strong commute imply commute}.
\end{proof}

\subsection{Producing commutative triples for sequences}

\begin{corollary} \label{cor integral sequence comm}
Let $\Delta, \Delta_1', \Delta_2'$ be segments and let $\pi \in \mathrm{Irr}$. Suppose $(\Delta, \Delta_1', \pi)$ and $(\Delta, \Delta_2', I_{\Delta_2'}(\pi))$ are strongly RdLi-commutative triples. If $I_{\Delta_2'}\circ I_{\Delta_1'}(\pi)\cong I_{\Delta_1'\cup \Delta_2'}\circ I_{\Delta_1'\cap \Delta_2'}(\pi)$, then the following conditions hold:
\begin{enumerate}
\item $(\Delta, \Delta_1'\cup \Delta_2',\pi)$ and $(\Delta, \Delta_1'\cap \Delta_2', I_{\Delta_1'\cup \Delta_2'}(\pi))$ are strongly RdLi-commutative triples;
\item $(\Delta, \Delta_1'\cap \Delta_2', \pi)$ and $(\Delta, \Delta_1'\cup \Delta_2', I_{\Delta_1'\cap \Delta_2'}(\pi))$ are strongly RdLi-commutative triples;
\item $I_{\Delta_2'}\circ I_{\Delta_1'}\circ D_{\Delta}(\pi)\cong I_{\Delta_1'\cup \Delta_2'}\circ I_{\Delta_1'\cap \Delta_2'}\circ D_{\Delta}(\pi)$.
\end{enumerate}
\end{corollary}
\begin{proof}
By using Theorem \ref{thm combinatorial def} twice, we have:
\[  \eta_{\Delta}(\pi)=\eta_{\Delta}(I_{\Delta_1'}(\pi))=\eta_{\Delta}(I_{\Delta_2'}\circ I_{\Delta_1'}(\pi)) .
 \]
On the other hand, by Lemma \ref{lem right multi submulti},
\[  \eta_{\Delta}(\pi) \leq \eta_{\Delta}(I_{\Delta_1'\cup \Delta_2'}(\pi)) \leq \eta_{\Delta}(I_{\Delta_1'\cap \Delta_2'}\circ I_{\Delta_1'\cup \Delta_2'}(\pi)) .
\]
Now, using $I_{\Delta_2'}\circ I_{\Delta_1'}(\pi)\cong I_{\Delta_1'\cap \Delta_2'}\circ I_{\Delta_1'\cup \Delta_2'}(\pi)$, $\eta_{\Delta}(I_{\Delta_2'}\circ I_{\Delta_1'}(\pi)) =\eta_{\Delta}(I_{\Delta_1'\cap \Delta_2'}\circ I_{\Delta_1'\cup \Delta_2'}(\pi))$. Thus, the above inequalities are equalities. Thus, by Theorem \ref{thm combinatorial def}, $(\Delta, \Delta_1', \pi)$ and $(\Delta, \Delta_2', I_{\Delta_1'}(\pi))$ are strongly RdLi-commutative triples. This proves (1).

Since $I_{\Delta_1'\cup \Delta_2'}\circ I_{\Delta_1'\cap \Delta_2'} =I_{\Delta_1'\cap \Delta_2'}\circ I_{\Delta_1'\cup \Delta_2'}$, one can similarly prove (2). Now (3) follows from Proposition \ref{prop strong commute imply commute}.
\end{proof}

\subsection{A consequence on duals}

For a segment $\Delta=[a,b]_{\rho}$, let $\Delta^{\vee}=[-b,-a]_{\rho^{\vee} }$. 

\begin{corollary} \label{cor dual dual triple}
Let $\pi \in \mathrm{Irr}$. Let $\Delta, \Delta'$ be segments. Then $(\Delta, \Delta', \pi)$ is a strongly RdLi-commutative triple if and only if $(\Delta^{\vee}, \Delta'{}^{\vee}, \pi^{\vee})$ is a strongly LdRi-commutative triple. 
\end{corollary}

\begin{proof}
We only prove the only if direction and a proof for the if direction is similar. By Corollary \ref{cor dual on derivatives}, $\eta_{\Delta}(\pi)=\eta^L_{\Delta^{\vee}}(\pi^{\vee})$ and $\eta_{\Delta}(I_{\Delta'}(\pi))=\eta_{\Delta^{\vee}}^L(I_{\Delta'}(\pi)^{\vee})=\eta_{\Delta^{\vee}}^L(I^R_{\Delta'{}^{\vee}}(\pi^{\vee}))$. By Theorem \ref{thm combinatorial def}, $\eta_{\Delta}(\pi)=\eta_{\Delta}(I_{\Delta'}(\pi))$. Combining equations, we have $\eta^L_{\Delta^{\vee}}(\pi^{\vee})=\eta^L_{\Delta}(I^R_{\Delta'{}^{\vee}}(\pi^{\vee}))$. Thus, by Theorem \ref{thm combinatorial def}, $(\Delta^{\vee}, \Delta'{}^{\vee}, \pi^{\vee})$ is a strongly LdRi-commutative triple.
\end{proof}

\section{Commutative triples from branching laws} \label{s branching laws}

\subsection{Generalized Gan-Gross-Prasad relevant pairs}

We shall assume $D=F$ from now on. We simply write $\nu(g)=|\mathrm{det}(g)|$. Gan-Gross-Prasad \cite{GGP20} introduces a notion of relevant pairs for Arthur type representations, governing their branching law for general linear groups \cite{Ch22}. Such notion can be reformulated in terms of the Bernstein-Zelevinsky derivatives, see \cite{Ch22}. This hints some connections with derivatives (also see \cite{Ch22+}). 

We first extend the notion of strongly commutative triples from segments to multisegments. For two segments $\Delta$ and $\Delta'$, we write $\Delta>\Delta'$ if $\Delta$ and $\Delta'$ are linked and $a(\Delta) \cong \nu^c\cdot a(\Delta')$ for some $c>0$.

\begin{definition}
Let $\mathfrak m$ be a multisegment and let $\pi \in \mathrm{Irr}$. We impose an ordering on $\mathfrak m=\left\{ \Delta_1, \ldots, \Delta_r \right\}$ such that for any $i<j$, $\Delta_i \not> \Delta_j$. Define 
\[  D_{\mathfrak m}^R(\pi) := D^R_{\Delta_r}\circ \ldots \circ D^R_{\Delta_1}(\pi) ,
\]
\[  I_{\mathfrak m}^L(\pi) := I^L_{\Delta_r}\circ \ldots \circ I^L_{\Delta_1}(\pi) .
\]
It is shown in \cite{Ch22+} that the $D_{\mathfrak m}^R(\pi)$ is independent of a choice of such ordering. Using $D^L_{\Delta}\circ I^L_{\Delta}$ is the identity map, we also have that $I^L_{\mathfrak m}(\pi)$ is independent of a choice of an ordering.
\end{definition}

To switch left and right versions suitably, we also need to change the ordering:

\begin{definition} \label{def ordering derivative}
Let $\mathfrak m, \mathfrak n$ be multisegments. We again impose an ordering on $\mathfrak m=\left\{ \Delta_1, \ldots, \Delta_r\right\}$ (resp. $\mathfrak n=\left\{ \Delta_1', \ldots, \Delta_s'\right\}$) given by (c.f. \cite[Theorem 6.1]{Ze80}): for any $i<j$, 
\[  \Delta_i \not> \Delta_j \quad (\mbox{resp.} \Delta_i' \not> \Delta_j') . 
\]
Let $\mathfrak m_i=\left\{ \Delta_1, \ldots, \Delta_i \right\}$ and let $\mathfrak n_j=\left\{ \Delta_1', \ldots, \Delta_j' \right\}$. Let $\mathfrak m, \mathfrak n \in \mathrm{Mult}$. Let $\pi \in \mathrm{Irr}$. We say that $(\mathfrak m, \mathfrak n, \pi)$ is a {\it strongly RdLi-commutative triple} if for any $1\leq i\leq r$ and $1\leq j \leq s$, $(\mathrm{St}(\Delta_i), \mathrm{St}(\Delta_j), I_{\mathfrak n_{j-1}}\circ D_{\mathfrak m_{i-1}}(\pi))$ is strongly RdLi-commutative in the sense of Definition \ref{def strong comm}. 
\end{definition}

We remark that when $\mathfrak n=\emptyset$, the condition of strong commutation is automatic. We can now define a notion of generalized relevant pairs.

\begin{definition} \label{def ggp relevant}
For $\pi \in \mathrm{Irr}$ and $\pi' \in \mathrm{Irr}$, we say that $(\pi, \pi')$ is {\it (generalized) relevant} if there exist multisegments $\mathfrak m, \mathfrak n$ such that $(\mathfrak m, \mathfrak n, \nu^{1/2} \pi)$ is a strongly RdLi-commutative triple and 
\[  D^R_{\mathfrak m}(\nu^{1/2}\cdot \pi) \cong D^L_{\mathfrak n}(\pi') .\]
\end{definition}

Such notion is expected to govern the quotient branching law:

\begin{conjecture}
Let $\pi \in \mathrm{Irr}(G_{n+1})$ and let $\pi' \in \mathrm{Irr}(G_n)$. Then $\mathrm{Hom}_{G_n}(\pi, \pi')\neq 0$ if and only if $(\pi, \pi')$ is a relevant pair.
\end{conjecture}

We will prove the conjecture in \cite{Ch22+c}.

\subsection{Duality}

For a multisegment $\mathfrak m=\left\{\Delta_1, \ldots, \Delta_r\right\}$, define $\mathfrak m^{\vee}=\left\{ \Delta_1^{\vee}, \ldots, \Delta_r^{\vee}\right\}$.

\begin{corollary} \label{cor dual on relevant pair}
Let $\pi, \pi' \in \mathrm{Irr}$. Then $(\pi, \pi')$ is relevant if and only if $(\pi'{}^{\vee}, \pi^{\vee})$ is relevant. 
\end{corollary}

\begin{proof}
Then there exist multisegments $\mathfrak m$ and $\mathfrak n$ such that
\[   D_{\mathfrak m}^R(\nu^{1/2}\pi) \cong D^L_{\mathfrak n}(\pi') 
\]
and $(\mathfrak m, \mathfrak n, \nu^{1/2}\pi)$ is a strongly RdLi-commutative triple. Then, by Corollary \ref{cor dual on derivatives}, 
\[  D_{\mathfrak m^{\vee}}^L(\nu^{-1/2}\pi^{\vee}) \cong D^R_{\mathfrak n^{\vee}}(\pi'{}^{\vee}) .
\]
Hence, $D_{\nu^{1/2}\mathfrak m^{\vee}}^L(\pi^{\vee}) \cong D^R_{\nu^{1/2}\mathfrak n^{\vee}}(\nu^{1/2}\pi'{}^{\vee})$.

It remains to check that $(\nu^{1/2}\mathfrak n^{\vee}, \nu^{1/2}\mathfrak m^{\vee}, \nu^{1/2}\cdot \pi'{}^{\vee})$ is a strongly RdLi-commutative triple. We write $\mathfrak m=\left\{ \Delta_1, \ldots, \Delta_r\right\} $ and $\mathfrak n=\left\{ \Delta_1', \ldots, \Delta_s'\right\}$ with the ordering in Definition \ref{def ordering derivative}. Let
\[  \mathfrak m_j=\left\{ \Delta_1, \ldots, \Delta_j \right\}, \quad \mathfrak n_j=\left\{ \Delta_1', \ldots, \Delta_j' \right\} 
\]
and 
\[  \bar{\mathfrak m}_j=\left\{ \Delta_{j+1}, \ldots, \Delta_r \right\}, \quad \bar{\mathfrak n}_j=\left\{ \Delta_{j+1}', \ldots, \Delta_s' \right\} .\]
By definition, we have that $(\Delta_i, \Delta_j', I_{\mathfrak n_{j-1}}\circ D_{\mathfrak m_{i-1}}(\pi))$ is a strongly RdLi-commutative triple. Thus, by the duality in Corollary \ref{cor dual dual triple} with Lemma \ref{lem dual on integrals} and Corollary \ref{cor dual on derivatives}, $(\Delta_i^{\vee}, \Delta_j'{}^{\vee}, I^R_{\mathfrak n_{j-1}^{\vee}}\circ D^L_{\mathfrak m_{i-1}^{\vee}}(\nu^{-1/2}\pi^{\vee}))$ is a strongly LdRi-commutative triple. By Corollary \ref{cor dual strong commutative triple}, $(\Delta_j'{}^{\vee},  \Delta_i^{\vee}, I^R_{\mathfrak n_{j}^{\vee}}\circ D^L_{\mathfrak m_{i}^{\vee}}(\nu^{-1/2}\pi^{\vee}))$ is a strongly RdLi-commutative triple.

By multiple uses of Definition \ref{def ordering derivative} and Proposition \ref{prop strong commute imply commute}, we have:
\[     I^L_{\bar{\mathfrak n}_{j}}\circ D^R_{\bar{\mathfrak m}_{i}}\circ I^L_{\mathfrak n_{j}} \circ D^R_{\mathfrak m_{i}}(\nu^{1/2}\pi) \cong I^L_{\mathfrak n}\circ D^R_{\mathfrak m}(\pi) \cong \pi'
\]
and so 
\[    I^L_{\mathfrak n_{j}}\circ D^R_{\mathfrak m_i}(\nu^{1/2}\pi) \cong I^R_{\bar{\mathfrak m}_{i}} \circ D^L_{\bar{\mathfrak n}_{j}}(\pi') .
\]
Now, taking the dual and using  Lemma \ref{lem dual on integrals} again, we have:
\[   I^R_{\mathfrak n_{j}^{\vee}}\circ D^L_{\mathfrak m_{i}^{\vee}}(\nu^{-1/2}\pi^{\vee}) \cong I^L_{\bar{\mathfrak m}_{i}^{\vee}}\circ D^R_{\bar{\mathfrak n}_{j}^{\vee}}(\pi'^{\vee}) .
\]

Thus, combining above, we have that $(\Delta_j'{}^{\vee}, \Delta_i^{\vee},  I^L_{\bar{\mathfrak m}_{i}^{\vee}}\circ D^R_{\bar{\mathfrak n}_{j}^{\vee}}(\pi'^{\vee}))$ is a strongly RdLi-commutative triple. Thus $(\mathfrak n^{\vee}, \mathfrak m^{\vee}, \pi')$ is a strongly RdLi-commutative triple. Now, we obtain that $(\nu^{1/2}\mathfrak n^{\vee}, \nu^{1/2}\mathfrak m^{\vee}, \nu^{1/2}\cdot \pi'{}^{\vee})$ is a strongly RdLi-commutative triple by imposing a shift of $\nu^{1/2}$. 
\end{proof}

\subsection{Examples}

\begin{example}
Let $\pi$ and $\pi'$ be trivial representations of $G_{n+1}$ and $G_n$ respectively. It is clear that $\mathrm{Hom}_{G_n}(\pi, \pi')\neq 0$. In such case,
\[  \pi=\mathrm{St}(\left\{[-\frac{n}{2}], \ldots, [\frac{n}{2}]\right\}), \quad \pi'=\mathrm{St}(\left\{[-\frac{(n-1)}{2}],\ldots, [\frac{(n-1)}{2}]\right\}) .
\]
We have that $(\nu^{1/2}\pi)_{N_1}=\pi' \boxtimes \nu^{(n+1)/2}$. Hence, $D_{[(n+1)/2]}(\nu^{1/2}\pi) \cong \pi'$ and so $(\left\{[1/2]\right\}, \emptyset, \pi)$ defines a strongly RdLi-commutative triple. 
\end{example}

\begin{example}
Let $\Delta_1, \Delta_2$ be segments such that $\mathrm{St}(\Delta_1)$ and $\mathrm{St}(\Delta_2)$ are representations of $G_{n+1}$ and $G_n$ respectively. It is a well-known result from the Rankin-Selberg theory of Jacquet--Pietaskii-Shapiro--Shalika that $\mathrm{Hom}_{G_n}(\mathrm{St}(\Delta), \mathrm{St}(\Delta'))\neq 0$. Write $\Delta=[a,b]_{\rho}$ and $\Delta'=[a',b']_{\rho'}$. We now produce strong commutations according to the following cases:
\begin{enumerate}
\item Case 1: $\rho \not\cong \nu^c\rho'$ for some integer $c$. In such case, by Example \ref{ex examples of st comm triples}(1) and Theorem \ref{thm pre imply strong}, $(\nu^{1/2}\Delta, \Delta', \pi)$ is a strongly RdLi-commutative triple and we have that $I_{\Delta'}\circ D_{\nu^{1/2}\Delta}(\nu^{1/2}\mathrm{St}(\Delta))\cong \mathrm{St}(\Delta')$.
\item Case 2: $\rho \cong \nu^c\rho'$ for some integer $c$. By rewriting the segments if necessary, we shall assume $\rho=\rho'$ from now on. 
\begin{itemize}
\item Case 2(a): $a'<a$. By Example \ref{ex examples of st comm triples}(2) and Theorem \ref{thm pre imply strong}, $(\nu^{1/2}\Delta, \Delta', \nu^{1/2}\pi)$ defines the desired strongly RdLi-commutative triple.
\item Case 2(b): $b'<b$. By Example \ref{ex examples of st comm triples}(3) and Theorem \ref{thm pre imply strong}, $(\nu^{1/2}\Delta, \Delta', \nu^{1/2}\pi)$ defines the desired strongly RdLi-commutative triple.
\item Case 2(c): $(\nu^{1/2}\Delta) \cap \Delta' =\emptyset$. By Example \ref{ex examples of st comm triples}(1) and Theorem \ref{thm pre imply strong}, $(\nu^{1/2}\Delta, \Delta', \nu^{1/2}\pi)$ defines the desired strongly RdLi-commutative triple. 
\item Case 2(d): $a\leq a' \leq b \leq b'$. Let $\widetilde{\Delta}= (\nu^{1/2}\Delta) \cap \Delta'$. Then $((\nu^{1/2}\Delta)\setminus \widetilde{\Delta}, \Delta'\setminus \widetilde{\Delta}, \nu^{1/2}\pi)$ defines a strongly RdLi-commutative triple. 
\end{itemize}
\end{enumerate}
\end{example}

\begin{example}
Let $\pi=\mathrm{St}(\left\{ [-n/2,c-1], [c,n/2] \right\})$ and let $\pi'=\mathrm{St}([-(n-1)/2, (n-1)/2])$ for some $-n/2\leq c \leq n/2$. Let 
\[  \Delta=[c+1/2, (n+1)/2], \quad \Delta'=[c+1/2, (n-1)/2] .
\]
Note that Example \ref{ex examples of st comm triples}(3) and Theorem \ref{thm pre imply strong}, $(\Delta, \Delta', \pi)$ is a strongly RdLi-commutative triple. This agrees with the expectation from \cite{Qa23+}. 

\end{example}

\section{Appendix: Compositions of supporting orbits}

We use the notations in Section \ref{ss supporting orbit}. In particular, we have $P, Q$ to be standard parabolic subgroups in $G$. Let $R \subset Q$ be a standard parabolic subgroup of $G$ and let $R'=R\cap M_Q$.

 We enumerate the elements in $W_{P,Q}(G)$ as $w_1, \ldots, w_r$ such that $i<j$ implies $w_i \not\geq w_j$. For each $i$, we also enumerate the elements in $W_{P^{w_i},R'}(M_Q)$ as $u_{i,1}, \ldots, u_{i,j_i}$. Then we obtain an enumeration on the elements in $W_{P,R}(G)$ as:
\[   w_1u_{1,1},\ldots, w_1u_{1,j_1}, w_2u_{2,1},\ldots, w_2u_{2,j_2}, \ldots, w_ru_{r,1}, \ldots, w_ru_{r,j_r} .
\]
If we relabel as $x_1, x_2, \ldots$, then the enumeration satisfies that $k<l$ implies that $x_k \not\geq x_l$.

\begin{proposition} \label{prop composition supp orbit}
We use the notations above. Let $\lambda$ be a submodule in $(\mathrm{Ind}_P^G\pi)_{N_Q}$ with the supporting orbit $PwQ$. Let $\lambda'$ be a simple submodule of $\lambda_{N_{R'}}$. Suppose the embedding 
\[\lambda' \hookrightarrow \lambda_{N_{R'}} \hookrightarrow \mathrm{Ind}_{P^{w}\cap M_Q}^{M_Q} (\pi_{M_P\cap N_Q^{w}})^{w} \]
has the supporting orbit $P^{w}w'R'$. Then the embedding
\[  \lambda' \hookrightarrow (\mathrm{Ind}_P^G\pi)_{N_R}
\]
has the supporting orbit $Pww'R$. 
\end{proposition}

\begin{proof}
This follows from the fact that if $w_iu_{i,j} \not\geq w_{i'}u_{i',j'}$, then either $w_i\not\geq w_{i'}$ or $u_{i,j}\not\geq u_{i',j'}$ (see \cite{De77}).
\end{proof}







\end{document}